\documentclass[12pt,leqno]{amsart}
\usepackage{enumerate}
\usepackage{amsmath}
\usepackage{amssymb}
\usepackage{amsthm}
\usepackage{amsfonts}
\usepackage{esint} % to use the average integral
\usepackage{mathrsfs}
\usepackage{mathtools}
\usepackage{tikz}
\usetikzlibrary{shadings,calc}
\usepackage{tikz-cd}
\usepackage{appendix}
\usepackage{relsize}
\usepackage{mathtools}
\usepackage{color}
\usepackage[
  hmarginratio={1:1},     % equal left and right margins
  vmarginratio={1:1},     % equal top and bottom margins
  outer=4cm,
  inner=4cm,
  marginparwidth=2.5cm,	  % making sure margin notes are displayed correctly
  marginparsep=1mm,
  textwidth=400pt,        % new text width
  heightrounded,          % always useful
  asymmetric              % in duplex layout, margin note always display on right
]{geometry}

\usepackage[colorlinks,citecolor=blue,pagebackref,hypertexnames=false]{hyperref}
\usepackage{graphicx}
\DeclareMathOperator*{\fiint}{\ensuremath{\iint\text{\kern-1.36em{\raisebox{5.87pt}{\rotatebox{-93}{$\setminus$}}}}}}

\allowdisplaybreaks

\newcommand{\OFN}{\Omega_{\mathcal{F}_N,Q}}
\newcommand{\WN}{\mathcal{W}_N}
\newcommand{\pWN}{\mathcal{W}_N^{\Sigma}}

\DeclareMathOperator{\interior}{int}
\DeclareMathOperator{\osc}{osc}

\DeclareMathOperator{\supp}{supp}

\numberwithin{equation}{section}

\newtheorem{theorem}[equation]{Theorem}
\newtheorem*{theorem*}{Theorem}
\newtheorem{lemma}[equation]{Lemma}

\newtheorem{prop}[equation]{Proposition}

\newtheorem{defn}[equation]{Definition}

\theoremstyle{remark}
\newtheorem{remark}[equation]{Remark}

\theoremstyle{definition}

\newcommand{\pO}{\Gamma}
\newcommand{\dt}{\widetilde\Delta}

\newcommand{\Carl}[1]{|\nabla{#1}|^2\delta(X)}

\newcommand{\ve}{\varphi_{\epsilon}}

\DeclareMathOperator{\dist}{dist}
\DeclareMathOperator{\divg}{div}
\DeclareMathOperator{\diam}{diam}
\newcommand{\RR}{\mathbb{R}}
\newcommand{\Hd}{\mathcal{H}}
\newcommand{\DD}{\mathbb{D}}
\newcommand{\WW}{\mathcal{W}}
\newcommand{\FF}{\mathcal{F}}

\newcommand{\wI}{\widetilde{I^{**}}}

\begin{document}
\allowdisplaybreaks

\title[$S<N$, BMO solvability and $A_\infty$ class]{Square function estimates, BMO Dirichlet problem, and absolute continuity of harmonic measure on lower-dimensional sets}

\author[S. Mayboroda]{Svitlana Mayboroda}
\address{Svitlana Mayboroda, Department of Mathematics, University of Minnesota, Minneapolis, MN 55455, USA}
\email{svitlana@math.umn.edu}

\author[Z. Zhao]{Zihui Zhao}
\address{Zihui Zhao, University of Washington, Department of Mathematics, Seattle, WA 98195-4350, USA}
\email{zhaozh@uw.edu}

\thanks{The first author was partially supported by the NSF INSPIRE Award DMS-1344235, NSF CAREER Award DMS-1220089, and the Simons Foundation grant 563916, SM. The second author was partially supported by NSF grant numbers DMS-1361823, DMS-1500098 and DMS-1664867.}
\thanks{This material is based upon work supported by the National Science Foundation under Grant No. DMS-1440140 while the authors were in residence at the Mathematical Sciences Research Institute in Berkeley, California, during the Spring 2017 semester.}



\begin{abstract}
In the recent work \cite{elliptic, Ainfty} G. David, J. Feneuil, and the first author have launched a  program devoted to an analogue of harmonic measure for lower-dimensional sets. A relevant class of partial differential equations, analogous to the class of elliptic PDEs in the classical context, is given by linear degenerate equations with the degeneracy suitably depending on the distance to the boundary. 

The present paper continues this line of research and focuses on the criteria of quantitative absolute continuity of the newly defined harmonic measure with respect to the Hausdorff measure, $\omega\in A_\infty(\sigma)$, in terms of solvability of boundary value problems. The authors establish, in particular, 
square function estimates and solvability of the Dirichlet problem in BMO for domains with lower-dimensional boundaries under the underlying assumption $\omega\in A_\infty(\sigma)$. More generally, it is proved that in all domains with Ahlfors regular boundaries the BMO solvability of the Dirichlet problem is necessary and sufficient for the absolute continuity of the harmonic measure.
\end{abstract}

\maketitle

\section{Introduction}

The last decade has seen great advances in understanding of the connections between
analytic, geometric, and PDE properties of sets. One of the central questions in this quest pertains to the necessary and sufficient conditions on the geometry of the domain which guarantee absolute continuity of the harmonic measure $\omega$  with respect to the surface measure $\sigma$ of the boundary. The interest to this problem begins with the classical 1916  F. and M. Riesz theorem \cite{RR} which asserts that for a simply connected planar domain with a rectifiable boundary, the harmonic measure is absolutely continuous with respect to the boundary surface measure (see \cite{La} for a quantitative version). A local analogue of this result was established in \cite{BJ}, which also showed that absolute continuity may fail in the absence of some topological hypothesis, even for a rectifiable domain. The emerging philosophy is that  the key geometric properties at play are smoothness (or to be precise, rectifiability) and connectedness of the domain.  In higher dimensions, the latter is much trickier, and without any pertinent details we mention that the absolute continuity of the harmonic measure with respect to the boundary surface measure has been proved in Lipschitz graph  domains \cite{D}, and later in the so-called chord-arc domains in \cite{DJ, S}, and more recent achievements in the field have progressively further weakened the underlying geometric hypotheses \cite{BL, Ba, HM, AHMNT, Mo, ABaHM, ABoHM, Az, HM2}, although the sharp assumptions, particularly in terms of connectedness, are not completely clear yet. %The above results do not hold in general for elliptic operators, unless the elliptic matrix has some decay towards the boundary, see \cite{FJK, KP}.
Meanwhile in the converse direction, the necessary conditions for the absolute continuity of harmonic measure with respect to the Hausdorff measure of the boundary have been obtained in 1-sided chord-arc domains in \cite{HMU} (see also \cite{AHMNT}), and later in more general domains in \cite{MT, HLMN}. As a culmination of this line of work, it was shown without any topological background assumptions that rectifiability is necessary for absolute continuity of the harmonic measure in \cite{AHMMMTV}. These results were extended to general elliptic operators and other manifestations of solvability of the Dirichlet boundary value problem in \cite{HMT, TZ, AM, HMMTZ, HMM, GMT, AGMT} to mention only a few: the area is blossoming and we do not aim at a complete listing of the related literature.

All of these advances heavily rely on the properties of harmonic functions, and as such, do not apply to domains with lower-dimensional boundaries, for instance, a complement of a curve in $\RR^3$. In fact, sets of higher co-dimension are not visible by classical Brownian travelers (that is, the probability to hit such a set is zero) and equivalently, by classical harmonic functions. Led by these considerations, 
G. David, J. Feneuil, and the first author have recently launched a  program devoted to a new type of degenerate elliptic PDEs \cite{elliptic}, such that the corresponding elliptic measure (still referred to as harmonic measure in the course of this discussion) is not only non-trivial, but absolutely continuous with respect to the Hausdorff  measure in favorable geometric circumstances. The goal of the present paper is to establish equivalence of absolute continuity of harmonic measure to the BMO solvability of the Dirichlet problem on arbitrary Ahlfors-regular domains and in the general class of degenerate elliptic operators, and to prove a technical but very important in many applications roadblock: the square function estimates for solutions. Let us discuss this in more details. 

We shall work in the general context of $d$-Ahlfors-David regular sets, which are roughly speaking, $d$-dimensional uniformly at all scales. 
\begin{defn}\label{def:ADR} Let $\Gamma \subset \RR^n$ be a closed set and $d\leq n$ be an integer. We say $\Gamma$ is $d$-Ahlfors  regular if there exists a constant $ C_0\geq 1$ such that for any $q\in \Gamma$ and $r>0$,
\[ C_0^{-1} r^{d} \leq \mathcal{H}^d(B(q, r)\cap \Gamma) \leq C_0 r^{d}, \]
where $\mathcal{H}^{d}$ is the $d$-dimensional Hausdorff measure. We shall often denote $\mathcal{H}^{d}|_\Gamma$, that is $\mathcal{H}^d$ restricted to the set $\Gamma$, by $\sigma$.
\end{defn}

Let $\Gamma$ be a $d$-Ahlfors regular set in $\RR^n$ with $d<n-1$, and $\Omega = \RR^n \setminus \Gamma$. Consider the degenerate elliptic operator $L=-\divg(A(X)\nabla)$ with a real, symmetric $n\times n$ matrix $A(X)$ satisfying
	\begin{equation}\label{eq:Aub}
		 A(X) \xi\cdot \zeta \leq C_1 |\xi| |\zeta| \delta(X)^{d-n+1} \text{ for } X\in\Omega \text{ and } \xi, \zeta\in\RR^n,
	\end{equation}
	\begin{equation}\label{eq:Alb}
		A(X) \xi \cdot \xi \geq C_1^{-1} |\xi|^2 \delta(X)^{d-n+1} \text{ for } X\in\Omega \text{ and } \xi \in\RR^n
	\end{equation}
	for some $C_1 \geq 1$, where $\delta(X) = \dist(X,\Gamma)$. We say a function $u $ in the Sobolev space $W_r(\Omega)$ (see the definition in \eqref{def:Wr}) is a weak solution to $Lu=0$, if
	\[ \iint_\Omega A(X) \nabla u \cdot \nabla \varphi ~ dX = 0, \quad \text{ for any } \varphi \in C_0^\infty(\Omega). \]
The basic elliptic theory of such equations was developed in \cite{elliptic}. In particular, it was shown that the Dirichlet problem 
\begin{equation}\tag{D}\label{ellp}
	\left\{\begin{array}{ll}
		Lu=0 & \text{in~} \Omega \\
		u=f & \text{on~} \Gamma.
	\end{array}\right.
\end{equation}
has a suitably interpreted weak solution for smooth compactly supported (and more general) $f$ on $\Gamma$, that such a solution is locally bounded and H\"older continuous in the interior and at the boundary, and finally, that it can be written in terms of the corresponding harmonic measure, and the latter satisfies the usual doubling, non-degeneracy, and change-of-pole conditions. We refer the reader to Section \ref{sect:prelim} for details. For now, we only recall that  the harmonic measure is a (family of) positive regular Borel measure(s) $\omega^X$ on $\Gamma$, $X\in \Omega$, such that, in particular, for any boundary function $f\in C_0^0(\pO)$ the solution to \eqref{ellp} can be written as 
\begin{equation}\label{solnhm}u(X) = \int_{\Gamma} f d\omega^X.
\end{equation}
\begin{defn}
	We say the harmonic measure $\omega$ is of class $A_\infty$ with respect to the surface measure $\sigma = \mathcal{H}^d|_{\Gamma}$, or simply $\omega \in A_\infty(\sigma)$, if for any $\epsilon>0$, there exists $\delta=\delta(\epsilon)>0$ such that for any surface ball $\Delta$, any surface ball $\Delta' \subset \Delta$ and any Borel set $E\subset \Delta'$, we have
	\begin{equation}
		\frac{\sigma(E)}{\sigma(\Delta')} <\delta \implies \dfrac{\omega^A(E)}{\omega^A(\Delta')} <\epsilon.
	\end{equation}
	Here $A = A_{\Delta} $ is a corkscrew point for $\Delta$ (see Lemma \ref{lm:cksc} for the definition and existence of corkscrew point).
\end{defn}
We remark that while the aforementioned basic properties of harmonic measure (existence, doubling, non-degeneracy, change-of-poles etc.) hold in full generality of $d$-Ahlfors regular sets, $d<n-1$, the $A_\infty$ property of the harmonic measure is much more delicate and is not expected on very rough domains. %\zihui{This is an example in codim one. Is there a such example in the case $d<n-1$?} 
In particular, already on a planar domain with 1-dimensional boundary rectifiability of the boundary is necessary for $\omega \in A_\infty(\sigma)$. On the other hand, it is not vacuous either, as the authors in \cite{Ainfty} have proved that for any $d<n-1$ and $\Gamma$ a $d$-dimensional Lipschitz graph with a small Lipschitz constant the harmonic measure is absolutely continuous with respect to the Hausdorff measure for the operator $L=-\divg(D(X)^{-n+d+1} \nabla)$ where 
\begin{equation} \label{1.3}
D(X) = \Big\{ \int_\Gamma |X-y|^{-d-\alpha} d \mathcal{H}^d(y) \Big\}^{-1/\alpha}, \quad X\in \Omega, 
\end{equation}
for some constant $\alpha > 0$. 
It is easy to see that 
$D(X)$ is equivalent to the Euclidean distance $\dist(X,\Gamma)$ (and this would even stay true when $\Gamma$ is an Ahlfors regular set) but not equal.

For any $q\in\Gamma$ and $r>0$, we use $\Delta = \Delta(q,r)$ to denote the surface ball $B(q,r)\cap \Gamma$, and use $T(\Delta) := B(q,r)\cap\Omega$ to denote the ``tent'' above $\Delta$.
A function $f$ defined on $\Gamma$ is a BMO function if
	\begin{equation}\label{eq:defBMO}
		\|f\|_{BMO} := \sup_{\Delta \subset \pO} \left(\fint_{\Delta} |f-f_\Delta|^2 d\sigma\right)^{\frac{1}{2}} <\infty.
	\end{equation}
Here $f_\Delta$ denotes the average $\fint_{\Delta} f d\sigma$.
\begin{defn} We say that the Dirichlet problem \eqref{ellp} is solvable in BMO if  for any boundary function $f\in C_0^0(\pO)$, the solution $u$ to \eqref{ellp} given by \eqref{solnhm} satisfies a condition that $|\nabla u|^2 \delta(X)^{d-n+2}\,dX$ is a Carleson measure with norm bounded by a constant multiple of $\|f\|_{BMO}^2$, that is, 
	\begin{equation}\label{MT:Carlest}
		\sup_{\Delta\subset\pO} \frac{1}{\sigma(\Delta)} \iint_{T(\Delta)} |\nabla u|^2 \delta(X)^{d-n+2}\,dX \leq C \|f\|_{BMO}^2.
	\end{equation}
\end{defn}

One of the main results of the present paper is as follows. 
\begin{theorem}\label{thm:main}
	Let $\Gamma$ be a $d$-Ahlfors regular set in $\RR^n$ with $d<n-1$ and $\Omega = \mathbb{R}^n \setminus \Gamma$. Consider the operator $L=-\divg(A(X)\nabla)$ with a real, symmetric $n\times n$ matrix $A(X)$ satisfying \eqref{eq:Aub} and \eqref{eq:Alb}.
	Then the harmonic measure $\omega \in A_{\infty}(\sigma)$ if and only if the Dirichlet problem \eqref{ellp} is BMO-solvable.
\end{theorem}

In co-dimension 1 this has been proved in \cite{DKP} for Lipschitz domains and in \cite{Zh} for uniform domains with Ahlfors regular boundaries. One of the main difficulties 
in our case is to prove an upper bound on the square function by the non-tangential maximal function. The latter, in co-dimension 1, goes back to the work of Dahlberg, Jerison, and Kenig for Lipschitz domains in \cite{DJK}, and their method can be extended to more general sets with the help of preliminary estimates proved in \cite{NTA}. This result, and even more so the method behind it, underpinned many later developments in the subject. To prove it, \cite{DJK} systematically use the harmonic measures of the sawtooth domains to get a good-$\lambda$ inequality. This technique is not available to us. The sawtooth domain is a domain inside $\Omega$ on top of a set $E\subset \partial\Omega = \Gamma$ that satisfies some desired properties, and roughly speaking, allows one to exchange local results with global ones. In some sense, it is the use of the sawtooth domains which allows one to exploit the fact that at every scale the $A_\infty$ condition only carries information on a big portion of a boundary ball, rather than the entire boundary ball - a crucial ingredient in this and many other arguments in the theory. In the case of the lower dimensional $\Gamma$, however, the boundary of a sawtooth domain may have arbitrarily small/large pieces of dimension $d$ and, simultaneously, pieces of dimension $n-1$. For that reason, it is not automatically clear if one can make sense of the harmonic measure for the sawtooth domain or to resolve the Dirichlet problem on the sawtooth domain. Instead we are bound to work with the Green function of the entire $\Omega$, and get a good-$\lambda$ inequality by using various considerations akin to the comparison principle. Needless to say, the geometric arguments for lower-dimensional sets are also very different and the technical side of the present paper ends up surprisingly far from \cite{DJK}, \cite{DKP} and \cite{Zh}. Moreover, since the theory of the lower-dimensional sets is still at its infancy, these technical geometric arguments, e.g. Lemma \ref{lm:stcutoff}, are likely to be useful in many future works.

The formal results in this direction are as follows. For any $q\in\Gamma$ and $\alpha>0$, we define the non-tangential cone $\Gamma^{\alpha}(q)$ with vertex $q$ and aperture $\alpha$ as 
\begin{equation}\label{def:NTC}
	\Gamma^{\alpha}(q) = \{X\in\Omega: |X-q| < (1+\alpha) \delta(X) \},
\end{equation} 
and a truncated cone as
\[ \Gamma_r^{\alpha}(q) = \Gamma^{\alpha}(q)\cap B(q,r). \]
When there is no confusion we drop the super-index $\alpha$ and simply denote them by $\Gamma(q)$ and $\Gamma_r(q)$, respectively. We define the non-tangential square function
\begin{equation}
	Su(q) = \left( \iint_{\Gamma(q)}  |\nabla u|^2 \delta(X)^{1-d} dm(X) \right)^{\frac{1}{2}}
\end{equation}
and the truncated square function
\begin{equation}
	S_r u(q) = \left( \iint_{\Gamma_r(Q)}  |\nabla u|^2 \delta(X)^{1-d} dm(X) \right)^{\frac{1}{2}}.
\end{equation}
We also define the non-tangential maximal function and its truncated analogue
\begin{equation}
	Nu(q) = \sup_{X\in \Gamma(q)} |u(X)|, \quad N_r u(q) = \sup_{X\in \Gamma_r(q)} |u(X)|.
\end{equation}
Given apertures $0< \alpha< \alpha_1 < \beta$, for simplicity we denote $Su, S'u$ as the square function on non-tangential cones of aperture $\alpha, \alpha_1$, respectively, and denote $Nu$ the non-tangential maximal function of aperture $\beta$.
%Suppose $\DD$ is a collection of dyadic cubes for the $d$-Ahlfors regular set $\Gamma$, which consists of nested cubes in $\Gamma$ that behave like surface balls. See Lemma \ref{lm:dcAR} for the detail. 
We have:
\begin{prop}[good-$\lambda$ inequality for $\omega$]\label{lm:goodlambdaomega0}
	Suppose $\Gamma$ is a $d$-Ahlfors regular set in $\RR^n$ with $d<n-1$, $\Omega= \RR^n \setminus \Gamma$ and $\DD$ is a collection of dyadic cubes for $\Gamma$, see Lemma \ref{lm:dcAR} for the detail.
	Let $u\in W_r(\Omega)$ be a non-negative solution of $Lu=0$ such that for some dyadic cube $Q\in \DD$ and $\lambda>0$ there exists $q_1 \in \Gamma$ with
	\[ S'u(q_1) \leq \lambda \quad \text{ and } \quad |q_1 - q| \leq C_2 \diam Q \text{ for all } q\in Q. \]
	Then for any $X_Q \notin B(x_Q, 2C_3 \ell(Q))$ and $\delta$ sufficiently small, we have
	\begin{equation}\label{eq:goodlambda}
		\omega^{X_Q}\left(\left\{ q\in Q: Su(q) > 2\lambda, Nu(q) \leq \delta\lambda \right\} \right) \leq C\delta^2 \omega^{X_Q}\left( Q \right)
	\end{equation}
	Here $x_Q, \ell(Q)$ are the ``center'' and ``size'' of $Q$, see Lemma \ref{lm:dcAR}. The constant $C>0$ depends on the allowable parameters $d, n, C_0, C_1$, the apertures $\alpha, \alpha_1, \beta$, and the given constants $C_2, C_3$.
\end{prop}
\noindent If, moreover, $\omega\in A_\infty(\sigma)$, then the good-$\lambda$ inequality for $\sigma$ follows and we conclude that 
\begin{equation}\label{eq:SlessthanN}
		\|Su\|_{L^p(\sigma)} \leq C\|Nu\|_{L^p(\sigma)}
	\end{equation}
	 for any $1\leq p<\infty$ and any solution $u\in W_r(\Omega)$ to $Lu=0$ such that the right hand side is finite.

\bigskip

The paper is structured as follows. In Section \ref{sect:prelim} we first state some lemmas proved in \cite{elliptic} and prove some preliminary results based off these lemmas. In Section \ref{sect:SlessN} we prove the above Proposition after a careful analysis of the sawtooth domains, and moreover we prove the upper bound of the square function by the non-tangential maximal function. This is an independdent result and will also be used in Section \ref{sect:showBMOs}, where we prove if the harmonic measure $\omega$ is of class $A_\infty(\sigma)$, the Dirichlet problem is BMO-solvable. We prove the converse in Section \ref{sect:showAinfty}, that is, BMO-solvability implies the harmonic measure $\omega$ is of class $A_\infty(\sigma)$.

\bigskip

\textbf{Acknowledgement} We would like to thank Guy David for very helpful discussions in connection with this paper.

%%%%%%%%%%%%%%%%%%%%%%%%%%%%%%%%%%%%%%%%%%%%%%%%%%%%%%%%%%%%%%%%%%%%%%%%%%%%%%%%%%%%%%%%%%%% 
%%%%%%%%%%%%%%%%%%%%%%%%%%%%%%%%%%%%%%%%%%%%%%%%%%%%%%%%%%%%%%%%%%%%%%%%%%%%%%%%%%%%%%%%%%%% 
\section{Preliminaries}\label{sect:prelim}

The ground work for harmonic measures associated to the (degenerate) elliptic operators $L$ on sets of lower dimensions $d<n-1$ has been laid out in the work of David, Feneuil and Mayboroda, see \cite{elliptic}. In this section we state some relevant preliminary results proven in \cite{elliptic}; we also prove a few lemmas that follow easily and are needed in later sections. For the convenience of readers familiar with this subject, we point out that the new lemmas we prove here are Lemmas \ref{lm:delta}, \ref{lm:fchiE} and \ref{lm:Poincare}. Unless specified otherwise, the constants that appear in the following lemmas would depend only on the allowable constants, namely the dimensions $n$, $d$, the Ahlfors regular constant $C_0$ and the ellipticity constant $C_1$.

We start with the following notations:
\begin{itemize}
	\item For any $X\in\Omega$, we denote $\delta(X) =\dist(X,\Gamma)$, the Euclidean distance from $X$ to $\Gamma$, and the weight $w(X) = \delta(X)^{d-n+1}$. 
	\item We denote
		\[ \mathcal{A}(X) := \frac{1}{w(X)} A(X) = \delta(X)^{n-1-d} A(X) . \]
		By \eqref{eq:Aub} and \eqref{eq:Alb}, $\mathcal{A}(X)$ is a uniformly elliptic matrix.
	\item We define a measure $m$ on Borel sets in $\RR^n$ by letting $m(E) = \iint_E w(X) dm(X)$. We may write $dm(X) = w(X) dX$. Since $0<w <\infty$ a.e. in $\RR^n$, $m$ and the Lebesgue measure are mutually absolutely continuous.
	\item For any $q\in\Gamma$ and $r>0$, we use the notation $\Delta(q,r)$, or sometimes simply $\Delta$, to denote the surface ball $ B(q,r) \cap \pO$, and $T(\Delta)$ to denote the ``tent'' $B(q,r)\cap \Omega$ over $\Delta$.
	\item We denote the surface measure $\sigma = \Hd^d|_{\pO}$.
	\item If $B=B(X,r)$ is a ball and $\alpha>0$ a constant, we use $\alpha B = B(X,\alpha r)$ to denote the concentric dilation of $B$. The same notation applies to surface balls $\alpha\Delta$.
\end{itemize}  

\begin{lemma}[Harnack chain condition, Lemma 2.1 of \cite{elliptic}]\label{lm:Hcc}
	Let $\Gamma$ be a $d$-Ahlfors regular set in $\RR^n$ and $d<n-1$. Then there exists a constant $c\in(0,1)$, that depends only on $d, n, C_0$, such that for $\Lambda \geq 1$ and $X_1, X_2 \in\Omega$ such that $\delta(X_i) \geq s$ and $|X_1 -X_2| \leq \Lambda s$, we can find two points $Y_i\in B(X_i, s/2) $ such that $\dist([Y_1, Y_2],\Gamma) \geq c\Lambda^{-d/(n-1-d)} s$. That is, there is a thick tube in $\Omega$ that connects the balls $B(X_i, s/2)$.
\end{lemma}

\begin{remark}\label{rmk:Hcc}
	Note that 
\begin{equation}\label{eq:Hclength}
	|Y_1 - Y_2| \leq |Y_1 - X_1| + |X_1 - X_2| + |X_2 + Y_2| < 2\Lambda s. 
\end{equation} 
Let $\tau = c\Lambda^{-d/(n-1-d)} s$ and $Z_1 = Y_1$. For $2\leq j \leq N$ let $Z_j$ be consecutive points on the line segment $[Y_1, Y_2]$ such that $|Z_j - Z_{j-1}| = \tau/3$. Then
\[ (N-1) \frac{\tau}{3} \leq |Y_1 - Y_2| < N \frac{\tau}{3}. \]
Combined with \eqref{eq:Hclength} we get that the integer
\begin{equation}\label{eq:Hccount}
	N \sim \frac{|Y_1- Y_2|}{\tau/3} \lesssim \Lambda^{\frac{n-1}{n-1-d}}.
\end{equation}
Let $B_0 = B(X_1, s/2), B_j = B(Z_j, \tau/4)$ for $1\leq j\leq N$ and $B_{N+1} = B(X_2, s/2) $. Clearly $B_j \cap B_{j+1} \neq \emptyset$ for all $0\leq j\leq N$. Moreover $\dist(B_0, \Gamma), \dist(B_{N+1}, \Gamma) \geq s/2$ and for $1\leq j\leq N$,
\begin{equation}\label{eq:Hclb}
	\dist(B_j,\Gamma) \geq \frac{3}{4}\tau = \frac{3}{4} c\Lambda^{-\frac{d}{n-1-d}} s,
\end{equation}
and
\begin{equation}\label{eq:Hcub}
	\dist(B_j,\Gamma) \leq \min\{\delta(X_1), \delta(X_2)\} + \frac{s}{2} + |Y_1 - Y_2| < \min\{\delta(X_1), \delta(X_2)\} + 3\Lambda s.
\end{equation}
\end{remark}

\begin{lemma}[estimates on the weight, Lemma 2.3 of \cite{elliptic}]\label{lm:weight}
	\leavevmode
	\begin{enumerate}[(i)]
		\item For any $\theta>0$ there exists $C_{\theta}>0$ such that for any $X\in\RR^n$ and $r>0$ satisfying $\delta(X) \geq (1+\theta) r$,
			\begin{equation}\label{eq:weightin}
				C_{\theta}^{-1} r^n w(X) \leq m\left(B(X,r)\right) = \iint_{B(X,r)} w(z) dz \leq Cr^n w(X).
			\end{equation}
		\item There exists $C>0$ such that for any $q\in\Gamma$ and $r>0$,
	\begin{equation}\label{eq:weight}
		C^{-1} r^{d+1} \leq m\left(B(q,r)\right) = \iint_{B(q,r)\cap\Omega} w(z) dz \leq Cr^{d+1}.
	\end{equation}
	\end{enumerate}
\end{lemma}

From the above we deduce the following estimate, which will be needed later.
\begin{lemma}\label{lm:delta} 
Let $\Gamma$ be $d$-Ahlfors regular. For any $\alpha>-1$, we have
	\begin{equation}\label{eq:deltaintegral}
		\iint_{T(2\Delta)} \delta(X)^{\alpha} dm(X) \lesssim r^{d+1+\alpha}. 
	\end{equation} 
\end{lemma}
\begin{proof}
	The proof is a simple use of Vitali covering.
	For $j= 0,1,\cdots$ let 
	\[ T_j = T(2\Delta) \cap \{x\in\Omega: 2^{-j}r \leq \delta(X) < 2^{-j+1}r\}, \]
	\[ T_{>j} = T(2\Delta) \cap \{x\in\Omega: \delta(X) < 2^{-j+1}r\}. \]
	Then
	\begin{equation}
		\iint_{T(2\Delta)} \delta(X)^{\alpha} dm(X) = \sum_{j=0}^{\infty} \iint_{T_j} \delta(X)^\alpha dm(X) \leq \sum_{j=0}^{\infty} (2^{-j}r)^{\alpha} m(T_{>j}). \label{eq:intdelta}
	\end{equation}

	For every fixed $j$, we consider a covering of $4\Delta$ by $\bigcup\limits_{q\in 4\Delta} B(q, 2^{-j+1}r/5)$, from which one can extract a countable Vitali sub-covering $4\Delta \subset \cup_{k} B(q_k, 2^{-j+1}r)$,	where $q_k\in 4\Delta$ and the balls $B_k = B(q_k, 2^{-j+1}r/5)$ are pairwise disjoint. The fact that $q_k\in 4\Delta = \Delta(q_0, 4r)$ implies 
	\[ B_k := B\left(q_k, \frac{2^{-j+1}r}{5}\right) \subset B\left(q_0,4r + \frac{2^{-j+1}r}{5} \right).  \]
	And the pairwise disjointness of $B_k$'s implies that for every fixed $j$, there are only finitely many of them. In fact,
	\begin{equation}
		\sum\limits_{k} \sigma(B_k) = \sigma\left( \bigcup\limits_{k} B_k \right) \leq \sigma\left( \Delta\left(q_0, 4r+ \frac{2^{-j+1}r}{5}\right)\right) \lesssim \left( 4r + \frac{2r}{5}\right)^{d}. \label{eq:disjoint}
	\end{equation} 
	%\reversemarginpar
	Note that $\sigma(B_k) \approx \left( 2^{-j+1}r/5 \right)^{d}$ independent of $k$.
	Let $N_j$ be the number of $B_k$'s, by \eqref{eq:disjoint}
	\begin{equation}
		N_j\cdot \left( \frac{2^{-j+1}r}{5} \right)^{d} \leq \left( 4r + \frac{2r}{5}\right)^{d} , \qquad \text{thus } N_j\lesssim 2^{jd}. \label{eq:boundN}
	\end{equation}
	
	For any $X\in T_{>j}$, let $q_X\in\pO$ be such that $|X-q_X|=\delta(X)$. Then
	\begin{equation} \label{eq:closetoX}
		|q_X-q_0| \leq |q_X-X| + |X - q_0| < 4r, \qquad \text{i.e. } q_X\in 4\Delta.
	\end{equation}
	Hence $q_X \in B(q_k, 2^{-j+1}r)$ for some $k$. Moreover $T_{>j} \subset \bigcup\limits_{k} B(q_k, 2\cdot 2^{-j+1}r)$.
	Therefore by \eqref{eq:boundN} and \eqref{eq:weight},
	\[ m(T_{>j}) \leq N_j \cdot \sup_{k} m\left(B(q_k, 2\cdot 2^{-j+1}r)\right) \lesssim 2^{jd} \left(2^{-j} r\right)^{d+1} \sim 2^{-j}r^{d+1}. \]
	Combined with \eqref{eq:intdelta} we get
	\[		\iint_{T(2\Delta)} \delta(X)^{\alpha} dm(X) \lesssim \sum_{j=0}^{\infty} (2^{-j}r)^{\alpha}\cdot 2^{-j}r^{d+1} = r^{d+1+\alpha} \sum_{j=0}^{\infty} 2^{-j(\alpha+1)} \lesssim r^{d+1+\alpha}. \]
	The last sum is convergent because $\alpha+1 >0$.	
\end{proof}

Now we define the suitable function spaces.
We denote by $C_0^0(\Gamma)$ the space of compactly supported continuous functions on $\Gamma$, that is, $f\in C_0^0(\Gamma)$ if $f$ is defined and continuous on $\Gamma$, and there exists a surface ball $\Delta$ such that $\supp f \subset \Delta$. 
We consider the weighted Sobolev space
\begin{equation}\label{def:W}
	W = \dot W_{w}^{1,2}(\Omega) = \{u\in L_{loc}^1(\Omega): \nabla u \in L^2(\Omega, dm) \}
\end{equation}
and set $\|u\|_W = \left( \iint_{\Omega} |\nabla u(X)|^2 dm(X) \right)^{\frac{1}{2}}$ for $u\in W$. 
In fact, it was proved in Lemma 3.3 of \cite{elliptic} that since $\Gamma$ is $d$-Ahlfors regular with $d<n-1$,
\begin{equation}
	W = \{u\in L_{loc}^1(\RR^n): \nabla u \in L^2(\RR^n, dm) \}.
\end{equation}
We also define a local version of $W$ as follows: Let $E\subset \RR^n$ be an open set, define
\begin{equation}
	W_r(E) = \{u\in L_{loc}^1(E): \varphi u \in W \text{ for all } \varphi \in C_0^\infty(E) \}.
\end{equation}
As observed in \cite{elliptic},
\begin{equation}\label{def:Wr}
	W_r(E) = \{u\in L^1_{loc}(E):\nabla u \in L_{loc}^2(E,dm) \}.
\end{equation}
It is easy to see that if $E\subset F$ are open subsets of $ \RR^n$, then the function space $W_r(F) \subset W_r(E)$.
We set
\begin{equation}
	H = \dot H^{\frac{1}{2}}(\Gamma) = \left\{ g\text{ a measurable function on } \Gamma:  \int_{\Gamma} \int_{\Gamma} \frac{|g(x) - g(y)|^2}{|x-y|^{d+1}} d\sigma(x) d\sigma(y) <\infty \right\}.
\end{equation}
The reader may recognize this is the homogeneous Sobolev space, 
%-Slobodeckij space, an $L^2$-H\"older space 
a special case of the Besov spaces.
The authors in \cite{elliptic} were able to define a trace operator $T: W\to H$, see Theorem 3.13 (and Lemma 8.3 for a local version $T: W_r(E) \to L_{loc}^1(\Gamma\cap E)$) there.

\begin{lemma}[interior Caccioppoli inequality, Lemma 8.26 of \cite{elliptic}]\label{lm:intCcpl}
	Let $E\subset \Omega$ be an open set, and let $u\in W_r(E)$ be a non-negative solution in $E$. Then for any $\phi\in C_0^\infty(E)$,
	\begin{equation}\label{eq:intCcplwithtf}
		\iint_\Omega \phi^2 |\nabla u|^2 dm \leq C \iint_\Omega |\nabla \phi|^2 u^2 dm,
	\end{equation}
	where $C$ depends only on $n, d$ and $ C_1$.
	
	In particular, if $B$ is a ball of radius $r$ such that $2B \subset \Omega$ and $u\in W_r(2B)$ is a non-negative sub-solution in $2B$, then 
	\begin{equation}\label{eq:intCcpl}
		\iint_B |\nabla u|^2 dm \leq C r^{-2} \iint_{2B} u^2 dm.
	\end{equation}
\end{lemma}
\begin{remark}
	\eqref{eq:intCcpl} holds if we replace $2B$ by $(1+\tau)B$, $\tau>0$, and in that case the constant $C$ depends on the value of $\tau$.
\end{remark}

\begin{lemma}[Harnack inequality, Lemmas 8.42 and 8.44 of \cite{elliptic}]\label{lm:Hnk}
	\leavevmode
	
	 \begin{enumerate}
	 	\item Let $B$ be a ball such that $3B \subset \Omega$ and let $u\in W_r(3B)$ be a non-negative solution in $3B$. Then 
	 		\begin{equation}\label{eq:Hnk}
	 			\sup_B u \leq C\inf_B u,
	 		\end{equation}
	 		where $C$ depends on $n, d$ and $C_1$.
	 	\item Let $K$ be a compact set of $\Omega$ and $u\in W_r(\Omega)$ be a non-negative solution in $\Omega$. Then
			\begin{equation}\label{eq:Hnkcpt}
				\sup_K u \leq C_K \inf_K u,
			\end{equation}
			where $C_K$ depends only on $n$, $d$, $C_0$, $C_1$, $\dist(K,\Gamma)$ and $\diam K$.
	 \end{enumerate}
\end{lemma}

\begin{lemma}[boundary Caccioppoli inequality, Lemma 8.47 of \cite{elliptic}]
	Let $B\subset \RR^n$ be a ball centered on $\Gamma$ of radius $r$, and let $u\in W_r(2B)$ be a non-negative subsolution in $2B\setminus \Gamma$ such that $Tu=0$ a.e. on $2B$. Then for any $\phi \in C_0^\infty(2B)$,
	\begin{equation}\label{eq:bdCcplintg}
		\iint_{2B} \phi^2 |\nabla u|^2 dm \leq C \iint_{2B} |\nabla \phi|^2 u^2 dm,
	\end{equation}
	where $C$ depends on $n, d$ and $C_1$. In particular \eqref{eq:bdCcplintg} implies that
	\begin{equation}\label{eq:bdCcpl}
		\iint_{B} |\nabla u|^2 dm \leq C r^{-2} \iint_{2B} u^2 dm.
	\end{equation}
\end{lemma}

\begin{lemma}[boundary Moser estimate, Lemma 8.71 of \cite{elliptic}]\label{lm:Moser}
	Let $p>0$. Let $B$ be a ball centered on $\Gamma$ and $u\in W_r(2B)$ be a non-negative sub-solution in $2B\setminus \Gamma$ such that $Tu = 0$ a.e. on $2B$. Then
	\begin{equation}
		\sup_B u \leq C_p \left( \frac{1}{m(2B)} \iint_{2B} u^p dm \right)^{\frac{1}{p}}.
	\end{equation}
\end{lemma}

\begin{lemma}[boundary H\"{o}lder regularity, Lemma 8.106 of \cite{elliptic}]\label{lm:bdHolder}
	Let $B= B(q,r)$ be a ball centered on $\Gamma$ and $u\in W_r(B)$ be a solution in $B$ such that $Tu \equiv 0$ on $B$. There exists $\beta\in (0,1]$ such that for any $0<s<r/2$,
	\begin{equation}\label{eq:bdHolder}
		\underset{B(q,s)}{\osc} u \leq C \left(\frac{s}{r} \right)^\beta \left( \frac{1}{m(B)} \iint_B |u|^2 dm \right)^{\frac{1}{2}}.
	\end{equation}
\end{lemma}

We are interested in the solution(s) of the Dirichlet problem \eqref{ellp}.
\begin{lemma}[existence and uniqueness of solution, Lemma 9.3 of \cite{elliptic}]
	For any $f\in H$, there exists a unique $u\in W$ such that
	\begin{equation}\label{ellppr}
		\left\{\begin{array}{ll}
			Lu = 0 & \text{in } \Omega \\
			Tu = f & \text{a.e. on } \Gamma.
		\end{array} \right.
	\end{equation}
	Moreover $\|u\|_W \leq C \|f\|_H$.
\end{lemma}

\begin{lemma}[properties of solutions for $f\in C_0^0(\Gamma)$, Lemma 9.23 of \cite{elliptic}]\label{lm:dfsol}
	 There exists a bounded linear operator 
	 \[ U: C_0^0(\Gamma) \rightarrow  C(\RR^n) \]
	 such that for every $f\in C_0^0(\Gamma)$
	 \begin{enumerate}[(i)]
	 	\item the restriction of $Uf$ to $\Gamma$ is $f$;
	 	\item $\sup_{\RR^n} Uf = \sup_{\Gamma} f$ and $\int_{\RR^n} Uf = \inf_{\Gamma} f$;
	 	\item $Uf\in W_r(\Omega)$ and is a solution of $L$ in $\Omega$;
	 	\item if $B$ is a ball centered on $\Gamma$ and $f\equiv 0$ on $B$, then $Uf$ lies in $W_r(B)$;
	 	\item if $f\in C_0^0(\Gamma) \cap H$, then $Uf \in W$ and is a unique solution of \eqref{ellppr}.
	 \end{enumerate}
\end{lemma}
\begin{remark}\label{rmk:dfsol}
	Since $Uf\in C(\RR^n)$, its trace $T(Uf)$ is exactly $f$. We also remark that $C_0^0(\Gamma) \cap H$ is dense in $C_0^0(\Gamma)$, with the supremum norm.
\end{remark}

\begin{lemma}[harmonic measure, Lemmas 9.30 and 9.33 of \cite{elliptic}]\label{lm:dfhm}
	For any $X\in\Omega$, there exists a unique positive regular Borel measure $\omega^X$ on $\Gamma$ such that
	\begin{equation}\label{eq:dfhm}
		Uf(X) = \int_{\Gamma} f d\omega^X, \quad \text{ for any } f\in C_0^0(\Gamma).
	\end{equation}
	Besides, for any Borel set $E\subset\Gamma$,
	\begin{equation}\label{eq:hmreg}
		\omega^X(E) = \sup \{\omega^X(K): E \supset K, K \text{ is compact } \} = \inf \{ \omega^X(V): E\subset V, V \text{ is open }\}.
	\end{equation}
	Moreover, $\omega^X(\Gamma) = 1$.
\end{lemma}

\begin{lemma}[Lemma 9.38 of \cite{elliptic}]\label{lm:hmassol}
	Let $E\subset \Gamma$ be a Borel set and define the function $u_E$ on $\Omega$ by $u_E(X) = \omega^X(E)$. Then
	\begin{enumerate}[(i)]
		\item if there exists $X\in\Omega$ such that $u_E(X) = 0$, then $u_E \equiv 0$;
		\item the function $u_E$ lies in $W_r(\Omega)$ and is a solution in $\Omega$;
		\item if $B\subset \RR^n$ is a ball such that $E\cap B = \emptyset$, then $u_E \in W_r(B)$ and $Tu_E = 0$ on $B\cap\Gamma$.
	\end{enumerate}
\end{lemma}

For now we are only able to write down the solution to \eqref{ellp} if the boundary function $f\in C_0^0(\Gamma)$, see Lemma \ref{lm:dfsol}. With the help of the harmonic measure, we prove the following lemma:
\begin{lemma}\label{lm:fchiE}
	For any function $f\in C_0^0 (\Gamma)$ and any Borel set $E\subset \Gamma$, the function 
	\begin{equation}\label{def:solbyitg}
		u(X): = \int_{E} f d\omega^X
	\end{equation} 
	defined on $\Omega$ satisfies the following:
	\begin{enumerate}
		\item it is continuous in $\Omega$;
		\item it is a solution of $Lu=0$ in $\Omega$ and lies in $W_r(\Omega)$;
		\item if $B\subset \RR^n$ is an open ball such that $E\cap B=\emptyset$, then $u$ is continuous in $B\cap\Omega$, $u$ can be continuously extended to zero on $B\cap\Gamma$, and that $u\in W_r(B)$. 
	\end{enumerate}
\end{lemma}

\begin{remark}
	We note the following:
	\begin{itemize}
		\item Compared with Lemma \ref{lm:dfhm} and Lemma \ref{lm:dfsol}, this lemma says that $f\chi_E$ integrated against the harmonic measure gives rise to a continuous solution, for any Borel set $E\subset \Gamma$.
		\item If the Borel set $E$ is bounded, then the same properties hold for any bounded continuous function $f\in C_b(\Gamma)$.
	\end{itemize}
\end{remark}
\begin{proof}
	Since the definition \eqref{def:solbyitg} is a linear integration, we may assume without loss of generality that $f$ is non-negative. Otherwise we just write $f=f_+ - f_-$, with $f_{\pm} \in C(\RR^n)$. We first assume that $E$ is an open set, and that $\omega^{X}(E)>0$ for some $X \in\Omega$. By Lemma \ref{lm:hmassol} (i) it follows that $\omega^{X}(E)>0$ for all $X\in\Omega$. Fix an arbitrary $X_0\in\Omega$. Let $K_j$ be an increasing sequence of compact sets in $E$, such that $\omega^{X_0}(E\setminus K_j) < 1/j$. By Urysohn's lemma we can construct $g_j \in C_0^0(\Gamma)$ such that $\chi_{K_j} \leq g_j \leq \chi_{E}$, and without loss of generality we can choose the sequence $g_j$ to be increasing. Note that $fg_j \in C_0^0(\Gamma)$, and hence by Lemma \ref{lm:dfsol} we may define $u_j = U(fg_j) \in C^0(\Gamma)$. 
	Then
	\[ 0 \leq u(X) - u_j(X) = \int f\left( \chi_E - g_j \right) d\omega^X \leq \omega^X(E\setminus K_j) \|f\|_{L^\infty}. \]
	By Lemmas \ref{lm:hmassol} and \ref{lm:Hnk}, for any compact subset $K$ in $\Omega$ containing $X_0$, we have 
	\[ \omega^X(E\setminus K_j) \leq C_K \omega^{X_0}(E\setminus K_j) \]
	holds for every $X\in K$. Here the constant $C_K$ only depends on $n, d, C_1, \dist(K,\Gamma)$ and $\diam K$, and in particular it is independent of $j$. Therefore
	\[ 0 \leq u(X) - u_j(X) \leq \frac{C_K \|f\|_{L^\infty}}{j}, \]
	namely $\{u_j\}$ converges uniformly on compact sets of $\Omega$ to $u$, and thus $u$ is continuous on $\Omega$.
	
	Let $\phi\in C_0^\infty(\Omega)$ be arbitrary, we claim that $\{u_j\}$ has a subsequence, which we relabel, such that
	\begin{equation}\label{eq:clmgradintg}
		\nabla (\phi u_j) \rightharpoonup \nabla (\phi u) \text{ in } L^2(\Omega, w).
	\end{equation}
	In particular $\nabla(\phi u) \in L^2(\Omega, w)$ for all $\phi \in C_0^\infty(\Omega)$, and thus $u\in W_r(\Omega)$.
	Indeed, by the interior Caccioppoli inequality \eqref{eq:intCcplwithtf}, we have
	\begin{equation}\label{eq:gradintg}
		\iint_{\Omega} |\nabla(\phi u_j)|^2 dm \leq 2 \iint_\Omega \left( |\nabla \phi|^2 u_j^2 + \phi^2 |\nabla u_j|^2 \right) dm \leq C \iint_\Omega |\nabla \phi|^2 u_j^2 dm. 
	\end{equation} 
	Recall that $u_j \to u$ uniformly on the compact set $\supp\phi$, the right hand side of \eqref{eq:gradintg} converges to $C\iint_\Omega |\nabla \phi|^2 u^2 dm$. As a consequence the left hand side of \eqref{eq:gradintg} is uniformly bounded in $j$. Therefore there is a subsequence (which we relabel) such that $\nabla (\phi u_j) $ converges weakly in $L^2(\Omega,w)$ to some function $v$. By the uniqueness of limit in the distributional sense, we conclude that $v=\nabla(\phi u)$, which finishes the proof of the claim \eqref{eq:clmgradintg}.
	
	Recall each $u_j$ is a solution of $L$ in $\Omega$. Let $\varphi \in C_0^\infty(\Omega)$ be an arbitrary test function. We choose $\phi \in C_0^\infty(\Omega)$ such that $\phi \equiv 1$ on $\supp\varphi$. In particular $\nabla(\phi u) = \nabla u$, $\nabla(\phi u_j) = \nabla u_j$ on $\supp\varphi$. Thus
	\begin{align}
		\iint_\Omega A \nabla u \cdot \nabla \varphi dX = \iint_\Omega \mathcal{A} \nabla u \cdot \nabla \varphi dm & = \iint_{\Omega} \mathcal{A} \nabla (\phi u) \cdot \nabla \varphi dm \nonumber \\
		& = \lim_{j\to\infty} \iint_{\Omega} \mathcal{A} \nabla (\phi u_j ) \cdot \nabla \varphi dm \nonumber \\
		& = \lim_{j\to\infty} \iint_\Omega \mathcal{A} \nabla u_j \cdot \nabla \varphi dm  = \lim_{j\to\infty} \iint_\Omega A \nabla u_j \cdot \nabla \varphi dX = 0.
	\end{align}
	
	If $E$ is not an open set, the proof is similar, and we just need to approximate $E$ from above by open sets. We omit the details here.
	
	\smallskip
	
	Going further, if $B\subset \RR^n$ is an open ball such that $E\cap B = \emptyset$, we first prove that $u$ can be continuously extended to zero on $\Gamma \cap B$. Take an arbitrary $q\in \Gamma \cap B$. Choose $r>0$ sufficiently small so that $B(q,2r) \subset B$. Consider a function $g \in C_0^\infty(\RR^n)$ satisfying $\chi_{B(q,r)} \leq g \leq \chi_{B(q,2r)}$. If $f\in C_0^0(\Gamma)$, then $f(1-g) \in C_0^0(\Gamma)$. If the Borel set $E$ is bounded and $f$ is only assumed to be bounded continuous, we let $\varphi \in C_0^\infty(\RR^n)$ be a function such that $\varphi \equiv 1$ on a compact set containing $E$ and $B(q,2r)$. Then $f(1-g)\varphi \in C_0^0(\Gamma)$.
	 Let 
	\[ \widetilde u(X) := U(f(1-g)\varphi ) = \int_{\pO} f(1-g)\varphi d\omega^X. \]
	(For simplicity we take $\varphi \equiv 1$ for case when $f\in C_0^0(\Gamma)$.)
	By the positivity of the harmonic measure and the fact that $E\subset \pO \setminus B(q,2r)$, we deduce that $0\leq u(X) \leq \widetilde u(X)$ for all $X\in\Omega$. Recall by Lemma \ref{lm:dfsol} that $\widetilde u\in C(\RR^n)$, and as $X\to q'\in B(q,r) \cap \Gamma$, the function $\widetilde u(X) \to f(1-g)\varphi(q')=0$. By the squeeze theorem $u$ can be continuously extended to zero on $B(q,r)\cap \Gamma$, and the resulting function, still denoted as $u$, is continuous in $B(q,r)$.
	
	Now we show that $u\in W_r(B)$. To this end, let $\phi \in C_0^\infty(B)$, it suffices to show that $\nabla\left( \phi u \right) \in L^2(B,w)$. From Lemma \ref{lm:dfsol} (iv), Remark \ref{rmk:dfsol} and the boundary Caccioppoli inequality \eqref{eq:bdCcplintg}, we have
	\begin{equation}\label{eq:gradintgbd}
		\iint_B |\nabla (\phi u_j)|^2 dM \leq 2 \iint_{B} \left( |\nabla \phi|^2 u_j^2 + \phi^2 | \nabla u_j|^2 \right) dm \leq C \iint_{B} |\nabla \phi|^2 u_j^2 dm. 
	\end{equation} 
	Recall that $u_j \to u$ pointwise on $B\setminus \Gamma$. Since $u$ is continuous on $B$, $u\in L^2(\supp\phi, w)$. Hence by the dominated convergence theorem the right hand side of \eqref{eq:gradintgbd} converges to $C\iint_{B} |\nabla \phi|^2 u^2 dm$. As a consequence the left hand side is uniformly bounded, and thus passing to a subsequence $\nabla(\phi u_j)$ converges weakly in $L^2(B,w)$ to some function $v$. By the uniqueness of the limit we deduce $v= \nabla(\phi u)$. In particular this implies $\nabla(\phi u)\in L^2(B,w)$.
\end{proof}

As a summary, we can write down the solution of $L$ using the harmonic measure, for the following classes of boundary data: continuous and compactly supported functions $f\in C_0^0(\Gamma)$ (see Lemma \ref{lm:dfsol}), characteristic functions $\chi_E$ for Borel sets $E\subset \Gamma$ (see Lemma \ref{lm:hmassol}), their products $f\chi_E$ (see the above Lemma \ref{lm:fchiE}), or a linear combination of the above. For the third case, if the Borel set $E$ is bounded, we only need to assume $f\in C_b(\Gamma)$.

\begin{lemma}[corkscrew point, Lemma 11.46 of \cite{elliptic}]\label{lm:cksc}
	There exists $M > 1$ such that for any $q\in\Gamma$ and $r>0$, there exists a point $A = A_r(q)\in\Omega$ such that
	\begin{equation}
		|A-q| < r, \quad \delta(A) \geq \frac{r}{M}.
	\end{equation} 
	This point will be referred to as a corkscrew point hereafter.
\end{lemma}
\begin{remark}
	Note that neither Lemma \ref{lm:Hcc} nor Lemma \ref{lm:cksc} is automatically true if $d=n-1$. In fact in the case of co-dimension $1$, people often work with domains that satisfy Harnack chain condition and the existence of corkscrew point at all scales, called uniform domains or $1$-sided NTA domains in the literature.
\end{remark}

\begin{lemma}[boundary Harnack inequality, Lemma 11.50 of \cite{elliptic}]\label{lm:bdHnk}
	Let $q\in\Gamma$ and $r>0$ be given, and let $A= A_r(q)$ be a corkscrew point as in Lemma \ref{lm:cksc}. Let $u\in W_r(B(q,2r))$ be a non-negative, non identically zero solution of $Lu = 0$ in $B(q,2r)\cap\Omega$, such that $Tu \equiv 0$ on $\Delta(q,2r)$. Then
	\begin{equation}\label{eq:bdHnk}
		u(X) \leq Cu(A) \quad \text{ for all } X\in B(q,r).
	\end{equation} 
\end{lemma}

We also recall the following ``classical'' Poincar\'e inequality for Sobolev functions.

\begin{lemma}[Poincar\'e inequality, Lemma 4.13 of \cite{elliptic}]\label{lm:ballPoincare}
	Let $\Gamma$ be a $d$-Ahlfors regular set in $\RR^n$ with $d<n-1$. For any function $v\in W$, $X\in\RR^n$ and $r>0$, let $B=B(X,r)$, then
	\begin{equation}
		\left( \frac{1}{m(B)} \iint_{B} |v(Y)-v_{B}|^2 dm(Y) \right)^{\frac{1}{2}} \leq C r \left( \frac{1}{m(B)} \iint_{B} |\nabla v(Y)|^2 dm(Y) \right)^{\frac{1}{2}},
	\end{equation}
	where $v_B$ denotes the average $m(B)^{-1} \int_B v dm$.
\end{lemma}

Suppose $\Delta= B(q_0,r) \cap\Gamma$ is a surface ball.
For any $q\in\Delta$ and any $j\in\mathbb{N} $, let 
\begin{equation}
	\Gamma_j(q) = \Gamma(q) \cap \left(B(q,2^{-j}r)\setminus B(q,2^{-j-1}r)\right) \label{def:stripe}
\end{equation}  
be a stripe in the cone $\Gamma(q)$ at height $2^{-j}r$, and 
\begin{equation}
	\Gamma_{j\rightarrow j+m}(q) = \bigcup_{i=j}^{j+m}\Gamma_i(q) = \Gamma(q) \cap \left( B(q, 2^{-j}r) \setminus B(q, 2^{-(j+m)-1}r) \right) , \label{def:unionstripe}
\end{equation} 
be a union of $(m+1)$ stripes. With these notations we can prove a less conventional form of Poincar\'e inequality, available for solutions with vanishing boundary values.

\begin{lemma}\label{lm:Poincare}
Suppose that $u\in W_r(\Omega)$ is a non-negative solution of $L$, $Tu=0$ on $3\Delta$ and $u\in W_r(B(q_0,3r))$.
There exist an aperture $\overline\alpha>\alpha$ and integers $m_1, m_2$, such that  
for all $q \in\Delta$, %$q\in \Delta/2$
	\begin{equation}
	\iint_{\Gamma_j^{\alpha}(q)} u^2 dm(X) \leq C (2^{-j}r)^2 \iint_{\Gamma_{j-m_1 \rightarrow j+m_2}^{\overline{\alpha} }(q)} |\nabla u|^2 dm(X).\label{Poincareineq}
\end{equation}
The constants $m_1, m_2, \overline{\alpha}$ and $C$ only depend on $n, d, \alpha, C_0, C_1$.
\end{lemma}

\begin{proof}
	Let $B$ be a ball compactly contained in $\Omega$. Recall that the solution $u \in W_r(\Omega)$, in particular, $\varphi u \in W$ for $\varphi \in C_0^{\infty}(\Omega)$ such that $\varphi \equiv 1$ on $B$. Apply the above Lemma \ref{lm:ballPoincare} to $\varphi u$ and square both sides, we get
\begin{equation}
	\label{eq:ballPoincare}
		 \iint_{B} |u(Y)-u_{B}|^2 dm(Y)  \leq C r_B^2  \iint_{B} |\nabla u(Y)|^2 dm(Y),
\end{equation}

For $j\in\mathbb{N}$, let $A_j$ denote a corkscrew point for $B(q,2^{-j}r)$, whose existence is guaranteed by Lemma \ref{lm:cksc}.
Let $m$ be a large integer whose value is to be determined later. Take $X\in \Gamma_j^\alpha(q)$, $X' = A_{j+m} $, then 
\begin{equation}\label{eq:tmp58}
	\delta(X) > \frac{1}{1+\alpha}|X-q| \geq  \frac{2^{-j-1} r}{1+\alpha}, \quad \delta(X') \geq \frac{ 2^{-(j+m)}r}{M},
\end{equation}
\[ |X-X'| \leq |X-q| + |q-X'| \leq  2^{-j} r + 2^{-(j+m)} r \leq 2^{1-j} r. \] 
Apply Lemma \ref{lm:Hcc} and Remark \ref{rmk:Hcc} to $X, X'$ with $s = 2^{-(j+m)}r/M$ and $\Lambda = 2^{m+1} M$, we can find balls $B_0 = B(X,s/2)$, $B_i = B(Z_i, \tau/4)$ with $\tau = c\Lambda^{-d/(n-1-d)}s$, $B_{N+1}=B(X',s/2)$ that form a Harnack chain connecting $X$ to $X'$, and satisfy \eqref{eq:Hccount}, \eqref{eq:Hclb} and \eqref{eq:Hcub}.
Hence by Lemma 2.3 (i) of \cite{elliptic} and \eqref{eq:Hcub}, \eqref{eq:Hclb}, we have
\begin{equation}\label{eq:Hccweightlb}
	m(B_i) \geq C^{-1} \left( \frac{\tau}{4} \right)^n \dist(B_i,\Gamma)^{d-n+1} \gtrsim \tau^n (\Lambda s)^{d-n+1} \sim \Lambda^{1-n} \tau^{d+1},
\end{equation}
and
\begin{equation}\label{eq:Hccweightub}
	m(B_i) \leq C\left(\frac{\tau}{4} \right)^n \dist(B_i,\Gamma)^{d-n+1} \lesssim \tau^n  \tau^{d-n+1} \sim  \tau^{d+1}
\end{equation}
for all $i=0,\cdots,N,N+1$. 
A simple computation shows $B_{i+1} \subset 3B_i$ for all $i=1, \cdots N-1$, and $B_1 \subset \frac{3}{2} B_0, B_N \subset \frac{3}{2} B_{N+1}$, if $m$ is sufficiently large.
Therefore for each $i=1, \cdots, N-1$,
\begin{align}
	|u_{B_{i+1}} -u_{3 B_{i}}|^2 
	& \leq \left( \frac{1}{m(B_{i+1})} \iint_{B_{i+1}}|u(X)-u_{3 B_i}| dm(X) \right)^2 \nonumber \\
	& \leq \frac{1}{m(B_{i+1})}  \iint_{3 B_i} |u(X) - u_{3 B_i}|^2 dm(X) \nonumber \\
	& \lesssim \Lambda^{n-1}\tau^{1-d}  \iint_{3B_i} |\nabla u(Y)|^2 dm(Y) \qquad \text{by \eqref{eq:ballPoincare}}, \eqref{eq:Hccweightlb}.
\end{align}
Similarly
\[ |u_{B_i}-u_{3 B_i}|^2 \lesssim \Lambda^{n-1}\tau^{1-d} \iint_{3 B_i}|\nabla u(Y)|^2 dm(Y). \]
Hence
\begin{equation}
	|u_{B_i}-u_{B_{i+1}}|^2 \leq C\Lambda^{n-1}\tau^{1-d} \iint_{3 B_i}|\nabla u(Y)|^2 dm(Y). \label{eq:compave}
\end{equation}
A similar argument shows that for the end-point case $i=0$ or $N+1$,
\begin{align}\label{eq:comptop}
	|u_{B_i} - u_{B_{i\pm 1}}|^2 & \lesssim  \max\{s^{1-d}, \Lambda^{n-1} s^2 \tau^{-1-d} \} \iint_{\frac{3}{2} B_i} |\nabla u(Y)|^2 dm(Y) \nonumber \\
	& \sim \Lambda^{n-1} s^2 \tau^{-1-d} \iint_{\frac{3}{2} B_i} |\nabla u(Y)|^2 dm(Y) . 
\end{align}
The last line is justified since $\Lambda \gg 1$ implies $ \tau \ll s$.
Combining this observation, \eqref{eq:compave}, \eqref{eq:comptop} and \eqref{eq:Hccount}, we get
\begin{align}
	\iint_{B_0} |u(X) - u_{B_{N+1}}|^2 dm(X) & \lesssim N \cdot \iint_{B_0} |u(X) - u_{B_0}|^2 dm(X) + N \cdot m(B_0) \sum_{i=0}^{N} |u_{B_{i}} - u_{B_{i+1}} |^2 \nonumber \\
	& \lesssim N\Lambda^{n-1} s^2 \left(\frac{s}{\tau} \right)^{d+1}  \iint_{\frac{3}{2}  B_0 \bigcup \left(\bigcup\limits_{i=1}^{N} 3B_i \right) \bigcup \frac{3}{2} B_{N+1}} |\nabla u(Y)|^2 dm(Y) \nonumber \\
	& \leq C' \Lambda^{\frac{n-1+d(d+1)}{n-1-d} + n-1} s^2 \iint_{\frac{3}{2}  B_0 \bigcup \left(\bigcup\limits_{i=1}^{N} 3B_i \right) \bigcup \frac{3}{2} B_{N+1} } |\nabla u(Y)|^2 dm(Y).\label{eq:compsbst}
\end{align}
 
On the other hand, by Harnack inequality
\[ u(X) \leq C u(X') \quad \text{ for all } X\in B_{N+1} = B(X',s/2). \] 
Recall that $X' = A_{j+m}$. For any $q\in \Delta$, by the assumption we know that $u\in W_r(B(q,2r))$ vanishes on $\Delta(q,2r)$. By the boundary H\"older regularity (Lemma \ref{lm:bdHolder}) and boundary Harnack principle (Lemma \ref{lm:bdHnk}) we have
\[ u(X') \leq C 2^{-m\beta} u(A_j), \]
with a constant $C$ independent of $j$ and $m$. Thus
\begin{equation}\label{eq:comptail}
	u_{B_{N+1}}^2 \lesssim u^2(X') \lesssim 2^{-2m\beta} u^2(A_j) \lesssim 2^{-2m\beta} \cdot \frac{1}{m(B_0)} \iint_{B_0} u^2 dm(X) .
\end{equation} 
The last inequality holds because $A_j$ is a corkscrew point and $B_0 = B(X,s/2)$ for some $X\in\Gamma_j(q)$.
Combining \eqref{eq:comptail} and \eqref{eq:compsbst} we obtain
\begin{align}
	& \iint_{B_0} u^2 dm(X) \nonumber \\
 & \leq 2  m(B_0)\left(u_{B_{N+1}}\right)^2 + 2\iint_B |u(x) - u_{B_{N+1}}|^2 dm(X) \nonumber \\
	& \leq A_1 2^{-2m\beta} \iint_{B_0} u^2 dm(X) + A_2 \Lambda^{\frac{n-1+d(d+1)}{n-1-d} + n-1} s^2 \iint_{\frac{3}{2} B_0 \bigcup \left(\bigcup\limits_{i=1}^{N} 3B_i\right) \bigcup \frac{3}{2} B_{N+1}} |\nabla u(Y)|^2 dm(Y). \label{eq:beforeabsorb}  
\end{align}
Choose $m$ big enough such that 
\begin{equation}\label{eq:largem}
	A_1 2^{-2m \beta} \leq \frac{1}{2}, \quad \text{ as well as } 2 \cdot \frac{2^{-m}}{M} \leq \frac{1}{2(1+\alpha)},
\end{equation}
 then we can absorb the first term on the right hand side of \eqref{eq:beforeabsorb} to the left. Recall that $B_0 = B(X, s/2)$ for $X$ satisfying \eqref{eq:tmp58}. The reason for the second assumption in \eqref{eq:largem} is to guarantee the enlarged ball $\frac{3}{2} B_0$ is compactly contained in $\Omega$. 
 Fix the value of $m$ from now on, thus the value of $\Lambda = 2^{m+1}/M $ is also fixed. We get
\begin{equation}\label{eq:PoincareNTball}
	\iint_{B_0} u^2 dm(X) \leq C s^2 \iint_{\frac{3}{2} B_0 \bigcup \left(\bigcup\limits_{i=1}^{N} 3 B_i\right) \bigcup \frac{3}{2} B_{N+1}} |\nabla u(y)|^2 dy, 
\end{equation}
where $s=2^{-(j+m)}r/M$ and the constant $C$ depends on $d,n,C_0,C_1$ (Recall the values of corkscrew constant $M$ and Harnack chain constant $c$ only depend on $d,n,C_0,C_1$). 
Since $B_0 = B(X, s/2)$ with center $X\in\Gamma_j^{\alpha}(q)$, it is a simple exercise to show that given the second assumption of \eqref{eq:largem}, there exists an aperture $\alpha_1>\alpha$ such that
\begin{equation}\label{eq:IR1}
	\frac{3}{2} B_0 \subset \mathlarger{\Gamma}^{\alpha_1}_{j-1 \rightarrow j+1}(q).
\end{equation}
A similar statement holds for $\frac{3}{2} B_{N+1}$.
Moreover \eqref{eq:Hclb} and \eqref{eq:Hcub} imply that for $i=1, \cdots, N$, there exist an aperture $\alpha_2>\alpha$ and an integer $m_0$ depending on the constants $c, M$ from Lemmas \ref{lm:Hcc} and \ref{lm:cksc}, such that
\begin{equation}\label{eq:IR2}
		3B_i \subset \mathlarger{\Gamma}_{j-3 \rightarrow j+m+m_0 }^{\alpha_2}(q).
\end{equation}
Let $\overline\alpha = \max\{\alpha_1, \alpha_2\}$.
Combining the above observations with \eqref{eq:PoincareNTball} we get
\begin{equation}\label{eq:PoincareNTballIR}
	\iint_{B_0} u^2 dm(X) \leq C s^2 \iint_{\mathlarger{\Gamma}_{j-3 \rightarrow j+m+m_0 }^{\overline\alpha}(q) } |\nabla u(y)|^2 dy. 
\end{equation}
	
Consider the covering
\begin{equation}
	\Gamma_j^{\alpha}(q) \subset \bigcup_{X\in\Gamma_j^\alpha(q)} B\left( X, \frac{s}{10} \right).
\end{equation}
We can extract a finite Vitali sub-covering $\{B^k = B(X_k, s/2)\}_k$ such that
\begin{equation}
	\Gamma_j^{\alpha}(q) \subset \bigcup_k B^k
\end{equation}
and $\{B^k/5 = B(X_k, s/10)\}_k$ is mutually disjoint. 
Moreover the number of balls $B^k$'s is uniformly bounded by a constant $C(n, m, M)$. Note that \eqref{eq:PoincareNTballIR} holds for all such balls $B^k$ in place of $B_0$, we deduce
\begin{align}
	\iint_{\Gamma_j^\alpha(q)} u^2 dm(X) \leq \sum_k \iint_{B^k} u^2 dm(X)  \leq CC(n,m,M) s^2 \iint_{\mathlarger{\Gamma}_{j-3 \rightarrow j+m+m_0 }^{\overline\alpha}(q) } |\nabla u(y)|^2 dy.
\end{align}
Since the value of $m$ is fixed, we finish the proof of Lemma \ref{lm:Poincare}.
\end{proof}

\begin{lemma}[non-degeneracy of harmonic measure, Lemma 11.73 of \cite{elliptic}]\label{lm:nondeg}
	Let $\lambda>1$ be given. There exists a constant $C_\lambda >1$ such that for any $q\in\Gamma$, $r>0$, and $A=A_r(q)$, a corkscrew point from Lemma \ref{lm:cksc}, we have
	\begin{equation}\label{eq:nondegclose}
		\omega^X(B(q,r)\cap\Gamma) \geq C_{\lambda}^{-1} \quad \text{for } X\in B(q,r/\lambda), 
	\end{equation}
	\begin{equation}\label{eq:nondegfar}
		\omega^X(B(q,r)\cap\Gamma) \geq C_{\lambda}^{-1} \quad \text{for } X\in B(A,\delta(A)/\lambda).
	\end{equation}
\end{lemma}

In \cite{elliptic} the authors also prove the existence, uniqueness and properties of the Green function, that is, formally, a function $G$ defined on $\Omega \times \Omega$ such that for any $Y\in\Omega$,
\[ \left\{ \begin{array}{ll}
 	LG(\cdot,Y) = \delta_Y & \text{ in } \Omega \\
 	G(\cdot,Y) = 0 & \text{ on }\Gamma
 \end{array}\right.
\]
where $\delta_Y$ is the delta function.

\begin{lemma}[estimates of Green function, Lemma 11.78 of \cite{elliptic}]\label{lm:CFMS}
	There exists a constant $C\geq 1$, such that for any $q\in\Gamma$ and $r>0$, $\Delta = B(q,r)\cap\Gamma$ and a corkscrew point $A = A_r(q)$, then
	\begin{equation}\label{eq:CFMS}
		C^{-1} r^{d-1} G(X_0,A) \leq \omega^{X_0}(\Delta) \leq C r^{d-1} G(X_0,A) \quad \text{for } X_0 \in\Omega \setminus B(q,2r).
	\end{equation}
\end{lemma}

\begin{lemma}[doubling of harmonic measure, Lemma 11.102 of \cite{elliptic}]\label{lm:doubling}
	 For $q\in\Gamma$ and $r>0$, we have
	\begin{equation}\label{eq:doubling}
		\omega^X(B(q,2r)\cap\Gamma) \leq C \omega^X(B(q,r)\cap\Gamma)
	\end{equation}
	for any $X\in\Omega \setminus B(q,4r)$.
\end{lemma}

\begin{lemma}[change of poles, Lemma 11.135 of \cite{elliptic}]\label{lm:cop}
	Let $q\in\Gamma$ and $r>0$ be given, and let $A = A_r(q)$ be a corkscrew point as in Lemma \ref{lm:cksc}. Let $E, F \subset \Delta(q,r)$ be two Borel subsets of $\Gamma$ such that $\omega^A(E)$ and $\omega^A(F)$ are positive. Then 
	\begin{equation}
		\frac{\omega^X(E)}{\omega^X(F)} \sim \frac{\omega^A(E)}{\omega^A(F)}, \quad \text{ for any }X\in \Omega \setminus B(q,2r).
	\end{equation}
	In particular with the choice $F=\Delta(q,r)$, 
	\begin{equation}
		\frac{\omega^X(E)}{\omega^X(\Delta(q,r))} \sim \omega^A(E) \quad \text{ for any } X\in\Omega \setminus B(q,2r).
	\end{equation}
\end{lemma}

%%%%%%%%%%%%%%%%%%%%%%%%%%%%%%%%%%%%%%%%%%%%%%%%%%%%%%%%%%%%%%%%%%%%%%%%%%%%%%%%%
%%%%%%%%%%%%%%%%%%%%%%%%%%%%%%%%%%%%%%%%%%%%%%%%%%%%%%%%%%%%%%%%%%%%%%%%%%%%%%%%%

Let us restate the definition of $\omega \in A_\infty(\sigma)$ and make a few remarks that will become useful later.
\begin{defn}\label{def:Ainfty}
	We say the harmonic measure $\omega$ is of class $A_\infty$ with respect to the surface measure $\sigma = \mathcal{H}^d|_{\Gamma}$, or simply $\omega \in A_\infty(\sigma)$, if for any $\epsilon>0$, there exists $\delta=\delta(\epsilon)>0$ such that for any surface ball $\Delta$, any surface ball $\Delta' \subset \Delta$ and any Borel set $E\subset \Delta'$, we have
	\begin{equation}\label{eq:defAinfty}
		\frac{\sigma(E)}{\sigma(\Delta')} <\delta \implies \dfrac{\omega^A(E)}{\omega^A(\Delta')} <\epsilon.
	\end{equation}
	Here $A = A_{\Delta} $ is a corkscrew point for $\Delta$ (see Lemma \ref{lm:cksc}).
\end{defn}
\begin{remark}\label{rmk:defAinfty}
	\begin{enumerate}[(i)]
		\item The reader may recall that the standard definition for $A_\infty$ is that the harmonic measure with a fixed pole, i.e. $\omega^{X_0}$, satisfies \eqref{eq:defAinfty}. For unbounded boundary $\Gamma$ though, the standard definition needs to be replaced by its scale-invariant analogue, which is Definition \eqref{def:Ainfty}. In fact since $\Gamma$ is unbounded, it is impossible to have $\omega^{X_0}\in A_\infty(\sigma)$ with a fixed pole $X_0$, see the comments after Theorem 1.18 of \cite{Ainfty}.
		\item The above definition is symmetric: suppose $\omega\in A_{\infty}(\sigma)$, then we also have $\sigma \in A_{\infty}(\omega)$ (in a scale-invariant sense), i.e., the smallness of $\omega^A(E)/\omega^A(\Delta')$ implies the smallness of $\sigma(E)/\sigma(\Delta')$.
		\item In particular, the assumption \eqref{eq:defAinfty} implies that $\omega^A \ll \sigma$ when restricted to $\Delta$. We denote the Radon-Nikodym derivative by $k^A = \frac{d\omega^A}{d \sigma}$. Since both $\omega^A$ and $\sigma$ are Radon measures, we have
\begin{equation}\label{eq:kernelbydensity}
	k^A(q) = \lim_{\Delta'=\Delta(q,r) \atop{r\to 0}} \frac{\omega^A(\Delta')}{\sigma(\Delta')}, \quad \text{ for } \sigma \text{-a.e. } q\in \Delta.
\end{equation}
 Moreover since $\sigma$ is doubling, by standard harmonic analysis techniques (see \cite{weight} for example for the proof) \eqref{eq:defAinfty} implies that $k^A$ satisfies a reverse H\"older inequality: there are constants $r_0>1, C>0$ such that for all $r\in (1,r_0)$,
			\begin{equation}\label{eq:RHkernel}
				\left( \fint_{\Delta} \left|k^{A} \right|^{r} d\sigma \right)^{\frac{1}{r}} \leq C \fint_\Delta k^A d\sigma.
			\end{equation} 
			The constants $r_0$ and $C$ only depend on the constants characterizing the $A_{\infty}$ property \eqref{eq:defAinfty}; in particular, they are independent of $\Delta$ and $ A$.
	\end{enumerate}
\end{remark}
			
Recall that one of our main goals is to prove Theorem \ref{thm:main}, which states the equivalence between $\omega \in A_\infty(\sigma)$ and the BMO solvability of the Dirichlet problem. We make a few preliminary remarks.

Note that $(\pO, \sigma)$ is a space of homogeneous type. By John-Nirenberg inequality for space of homogeneous type, we may also use any $L^p$ norm ($1\leq p<\infty$) in the definition \eqref{eq:defBMO}, and the resulting BMO norms are all equivalent. See \cite{SHT} and \cite{JN}. Also it is easy to see that if $f\in L^\infty(\pO)$, then $f$ is a BMO function with $\|f\|_{BMO} \leq \sqrt{2} \|f\|_{L^\infty}$.

We observe that the Carleson measure norm of $|\nabla u|^2\delta(X) dm(X)$ is in some sense equivalent to the integral of the truncated square function.
Suppose $\Delta=\Delta(q_0,r)$ is an arbitrary surface ball. For any $X\in T(\Delta)$, we define $\Delta^X=\{q\in\pO: X\in \Gamma(q)\}$. Let $q_X\in\pO$ be a point such that $|X-q_X|=\delta(X)$. Then 
\begin{equation}
	\Delta(q_X, \alpha\delta(X))\subset \Delta^X\subset \Delta(q_X, (\alpha+2)\delta(X)). \label{DeltaX}
\end{equation} 
Since $\Gamma$ is $d$-Ahlfors regular, \eqref{DeltaX} implies $\sigma(\Delta^X) \approx \delta(X)^{d}$. Thus
\begin{align}
 	\iint_{T(\Delta)} \Carl{u} dm(X) & \approx \iint_{T(\Delta)} |\nabla u|^2 \delta(X)^{1-d} \sigma(\Delta^X ) dm(X) \nonumber \\
 	& =\iint_{T(\Delta)} |\nabla u|^2 \delta(X)^{1-d} \int_{ \Delta^X} d\sigma(q) dm(X). \label{eq:CarlesonSquarepre}
\end{align} 
Changing the order of integration, on one hand we get an upper bound
\begin{align}
  \iint_{T(\Delta)} |\nabla u|^2 \delta(X)^{1-d} \int_{\Delta^X} d\sigma(q) dm(X) & \leq \int_{|q-q_0|<(\alpha+2)r}\iint_{\Gamma_{(\alpha+1) r}(q)} |\nabla u|^2 \delta(X)^{1-d} dm(X)d\sigma \nonumber \\
  & \leq \int_{(\alpha+2)\Delta}|S_{(\alpha+1) r}u|^2d\sigma. \label{eq:CarlesonSquareU}
\end{align}
On the other hand, we get a lower bound
\begin{align}
  	\iint_{T(\Delta)} |\nabla u|^2 \delta(X)^{1-d} \int_{ \Delta^X} d\sigma(q) dm(X) & \geq \int_{|q-q_0|<r/2}\iint_{\Gamma_{r/2}(q)} |\nabla u|^2 \delta(X)^{1-d} dm(X)d\sigma \nonumber \\
  & \geq \int_{\frac{1}{2} \Delta} |S_{r/2} u|^2 d\sigma. \label{eq:CarlesonSquareL}
\end{align}
Therefore for any $q_0\in\pO$,
\begin{equation}
	\sup_{\substack{\Delta=\Delta(q_0,s) \\ s>0}} \frac{1}{\sigma(\Delta)} \iint_{T(\Delta)} \Carl{u} dm(X) \approx \sup_{\substack{\Delta=\Delta(q_0,r) \\ r>0}} \frac{1}{\sigma(\Delta)} \int_{\Delta} |S_{r} u|^2 d\sigma.
\end{equation} 

%%%%%%%%%%%%%%%%%%%%%%%%%%%%%%%%%%%%%%%%%%%%%%%%%%%%%%%%%%%%%%%%%%%%%%%%%%%%%%%%%%%%%%%%%%%
%%%%%%%%%%%%%%%%%%%%%%%%%%%%%%%%%%%%%%%%%%%%%%%%%%%%%%%%%%%%%%%%%%%%%%%%%%%%%%%%%%%%%%%%%%%

\section{Bound of the square function by the non-tangential maximal function}\label{sect:SlessN}

The goal of this section is to prove:

\begin{theorem}\label{thm:SlessthanN}
	Let $\Gamma$ be a $d$-Ahlfors regular set in $\RR^n$ with an integer $d\leq n-1$, and let $\omega$ be the harmonic measure of the domain $\Omega=	\RR^n \setminus \Gamma$.
	 Suppose $\omega\in A_{\infty}(\sigma)$, then
	\begin{equation}\label{eq:SlessthanN}
		\|Su\|_{L^p(\sigma)} \leq C\|Nu\|_{L^p(\sigma)}
	\end{equation}
	 for any $1\leq p<\infty$ and any solution $u\in W_r(\Omega)$ to $Lu=0$ such that the right hand side is finite. Here the constant $C>0$ depends on the allowable parameters $d, n, C_0, C_1$, the aperture $\alpha$ and the $A_\infty$ constant(s).
\end{theorem}
It suffices to prove \eqref{eq:SlessthanN} for non-negative harmonic functions $u$, because otherwise, we just split $u= u_+ - u_-$ and use the linearity of $L$ and the triangle inequality.
Before starting to prove the theorem we need to recall some notation and preliminary results.

\begin{lemma}[dyadic cubes for Ahlfors regular set, \cite{DS1, DS2, Ch}]\label{lm:dcAR}
	Let $\Gamma\subset \RR^n$ be a $d$-Ahlfors regular set. Then there exist constants $a_0, A_1, \gamma>0$, depending only on $d, n$ and $C_0$, such that for each $k\in\mathbb{Z}$, there is a collection of Borel sets (``dyadic cubes'')
	\[ \mathbb{D}_k : = \{Q_j^k \subset \Gamma: j\in \mathscr{J}_k \}, \]
	where $\mathscr{J}_k$ denotes some index set depending on $k$, satisfying the following properties.
	\begin{enumerate}[(i)]
		\item $\Gamma = \bigcup_{j\in\mathscr{J}_k} Q_j^k$ for each $k\in\mathbb{Z}$.
		\item If $m\geq k$ then either $Q_i^m \subset Q_j^k$ or $Q_i^m \cap Q_j^k = \emptyset$.
		\item For each pair $(j,k)$ and each $m<k$, there is a unique $i\in\mathscr{J}_m$ such that $Q_j^k \subset Q_i^m$.
		\item $\diam Q_j^k \leq A_1 2^{-k}$.
		\item Each $Q_j^k$ contains some surface ball $\Delta(x_j^k, a_0 2^{-k}) := B(x_j^k, a_0 2^{-k}) \cap\Gamma$.
		\item $\mathcal{H}^d \left( \left\{ q\in Q_j^k: \dist(q, \Gamma\setminus Q_j^k) \leq \rho 2^{-k} \right\} \right) \leq A_1 \rho^{\gamma} \mathcal{H}^d (Q_j^k)$, for all $(j, k)$ and all $\rho\in (0,a_0)$.
	\end{enumerate}
	We shall denote by $\mathbb{D} = \mathbb{D}(\Gamma)$ the collection of all relevant $Q_j^k$, i.e.
	\[ \DD = \bigcup_k \DD_k. \]
\end{lemma}
\begin{remark}
	\begin{enumerate}[i)]
		\item For a dyadic cube $Q\in\DD$, we let $k(Q)$ denote the ``dyadic generation'' to which $Q$ belongs, i.e. we set $k(Q)=k$ if $Q\in\DD_k$. We also set its ``length'' $\ell(Q) = 2^{-k(Q)}$. Thus $\ell(Q) = 2^{-k(Q)} \sim \diam Q$.
		\item Properties (iv) and (v) imply that for each cube $Q\in\DD$, there is a point $x_Q \in \Gamma$ such that 
			\begin{equation}\label{nesteddcube}
				\Delta(x_Q, r_Q) \subset Q \subset \Delta(x_Q, C_2 r_Q),
			\end{equation}
			where $r_Q = a_0 2^{-k(Q)}\sim \diam Q$ and $C_2 = A_1/a_0$.
	\end{enumerate}	
\end{remark}

Now we define sawtooth domains following the definitions of Hofmann and Martell, see for example \cite{HM}, \cite{HMM} and \cite{HMT}.
Since $\Omega$ is an open set, it has a Whitney decomposition, that is, a collection of closed ``Whitney'' boxes in $\Omega$, denoted by $\WW = \WW(\Omega)$, which form a covering of $\Omega$ with pairwise non-overlapping interiors and satisfy
\begin{equation}\label{Wbox}
	4\diam I \leq \dist(4I, \pO) \leq \dist(I,\pO) \leq 40 \diam I, \quad \text{ for any } I\in\WW,
\end{equation}
and also 
\begin{equation}\label{ngWbox}
	\frac{1}{4} \diam I_1 \leq \diam I_2 \leq 4\diam I_1
\end{equation}
whenever $I_1$ and $I_2$ in $\WW$ touch. (See \cite{SteinSI} for reference.) Let $X_I$ denote the center of $I$ and $\ell(I) $ the side length of $I$, then $\diam I \sim \ell(I)$. We also write $k(I) = k$ if $\ell(I) = 2^{-k}$. 

Let $\mathbb{D}$ be a collection of dyadic cubes for the Ahlfors regular set $\Gamma$, as in Lemma \ref{lm:dcAR}.
For any dyadic cube $Q\in\DD$, pick two parameters $\eta\ll 1$ and $K\gg 1$, and define
\begin{equation}\label{def:WQ0}
	\WW_Q^0 : = \{I\in\WW: \eta^{\frac{1}{4}} \ell(Q) \leq \ell(I) \leq K^{\frac{1}{2}} \ell(Q), \dist(I,Q) \leq K^{\frac{1}{2}} \ell(Q)\}.
\end{equation} 
Let $X_Q$ denote a corkscrew point for the surface ball $\Delta(x_Q, r_Q/2)$. We can guarantee that $X_Q $ is in some $I\in\WW_Q^0$ provided we choose $\eta$ small enough and $K$ large enough. For each $I\in\WW_Q^0$, by Lemma \ref{lm:Hcc} and the discussions after that, there is a Harnack chain connecting $X_I$ to $X_Q$, we call it $\mathcal{H}_I$. By the definition of $\WW_Q^0$ we may construct this Harnack chain so that it consists of a bounded number of balls (depending on the values of $\eta, K$), and stays a distance at least $c\eta^{\frac{n-1}{4(n-1-d)}}\ell(Q)$ away from $\pO$ (see \eqref{eq:Hclb}). We let $\WW_Q$ denote the set of all $J\in\WW$ which meet at least one of the Harnack chains $\mathcal{H}_I$, with $I\in \WW_Q^0$, i.e.
\begin{equation}
	\WW_Q := \{J\in\WW: \text{there exists } I\in\WW_Q^0 \text{ for which } \mathcal{H}_I \cap J \neq \emptyset\}.
\end{equation}
Clearly $\WW_Q^0 \subset \WW_Q$. Besides, it follows from the construction of the augmented collections $\WW_Q$ and the properties of the Harnack chains (in particular \eqref{eq:Hclb} and \eqref{eq:Hcub}) that there are uniform constants $c$ and $C$ such that
\begin{equation}\label{def:WQ}
	c\eta^{\frac{n-1}{4(n-1-d)} } \ell(Q) \leq \ell(I) \leq CK^{\frac{1}{2}} \ell(Q), \quad \dist(I,Q) \leq CK^{\frac{1}{2}} \ell(Q) 
\end{equation}
for any $I\in \WW_Q$. In particular once $\eta, K$ is fixed, for any $Q\in \DD$ the cardinality of $\WW_Q$ is uniformly bounded, which we denote by $N_0$.

Next we choose a small parameter $\theta\in (0,1)$ so that for any $I\in\WW$, the concentric dilation $I^* = (1+\theta)I$ still satisfies the Whitney property
\begin{equation}\label{eq:Wboxdl}
	\diam I \sim \diam I^* \sim \dist(I^*,\pO) \sim \dist(I,\pO).
\end{equation}
Moreover by taking $\theta$ small enough we can guarantee that $\dist(I^*, J^*) \sim \dist(I,J)$ for every $I, J\in\WW$, $I^*$ meets $J^*$ if and only if $\partial I$ meets $\partial J$ and that $\frac{1}{2} J\cap I^*=\emptyset$ for any distinct $I, J\in\WW$. In what follows we will need to work with further dilations $I^{**} = (1+2\theta) I$ or $I^{***} = (1+4\theta)I$ etc.. (We may need to take $\theta$ even smaller to make sure the above properties also hold for $I^{**}, I^{***}$ etc..)
Given an arbitrary $Q\in\DD$, we may define an associated Whitney region $U_Q$, $U_Q^*$ as follows
\begin{equation}
	U_Q: = \bigcup_{I\in\WW_Q} I^*, \quad U_Q^*: = \bigcup_{I\in\WW_Q} I^{**}.
\end{equation}

Let $\DD_Q = \{Q'\in\DD: Q' \subset Q\}$. For any $Q\in\DD$ and any family $\mathcal{F}= \{Q_j\}$ of disjoint cubes in $ \DD_Q\setminus\{Q\}$, we define the local discretized sawtooth relative to $\FF$ by
\begin{equation}
	\DD_{\FF,Q} : = \DD_Q \setminus \bigcup_{Q_j\in\FF} \DD_{Q_j}.
\end{equation} 
We also define the local sawtooth domain relative to $\FF$ by
\begin{equation}\label{def:OFQ}
	\Omega_{\FF,Q}:= \interior \left( \bigcup_{Q'\in\DD_{\FF,Q}} U_{Q'}\right), \quad \Omega_{\FF,Q}^*:= \interior \left( \bigcup_{Q'\in\DD_{\FF,Q}} U_{Q'}^*\right)
\end{equation}
For convenience we set
\begin{equation}\label{def:WFQ}
	\WW_{\FF,Q} := \bigcup_{Q'\in\DD_{\FF,Q}} \WW_{Q'},
\end{equation}
so that in particular, we may write
\begin{equation}
	\Omega_{\FF,Q} = \interior \left(\bigcup_{I\in\WW_{\FF,Q}} I^* \right), \quad \Omega_{\FF,Q}^* = \interior \left(\bigcup_{I\in\WW_{\FF,Q}} I^{**} \right).
\end{equation}
We will need further fattened sawtooth domain $\Omega_{\FF,Q}^{**}$ etc. whose definitions follow the same lines as above. We remark that by \eqref{def:WQ}, there is a constant $C_3$ depending on $K, \theta$ such that
\begin{equation}\label{stinball}
	\Omega_{\mathcal{F},Q} \subset B(x_Q, C_3 \ell(Q)) \cap \Omega
\end{equation}
for any $Q\in\DD$ and collection of maximal cubes $\mathcal{F}$, where $x_Q$ is the ``center'' of $Q$ as in \eqref{nesteddcube}.

Finally, to work with sawtooth domains, it is more natural to use a discrete dyadic version of the approach region rather than the standard non-tangential cone defined in \eqref{def:NTC}: for every $q\in\pO$, we define the dyadic non-tangential cones as
\begin{equation}\label{def:dyaNTC}
	\Gamma_{d}(q) = \bigcup_{Q\in\DD: Q\ni q} U_Q, \quad  \widehat{\Gamma}_d(q) = \bigcup_{Q\in\DD: Q\ni q} U_Q^{***}
\end{equation}
where we use $\widehat{\Gamma}_d$ to denote a cone with bigger ``aperture'' or fattened region;
we also define the local dyadic non-tangential cones as
\begin{equation}
	\Gamma_d^Q(q) = \bigcup_{Q'\in\DD_Q : Q'\ni q} U_{Q'}, \quad \widehat{\Gamma}_d^Q (q) = \bigcup_{Q'\in\DD_Q : Q'\ni q} U_{Q'}^{***}.
\end{equation}

We claim that given an aperture $\alpha>0$, there exists $K$ (in the definition \eqref{def:WQ0}) sufficiently large such that the standard non-tangential cone $\Gamma^{\alpha}(q) \subset \Gamma_d(q)$ for all $q\in\pO$; and vice versa, for fixed values of $\eta, K$ and the dilation constant $\theta$, there exists $\alpha_1>0$ such that the dyadic cone $\Gamma_d(q) \subset \Gamma^{\alpha_1}(q)$ for all $q\in\pO$.
For any $X\in \Gamma^\alpha(q)$, let $I$ be a Whitney box such that $X\in I^*$. By \eqref{Wbox} we know $ \ell(I) \sim \delta(X) $.
Let $Q$ be a cube containing $q$ with length $\ell(Q) = \ell(I)$. Then
\begin{equation}\label{eq:dyacNTC}
	\dist(I,Q) \leq |X-q| < (1+\alpha) \delta(X) \leq C (1+\alpha) \ell(I) = C(1+\alpha)\ell(Q). 
\end{equation} 
If $K$ is sufficiently large so that $K^{\frac{1}{2}} \geq C(1+\alpha)$, then \eqref{eq:dyacNTC} and $\ell(I) = \ell(Q)$ implies that $I\in \WW_{Q}^0$. By the definition \eqref{def:dyaNTC} it follows that $X\in \Gamma_d(q)$.
In particular, since $\Gamma^\alpha(q)$ is open, we also have $\Gamma^{\alpha}(q) \subset \interior \Gamma_d(q)$.
On the other hand, suppose $X\in \Gamma_d(q)$, by definition \eqref{def:dyaNTC} $X$ is contained in some $I^*= (1+\theta)I$ for a Whitney box $I\in \WW_Q$ and dyadic cube $Q$ containing $q$. Then by \eqref{def:WQ},
\[ |X-q| \leq \diam I^* + \dist(I,Q) + \diam Q \leq C(K,\theta) \ell(Q), \]
\[ \delta(X) \sim \ell(I) \geq C(\eta) \ell(Q). \]
Therefore there exists $\alpha_1$ sufficiently large, depending on the values of $\eta, K, \theta$, such that
\[ |X-q| < (1+\alpha_1)\delta(X), \]
i.e. $X\in \Gamma^{\alpha_1}(q)$.
 We summarize that now we have
\begin{equation}\label{dntc}
	\Gamma^{\alpha}(q) \subset \interior \Gamma_d(q) \subset  \Gamma_d(q) \subset \Gamma^{\alpha_1}(q), \quad \text{ for all }q\in\pO.
\end{equation}
Clearly $\alpha_1>\alpha$. Moreover, there exists $\beta>\alpha_1$ depending on $\eta, K, \theta$ such that the fattened dyadic non-tangential cone
\begin{equation}\label{fdntc}
	\widehat{\Gamma}_d(q) \subset \Gamma^{\beta}(q) \quad \text{ for all }q\in\pO.
\end{equation}
From now on we fix the values of $ \eta, K, \theta$ and $\beta> \alpha_1>\alpha>0$.

Let $F= Q\setminus \bigcup_{Q_j\in\FF} Q_j$ and suppose it is not empty. We claim that
\begin{equation}\label{conesawtooth}
	\interior \left(\bigcup_{q\in F} \Gamma^Q_d(q) \right) \subset \Omega_{\FF, Q} \subset \overline{\Omega_{\FF, Q}} \subset \Omega_{\FF,Q}^{***} \subset \bigcup_{q\in F} \widehat{\Gamma}^Q_d(q).
\end{equation}
In fact, for any $q\in F$, it is clear that $q$ is in some $Q'\in \DD_{\FF,Q}$; and by \eqref{def:OFQ}, the definition of $\Omega_{\FF,Q}$, we have the first incluement. On the other hand any $X\in \Omega_{\FF, Q}^{***}$ belongs to some $U_{Q'}^{***}$ with $Q'\in \DD_{\FF,Q}$, and thus $X\in \widehat{\Gamma}_d^Q(q)$ for arbitrary $q\in Q'$. By the definition of $\DD_{\FF,Q}$, we know $Q'\cap F\neq \emptyset$, so by taking $q\in Q'\cap F$ we get $X\in \bigcup_{q\in F} \widehat{\Gamma}^Q_d(q)$.

For $N$ sufficiently large, we augment the collection of maximal cubes $\FF$ by adding all dyadic cubes in $\mathbb{D}$ of size smaller than or equal to $2^{-N}\ell(Q)$, and we denote by $\FF^N$ a collection consisting of all maximal cubes of the above augmented collection. In particular $Q'\in \DD_{\FF^N,Q}$ if and only if $Q'\in\DD_{\FF,Q}$ and $\ell(Q') > 2^{-N}\ell(Q)$. By doing this we guarantee that the sawtooth domain $\Omega_{\FF^N,Q}$ is compactly contained in $\Omega$ (roughly speaking $\dist(\Omega_{\FF^N,Q}, \Omega^c) \sim 2^{-N}\ell(Q)$).
Similar to Lemma 4.44 of \cite{HMT} we can construct a smooth cutoff function of $\Omega_{\FF^N,Q}$:

\begin{lemma}[cut-off function of sawtooth domain]\label{lm:stcutoff}
	There exists $\psi_N\in C_0^{\infty}(\RR^n)$ such that
	\begin{enumerate}[(i)]
		\item $\chi_{\Omega_{\FF^N,Q}^{*}} \lesssim \psi_N \leq \chi_{\Omega_{\FF^N,Q}^{**}}$;
		\item $\sup_{X\in\Omega} |\nabla \psi_N (X)|\delta(X) \lesssim 1$;
		\item We abbreviate $\mathcal{W}_{\mathcal{F}^N,Q}  $ as $\mathcal{W}_N$ and set $\Sigma = \partial\Omega_{\FF^N,Q}^*$,
			 \[ \mathcal{W}_N^{\Sigma} = \{I\in\mathcal{W}_N: \text{there exists } J\in \mathcal{W} \setminus \mathcal{W}_N \text{ with } \partial I \cap \partial J \neq \emptyset\}. \] Then
			\begin{equation}\label{eq:intsawtooth}
				\nabla \psi_N \equiv 0 \text{ in } \bigcup_{I\in \WN \setminus \pWN} I^{***}.
			\end{equation}
		\item For each $I\in\WW_{N}$, let $Q_I$ denote a cube in $\DD_{\FF^N,Q}$ such that $I\in \WW_{Q_I}$. Suppose $\omega$ is the harmonic measure with pole $X_0$ and $X_0$ satisfies $\dist(X_0, \Omega_{\FF^N,Q}^{***}) \gtrsim \ell(Q)$. Then
			\begin{equation}\label{hmalmostdisjoint}
				\sum_{I\in\mathcal{W}_N^{\Sigma}} \omega(Q_I) \lesssim \omega(Q)
			\end{equation}
			with a constant depending on $\eta, K, a_0, C_1, d$ and the Ahlfors regular constant of $\Gamma$.
	\end{enumerate}
\end{lemma}
\begin{remark}\label{rmk:eqvhm}
	\begin{enumerate}
		\item We remark that the construction of $\psi_N$ and the proof of its properties (i), (ii), (iii) are higher codimensional analogues of Lemma 4.44 of \cite{HMT}. However we prove (iv) instead of the second estimate in their (4.46), because we will need to prove a good-$\lambda$ inequality for the harmonic measure, instead of the surface measure. Since harmonic measure could have much worse decay properties than the surface measure, not to mention that $\pO$ and $\partial\Omega_{\FF^N,Q}$ are objects of different dimensions, proving (iv) requires a different argument.
		\item Note that in (iv), the choice of $Q_I$ may not be unique. Suppose both $Q_I, \widetilde{Q}_I$ are cubes in $\DD_{\FF^N,Q}$ such that $I\in \WW_{Q_I}$ and $I\in\WW_{\widetilde{Q}_I}$. By the construction of $\WW_{Q}$'s and in particular \eqref{def:WQ}, we know
	\begin{equation}
		\ell(Q_I) \sim \ell(I) \sim \ell(\widetilde{Q}_I), \quad \dist(Q_I, \widetilde{Q}_I) \lesssim \ell(Q_I)
	\end{equation}
	with constants depending on $\eta, K$. Since harmonic measure is doubling, we have
	\begin{equation}
		C_1 \omega(Q_I) \leq \omega(\widetilde{Q}_I) \leq C_2 \omega(Q_I)
	\end{equation}
	with constants only depending on the doubling constant and $\eta, K$. That is to say, for different choices of $Q_I$ the left hand side of \eqref{hmalmostdisjoint} differs at most by a constant multiple. But once we associate a cube $Q_I$ to $I$, the choice will be fixed.
	\end{enumerate}
\end{remark}

\begin{proof}
	The proof of (i) is a modification of the proof from \cite{HMT} in higher codimensions. We recall that given $I$, any closed dyadic cube in $\RR^n$, we set $I^{**} = (1+2\theta)I$ and $I^{***} = (1+4\theta)I$. Let us introduce $\widetilde{I^{**}}= (1+3\theta)I$ so that
	\begin{equation}
		I^{**} \subsetneq \interior \wI \subsetneq \wI \subset \interior I^{***}.
	\end{equation}
	Given $I_0 = [-\frac{1}{2},\frac{1}{2}]^n\subset \RR^n$, we fix $\phi_0\in C_0^\infty(\RR^n)$ such that $\chi_{I_0^{**}} \leq \phi_0\leq \chi_{\widetilde{I_0^{**}}}$ and $|\nabla \phi_0|\lesssim 1$, with the implicit constant depending on $\theta$. For every $I\in\WW$ we set $\phi_I= \phi_0\left((\cdot-X_I)/\ell(I)\right)$ where $X_I$ is the center of $I$, so that $\phi_I\in C_0^\infty(\RR^n)$, $\chi_{I^{**}}\leq \phi_I \leq \chi_{\wI}$ and $|\nabla \phi_I| \lesssim 1/\ell(I)$. Let $\Phi(X):= \sum_{I\in\WW} \phi_I(X)$ for every $X\in\Omega$. Since for each compact subset of $\Omega$, the previous sum has finitely many non-vanishing terms, we have $\Phi\in C^\infty_{loc}(\Omega)$. Also $0 \leq \Phi(X) \lesssim C_{\theta}$ since the family $\{\wI\}_{I\in\WW}$ has bounded overlap. Hence we can set $\Phi_I = \phi_I/\Phi$ and one can easily see that $\Phi_I\in C_0^\infty(\RR^n)$, $C_{\theta}^{-1} \chi_{I^{**}} \leq \Phi_I \leq \chi_{\wI}$ and $|\nabla \Phi_I|\lesssim 1/\ell(I)$. Recall the definition of $\WW_N = \WW_{\FF^N,Q}$ in \eqref{def:WFQ}, we set
	\begin{equation}
		\psi_N(X) = \sum_{I\in\WW_N} \Phi_I(X) = \dfrac{\sum_{I\in\WW_N} \phi_I(X)}{\sum_{I\in\WW} \phi_I(X)}, \quad X\in\Omega.
	\end{equation}
	We first note that the number of terms in the sum defining $\psi_N$ is bounded depending on $N$. Indeed if $Q'\in\DD_{\FF^N,Q}$ then $Q' \in\DD_Q$ and $2^{-N}\ell(Q) < \ell(Q')\leq \ell(Q)$, which implies $\DD_{\FF^N,Q}$ has finite cardinality with bounded depending only the Alhfors regular constant and $N$. Also by construction $\WW_Q$ has cardinality depending only in the allowable parameters $\eta, K$. Hence $\#\WW_N \leq C_N<\infty$. This and the fact that each $\Phi_I\in C_0^\infty(\RR^n)$ yield that $\psi_N\in C_0^\infty(\RR^n)$. Moreover
	\begin{equation}
		\supp\psi_N \subset \bigcup_{I\in\WW_N} \wI = \bigcup_{Q'\in\DD_{\FF^N,Q}}\bigcup_{I\in\WW_Q} \wI \subset \interior \left( \bigcup_{Q'\in\DD_{\FF^N,Q}} U_{Q'}^{**} \right) = \Omega_{\FF^N,Q}^{**}.
	\end{equation}
	This and the definition of $\psi_N$ immediately gives $\psi_N \leq \chi_{\Omega_{\FF^N,Q}^{**}}$. On the other hand if $X\in \Omega_{\FF^N,Q}^{*}$ then there exists $I\in \WW_N$ such that $X\in I^{**}$, in which case $\psi_N(X) \geq \Phi_I(X) \geq C_{\theta}^{-1}$. This completes the proof of (i).
	
	To obtain (ii) we note that for every $X\in\Omega$
	\begin{equation}
		|\nabla \psi_N(X)| \leq \sum_{I\in\WW_N} |\nabla \Phi_I(X)| \lesssim \sum_{I\in \WW} \frac{1}{\ell(I)} \chi_{\wI}(X) \lesssim \frac{1}{\delta(X)},
	\end{equation}
	where we have used that if $X\in\wI$ then $\ell(I) \sim \delta(I)$ and also that the family $\{\wI\}_{I\in\WW}$ has bounded overlap.
	
	Now we turn to (iii). Fix $I\in\WW_N\setminus \WW_N^{\Sigma}$ and $X\in I^{***}$, and set $\WW_X = \{J\in\WW: \phi_J(X) \neq 0\}$. We first note that $\WW_X \subset \WW_N$. Indeed if $\phi_J(X) \neq 0$ then $X\in \widetilde{J^{**}}$. Hence $X\in I^{***}\cap J^{***}$ and our choice of $\theta$ gives that $\partial I$ meets $\partial J$, this in turn implies that $J\in \WW_N$ since $I\notin \WW_N^{\Sigma}$. All these imply
	\begin{equation}
		\psi_N(X) = \dfrac{\sum_{J\in\WW_N} \phi_J(X)}{\sum_{J\in\WW} \phi_J(X)} = \dfrac{\sum_{J\in\WW_N\cap \WW_X} \phi_J(X)}{\sum_{J\in\WW\cap \WW_X} \phi_J(X)} = \dfrac{\sum_{J\in\WW_N\cap \WW_X} \phi_J(X)}{\sum_{J\in\WW_N\cap \WW_X} \phi_J(X)} = 1.
	\end{equation}
	Hence $\psi_N|_{I^{***}} \equiv 1$ for every $I\in\WW_N\setminus\WW_N^{\Sigma}$. This and the bounded overlap of the family $\{I^{***}\}_{I\in\WW_N}$ immediately give that $\nabla \psi_N \equiv 0$ in $\bigcup_{I\in\WW_N\setminus \WW_N^{\Sigma}} I^{***}$.
	
	Finally, it remains to prove the most difficult property, (iv).
	For any $I\in\WW_N^{\Sigma}$, by definition there exists some $J_I \in\WW\setminus \WW_N$ such that $\partial I \cap \partial J_I \neq \emptyset$. Roughly speaking, this is to say that $I$ is a Whitney box living in the ``boundary'' of $\Omega_{\FF^N,Q}^*$.
	 Thus pick any $Q'_I \in \DD$ such that $\WW_{Q'_I}$ contains $J_I$, we know $Q'_I \notin \DD_{\FF^N,Q}$, that is, either $Q'_I \in \DD_{Q_j}$ for some $Q_j\in\FF^N$, or $Q'_I \notin \DD_Q $. We classify $I\in \WW_N^{\Sigma}$ based on which category its associated cube $Q'_I$ lives in: We denote 
	 \[ \Sigma_j = \{I\in\WW_N^{\Sigma}: Q'_I \in \DD_{Q_j}\} \text{ for any } Q_j\in\FF^N,\]
	 and 
	 \[ \Sigma_0 = \{I\in\WW_N^{\Sigma}:  Q'_I \notin \DD_{Q}\}. \] 
	 (Note that for each $I\in\WW_N^{\Sigma}$, we associate it to a unique $Q'_I$, even though the choice itself is not unique.)
	Recall \eqref{ngWbox} we have $\ell(I)\sim\ell(J_I)$. Moreover by the definition of $\WW_Q$ and \eqref{def:WQ},
	\begin{equation}\label{lengthQ}
		\ell(Q'_I) \sim \ell(J_I) \sim \ell(I) \sim \ell(Q_I)
	\end{equation}
	and
	\begin{equation}\label{distQ}
		\dist(Q_I,Q'_I) \leq \dist(Q_I, I) + \dist(I,J_I) + \dist(J_I,Q'_I) \lesssim \ell(Q_I) + \ell(Q'_I) \lesssim \ell(Q'_I).
	\end{equation}
	By similar argument as in remark \ref{rmk:eqvhm} (2) and the doubling property of harmonic measure, we have $\omega(Q_I) \sim \omega(Q'_I)$ for any $I\in\WW_N^{\Sigma}$, with a uniform constant depending on $\eta, K$. Therefore to prove \eqref{hmalmostdisjoint} it suffices to show
	\[ \sum_{I\in\mathcal{W}_N^{\Sigma}} \omega(Q'_I) \lesssim \omega(Q). \]
	
	 We claim that for any $Q_j\in \FF^N$,
	\begin{equation} 
		\sum_{I\in \Sigma_j } \omega(Q'_I) \lesssim \omega(Q_j).
	\end{equation}
	Recall that all such $Q'_I$'s live in $\DD_{Q_j}$.
	For each $k\in\mathbb{N}$ we denote $\Sigma_j^k = \{I\in\Sigma_j: \ell(Q'_I) = 2^{-k} \ell(Q_j) \}$.
	Since $Q_I \in \DD_{\FF^N,Q}$, $Q_j \in \FF^N$, we always have $Q_j \cap Q_I = \emptyset$, so by \eqref{distQ}
	\begin{equation}\label{eq:bdQj}
		\dist\left(Q'_I, \left(Q_j\right)^c \right) \leq \dist(Q'_I, Q_I) \lesssim \ell(Q'_I) = 2^{-k}\ell(Q_j).
	\end{equation}
	That is, the smaller $Q'_I$ is, the closer it is to the ``boundary'' of $Q_j$. The $Q'_I$'s of different generations are very far from being disjoint, however we will sum up the $\omega(Q'_I)$'s by swapping them for the harmonic measure of mutually disjoint cubes.
	By \eqref{eq:bdQj}, for $\rho$ sufficiently small there is an integer $k_1= k_1(\rho)$ such that for any integer $k\geq k_1$,
	\begin{equation}\label{bdQj}
		\bigcup_{k'\geq k} \bigcup_{I\in\Sigma_j^{k'}} Q'_I  \subset \left\{q\in Q_j: \dist\left(q, (Q_j)^c \right) \leq \frac{\rho}{2} \ell(Q_j) \right\}.
	\end{equation} 
	In fact by choosing $k_1$ slightly bigger, we can even guarantee that for any integer $k\geq k_1$,
	\begin{equation}\label{allfuturegnr}
		\bigcup_{k'\geq k} \bigcup_{I\in\Sigma_j^{k'}} Q'_I \subset \bigcup_{i\in\mathcal{I}_k} Q_j^i  \subset \left\{q\in Q_j: \dist\left(q, (Q_j)^c \right) \leq \frac{\rho}{2} \ell(Q_j) \right\},%\tag{boundary-$k_0$}
	\end{equation}
	where $\{Q_j^i\}_{i\in\mathcal{I}_k}$ is the collection of all dyadic cubes in $\DD_{Q_j}$ of length $2^{-k}\ell(Q_j)$ such that $Q_j^i \subset \{q\in Q_j: \dist(q,(Q_j)^c) \leq \rho\ell(Q_j)/2 \}$. By Lemma \ref{lm:dcAR} (v) (vi) the index set $\mathcal{I}_k$ has finite cardinality and $\# \mathcal{I}_k \leq C2^{kd}$. (A priori the set $\mathcal{I}_k$ could be empty, in which case \eqref{allfuturegnr} just means there is no $Q'_I$ corresponding to any $I\in \bigcup_{k'\geq k} \Sigma_j^{k'} $. This case is easy to deal with.)
	
	On the other hand by Lemma \ref{lm:dcAR}, as long as we fix $\rho \in (0,a_0)$ satisfying $A_1 \rho^{\gamma}<1$, the set $\left\{q\in Q_j: \dist(q,(Q_j)^c) > \frac{\rho}{2}\ell(Q_j) \right\} $ is not empty; moreover, there is an integer $k_2$ sufficiently large such that for each $k\geq k_2$ we can find a cube $\widehat{Q}_j$ such that $\ell(\widehat{Q}_j ) = 2^{-k} \ell(Q_j)$ and
	\begin{equation}\label{centerQj}
		\widehat{Q}_j \subset \left\{q\in Q_j: \dist(q,(Q_j)^c) > \frac{\rho}{2}\ell(Q_j) \right\}.
	\end{equation}
	We may think of $\widehat{Q}_j$ as sitting in the ``center'' of $Q_j$, and all $Q'_I$'s in a $\rho/2$-boundary layer of $Q_j$. Let $k_0 = \max\{k_1, k_2\}$, and let $N_1$ denote the (maximal) number of $Q'_I$'s with $\ell(Q'_I) = 2^{-k_0} \ell(Q_j)$. By \eqref{bdQj} and Lemma \ref{lm:dcAR} (vi), $N_1$ is uniformly bounded by a constant depending on $a_0, A_1, \rho, k_0$ and $d$.  Moreover by the doubling property of $\omega$, each such $Q'_I$ satisfies 
	\begin{equation}
		\omega(Q'_I) \leq \omega(Q_j) \leq C(k_0)\omega(\widehat{Q}_j),
	\end{equation}
	with the constant $C(k_0)$ depending on $k_0$ as well as the doubling constant of $\omega$.
	Recall that for each $Q'_I$, the number of all possible $I$'s corresponding to it is uniformly bounded by $C(N_0)$. 
	Therefore
	\begin{equation}\label{levelk0}
		\sum_{I\in\Sigma_j^{k_0}} \omega(Q'_I) \leq C(N_0) \sum_{Q'_I: \ell(Q'_I) = 2^{-k_0} \ell(Q_j)} \omega(Q'_I) \leq  C(N_0) N_1 C(k_0) \omega(\widehat{Q}_j).
	\end{equation}
	Now for any $I\in \Sigma_j^{k}$ with $k=1,\cdots, k_0-1$, again by the doubling property of harmonic measure we have $\omega(Q'_I) \leq C(k_0) \omega(\widehat{Q}_j)$. By Lemma \ref{lm:dcAR} (iv) (v), the total number of $Q'_I$'s in $\DD_{Q_j}$ such that $\ell(Q'_I)=2^{-k}\ell(Q_j) $ with $k=1,\cdots, k_0-1$ is uniformly bounded by a constant depending only on $k_0, a_0, C_1, d$ and the Ahlfors regular constant of $\Gamma$. Thus the total number of $I$'s in $\Sigma_j^k$ with $k=1, \cdots, k_0-1$ is also uniformly bounded. Therefore combining with  \eqref{levelk0}, we get
	\begin{equation}\label{uptok0}
		\sum_{k=1}^{k_0} \sum_{I\in\Sigma_j^{k}} \omega(Q'_I) \lesssim \omega(\widehat{Q}_j).\tag{estimate-$k_0$}
	\end{equation}
	
	For future generations, we recall \eqref{allfuturegnr}, which says all the $Q'_I$'s corresponding to some $I\in \Sigma_j^{k}$ with $k\geq k_0$ are contained in $\bigcup_{i\in\mathcal{I}_{k_0}} Q_j^i$. 
	The following proof is illustrated in the (idealized) Figure \ref{fig}, where each label denotes the cube near it enclosed or shaded by the same color. 
	\begin{figure}[h]
		\centering
		\includegraphics[scale=0.8]{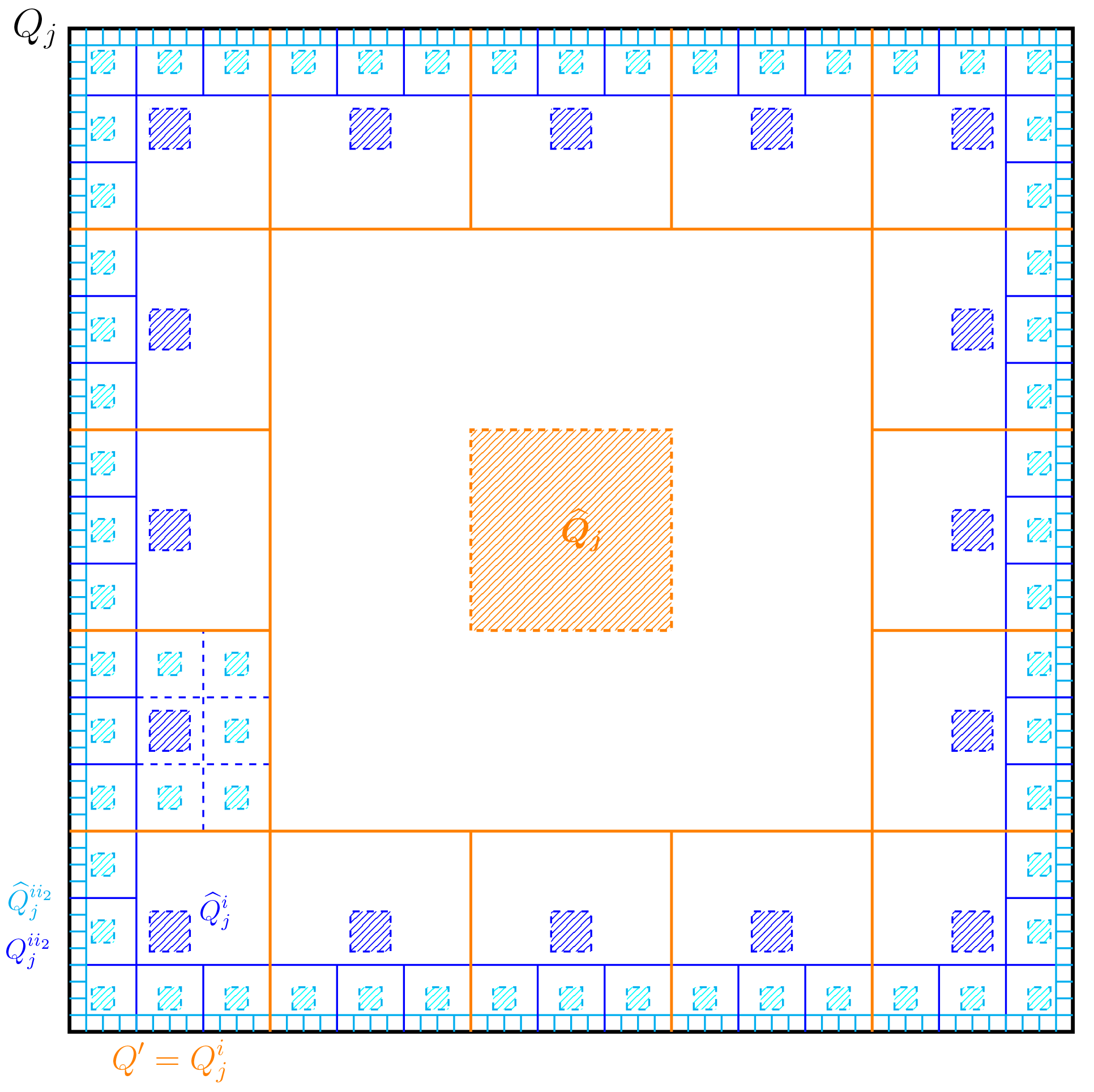}
		%\includegraphics[trim=0 6.5cm 0 0, clip, scale=0.75]{graph.eps}
		%trim = left bottom right top
		\caption{Illustration of the swap of cubes in iteration}
		\label{fig}
	\end{figure}
	Consider any cube $Q'= Q_j^{i}$ for an arbitrary $i\in\mathcal{I}_{k_0}$.
	Apply the above argument to $Q'$ in place of $Q_j$, we can find a cube $\widehat{Q}' = \widehat{Q}^i_j \in\DD_{Q'}$ with length $\ell(\widehat{Q}') = 2^{-k_0} \ell(Q')= 2^{-2k_0}\ell(Q_j)$ sitting in the ``center'' of $Q'$, in the sense that
	\begin{equation}
		%\widehat{Q}' = 
		\widehat{Q}^i_j \subset \left\{q\in Q': \dist(q,(Q')^c) > \frac{\rho}{2} \ell(Q') \right\};
	\end{equation}
	and all future generations satisfy
	\begin{equation}
		\bigcup_{k\geq 2 k_0} \bigcup_{I\in\Sigma_j^k \atop{Q'_I \in \DD_{Q'}}} Q'_I \subset \bigcup_{i_2\in\mathcal{I}_{k_0} }Q_j^{i i_2 } \subset \left\{q\in Q': \dist(q, (Q')^c) \leq \frac{\rho}{2} \ell(Q') \right\},%\tag{boundary-$2k_0$}
	\end{equation}
	where $\{Q_j^{i i_2}\}_{i_2\in\mathcal{I}_{k_0}}$ is the collection of all dyadic cubes of length $2^{-k_0} \ell(Q') = 2^{-2k_0} \ell(Q_j)$ that is completely contained in $\{q\in Q'=Q_j^i: \dist(q, (Q')^c) \leq \rho\ell(Q')/2\}$. (The index set for $i_2$ may not be the same as the index set for $i$, but their cardinalities are uniformly bounded by $C2^{k_0 d}$, so we abuse the notation here and simply assume they are the same.)
	Moreover we can get an analogous estimate of \eqref{uptok0}:
	\begin{equation}\label{analog}
		\sum_{k=k_0+1 }^{2k_0} \sum_{I\in\Sigma_j^k \atop{ Q'_I \in \DD_{Q'}}} \omega(Q'_I) \lesssim \omega(\widehat{Q}^i_j). 
		%\omega(\widehat{Q}').
	\end{equation}
	Summing up \eqref{analog} over all cubes $Q' \in \{Q_j^i\}_{i\in\mathcal{I}_{k_0}}$, recall \eqref{allfuturegnr} we get
	\begin{equation}
		\sum_{k=k_0+1 }^{2k_0} \sum_{I\in\Sigma_j^k } \omega(Q'_I) \lesssim \sum_{i\in\mathcal{I}_{k_0}} \omega(\widehat{Q}_j^i).
	\end{equation}
	%where for each $i$ we use $\widehat{Q}_j^i$ to denote $\widehat{Q}'$, the ``center'' of $Q' = Q_j^i$. 
	Since $\{Q_j^i\}_{i\in\mathcal{I}_{k_0}}$ is a collection of cubes in the same generation, they are mutually disjoint, and their sub-cubes $\{\widehat{Q}_j^i\}_{i\in\mathcal{I}_{k_0}}$ are also mutually disjoint.
	Hence
	\begin{equation}\label{k0to2k0}
		\sum_{k=k_0+1 }^{2k_0} \sum_{I\in\Sigma_j^k } \omega(Q'_I) \lesssim \sum_{i\in\mathcal{I}_{k_0}} \omega(\widehat{Q}_j^i) = \omega\left(  \bigsqcup_{i\in\mathcal{I}_{k_0}} \widehat{Q}_j^i \right) .\tag{estimate-$2k_0$}
	\end{equation}
	Moreover, recall the second inclument of \eqref{allfuturegnr} and \eqref{centerQj}, each $\widehat{Q}_j^i$ is disjoint from $\widehat{Q}_j$, so we can add up \eqref{uptok0} and \eqref{k0to2k0} with ease. 
	We can repeat this argument iteratively: for any $l\in\mathbb{N}$ we apply the argument to cube $Q'=Q_j^{i_1 i_2 \cdots i_{l} }$ with $i_1,\cdots,i_l\in\mathcal{I}_{k_0} $ to get an analogous estimate of \eqref{analog}, then we sum up over the index sets and get
	\begin{equation}\label{sumforl}
		\sum_{k=lk_0+1}^{(l+1)k_0} \sum_{I\in\Sigma_j^k} \omega(Q'_I) \lesssim \sum_{i_1, \cdots,i_l \in \mathcal{I}_{k_0}} \omega\left(\widehat{Q}_j^{i_1\cdots i_l}\right) = \omega\left( \bigsqcup_{i_1,\cdots,i_l\in\mathcal{I}_{k_0}} \widehat{Q}_j^{i_1\cdots i_l} \right).\tag{estimate-$(l+1)k_0$}
	\end{equation}
	Most significantly for us, for each $l\in\mathbb{N}$ the union of cubes on the right hand side of \eqref{sumforl} is disjoint from all the cubes from all previous summations. Therefore we conclude that
	\begin{equation}
		\sum_{k=1}^{\infty} \sum_{I\in\Sigma_j^k} \omega(Q'_I) \lesssim  \omega\left(\bigsqcup_{l\in\mathbb{N}} \left( \bigsqcup_{i_1,\cdots,i_l\in\mathcal{I}_{k_0}} \widehat{Q}_j^{i_1\cdots i_l} \right) \right) \leq \omega(Q_j).
	\end{equation}
	It is trivial to see $\sum_{I\in\Sigma_j^0} \omega(Q'_I) \lesssim \omega(Q_j)$, so
	\begin{equation}
		\sum_{I\in \Sigma_j} \omega(Q'_I) =  \sum_{k\in\mathbb{N}} \sum_{I\in\Sigma_j^k} \omega(Q'_I) \lesssim  \omega(Q_j).
	\end{equation}
	Since the maximal cubes $Q_j$ in $\mathcal{F}^N$ are mutually disjoint and contained in $Q$, we have
	\begin{equation}\label{Sjall}
		\sum_{Q_j\in \FF^N} \sum_{I\in\Sigma_j} \omega(Q'_I) \lesssim \sum_{Q_j\in\FF^N} \omega(Q_j) \leq \omega(Q).
	\end{equation}
	
	Now we consider $I\in\Sigma_0$, which by definition means $Q'_I \notin \DD_Q$. Recall \eqref{lengthQ} and \eqref{distQ}, and that $\ell(I) \leq C \ell(Q)$ for all $I\in \WW_N = \WW_{\FF^N,Q}$, we have
	\begin{equation}
		\ell(Q'_I) \sim \ell(I) \leq C \ell(Q), \quad \dist(Q_I, Q'_I )\lesssim \ell(Q'_I) \leq C \ell(Q).
	\end{equation}
	In particular since $Q_I\in\DD_Q$, we have
	\begin{equation}\label{distQ0}
		\dist(Q'_I, Q) \leq \dist(Q'_I, Q_I) \lesssim \ell(Q'_I) \leq C \ell(Q).
	\end{equation}
	If $\ell(Q'_I) \geq \ell(Q)$, then 
	\begin{equation}
		\ell(Q'_I) \sim \ell(Q), \quad \dist(Q'_I,Q) \leq C\ell(Q).
	\end{equation}
	There are finitely many such $Q'_I$'s and by the doubling property of harmonic measure $\omega(Q'_I) \sim \omega(Q)$.
	If $\ell(Q'_I) < \ell(Q)$, let $Q_0\in \DD$ be the cube containing $Q'_I$ with length $\ell(Q_0) = \ell(Q)$. By the assumption $Q'_I \notin \DD_Q$, we know $ Q_0 $ is disjoint from $Q$. On the other hand \eqref{distQ0} implies
	\begin{equation}\label{sblngh}
		\dist(Q_0,Q) \leq \dist(Q'_I, Q) \leq C\ell(Q),
	\end{equation}
	that is, $Q_0$ is a sibling (i.e. of the same generation) of $Q$ in a $C\ell(Q)$-neighborhood of $Q$. There are finitely many such $Q_0$'s. Moreover
	\begin{equation}
		\dist(Q'_I, (Q_0)^c) \leq \dist(Q'_I, Q) \lesssim \ell(Q'_I).
	\end{equation}
	So if $\ell(Q'_I) \ll \ell(Q) $, we can guarantee that $Q'_I$ lies in the $\rho/2 $-boundary layer of $Q_0$: $Q'_I \subset \{q\in Q_0: \dist(q,(Q_0)^c) \leq \rho\ell(Q_0)/2\}$. Apply the same argument to $Q_0$ in place of $Q_j$, we get 
	\begin{equation}\label{awayQ0}
		\sum_{I\in\Sigma_0 \atop{Q'_I \in \DD_{Q_0}}} \omega(Q'_I) \lesssim \omega(Q_0) \sim \omega(Q).
	\end{equation}
	Summing up \eqref{awayQ0} over all (finitely many) $Q_0$'s satisfying \eqref{sblngh}, we get
	\begin{equation}\label{S0}
		\sum_{I\in\Sigma_0} \omega(Q'_I) \lesssim \omega(Q).
	\end{equation}
	Finally we combine \eqref{Sjall} and \eqref{S0} and conclude that
	\begin{equation}
		\sum_{I\in\WW_N^{\Sigma}} \omega(Q'_I) = \sum_{Q_j\in \FF^N} \sum_{I\in \Sigma_j} \omega(Q'_I) + \sum_{I\in \Sigma_0} \omega(Q'_I) \lesssim \omega(Q).
	\end{equation}
	Therefore
	\begin{equation}
		\sum_{I\in\WW_N^{\Sigma}} \omega(Q'_I) \lesssim \omega(Q).
	\end{equation}
	
\end{proof}

Now that all the preparatory work has been done, we proceed to sketch the basic idea for the proof of Theorem \ref{thm:SlessthanN}. It is well known in harmonic analysis that the proof of $\|Su\|_{L^p(\sigma)} \leq C \|Nu\|_{L^p(\sigma)} $ can be reduced to the proof of a certain good-$\lambda$ inequality measured by $\sigma$. We first prove Proposition \ref{lm:goodlambdaomega0}, which is a good-$\lambda$ inequality measured by $\omega$; then we use the assumption $\omega \in A_\infty(\sigma)$ to obtain the desired good-$\lambda$ inequality for $\sigma$.

Recall that we use $Su, S'u, S{''}u$ to denote the square functions on standard non-tangential cones of aperture $\alpha,\alpha_1, \beta$, respectively, and $Nu$ non-tangential maximal function on cones of aperture $\beta$, where $\beta>\alpha_1>\alpha$ are fixed apertures (see the discussion before Lemma \ref{lm:stcutoff}).
Also recall from \eqref{stinball} that for any collection $\mathcal{F}$ of dyadic cubes, the sawtooth domain $\Omega_{\mathcal{F},Q} \subset B(x_Q, C_3 \ell(Q)) \cap \Omega$. In fact, by choosing a slightly bigger constant $C_3$ we can also guarantee $\Omega_{\FF, Q}^{***} \subset B(x_Q, C_3 \ell(Q)) \cap \Omega$. 
% good-$\lambda$ for $\omega$ and remark.

\begin{proof}[Proof of Proposition \ref{lm:goodlambdaomega0}]
	For simplicity we denote $\omega = \omega^{X_Q}$.  
	Let $E = \{q\in Q:Su(q)> 2\lambda, Nu(q)\leq \delta\lambda\}$ and $F=\{q\in Q: Nu(q) \leq \delta\lambda\}$. If $F$ is empty, then the left hand side of \eqref{eq:goodlambda} is zero, and there is nothing to prove. So we assume $F\neq \emptyset$.
	Note that $Nu(q)$ is a continuous function, so $Q\setminus F=\{q\in Q: Nu(q) >\delta\lambda\}$ is relatively open in $Q$. We run a stopping time procedure for the descendants of $Q$, and stop at $Q' \in \DD_Q$ whenever $Nu(q) > \delta\lambda$ for all $q\in Q'$. We denote the collection of all maximal cubes by $\FF_2 = \{Q_j\}\subset \DD_Q \setminus \{Q\}$. 
	We claim that they form a partition:
\begin{equation}\label{eq:nm}
	Q\setminus F = \{q\in Q: Nu(q) > \delta\lambda\} = \bigcup_{Q_j\in \FF_2} Q_j. 
\end{equation} 
Clearly by construction $\cup_{Q_j\in \FF_2} Q_j$ is contained in the set on the left. For any $q_0 \in Q$ such that $Nu(q_0)>\delta\lambda$, since the set $\{q\in\pO: Nu(q)> \delta\lambda\}$ is open, $Q\setminus F \neq Q$ and the cubes in $\DD$ are nested, there exists a small cube $Q'\in \DD_Q \setminus \{Q\}$ containing $q_0$, such that $Nu(q)>\delta\lambda$ for all $q\in Q'$. By the stopping time procedure, either $Q'\in \FF_2$, or $Q'$ is contained in some cube $Q_j\in \FF_2$. Hence $q_0 \in Q' \subset \bigcup_{Q_j\in\FF_2} Q_j$, and we prove the claim \eqref{eq:nm}.
	Recall \eqref{conesawtooth}, which we rewrite here:
	\begin{equation}\label{conesawtooth2}
		\interior \left( \bigcup_{q\in F} \Gamma^Q_d(q) \right) \subset \Omega_{\FF_2,Q} \subset \overline{\Omega_{\FF_2,Q}} \subset \Omega_{\FF_2,Q}^{***} \subset \bigcup_{q\in F} \widehat{\Gamma}^Q_d(q).
	\end{equation}
	We claim that $|u(X)|\leq \delta\lambda$ for all $X\in \Omega_{\FF_2,Q}^{***}$. In fact, by \eqref{fdntc} and \eqref{conesawtooth2} we know that every $X\in \Omega_{\FF_2,Q}^{***}$ is contained in some $ \widehat{\Gamma}^Q_d(q) \subset \Gamma^{\beta}(q)$ for some $q\in F$. 
	Since $Nu(q) = \sup_{X\in \Gamma^{\beta}(q)} |u(X)| \leq \delta\lambda$ for $q\in F$, we get $|u(X)| \leq \delta\lambda$.

	\textbf{Step 1.}
	Recall the assumption that $S^{'}u(q_1) \leq\lambda$ for some $q_1$ satisfying $|q_1 - q| \leq C_2\diam Q$ for all $q\in Q$. Denote $r=\diam Q$. We claim that for any $\tau>0$ there exists $\delta>0$ sufficiently small such that the truncated square function $S^{}_{\tau r}u(q)>\lambda$ for any $q\in E$. 
	
	Fix $q\in E$.
	Recall that $Su(q) > 2\lambda$ for $q\in E$. We denote $U = \Gamma^{\alpha}(q) \setminus B(q,\tau r)$, then we aim to show
	\begin{equation}\label{eq:tmpSub}
		\iint_{U} |\nabla u(X)|^2 \delta(X)^{1-d} dm(X) \leq 3\lambda^2.
	\end{equation}
	 Let $U_1 = \Gamma^{\alpha}(q) \setminus B(q,t r) $ for a constant $t>\tau$ to be chosen later, and $U_2 = \Gamma^{\alpha}(q) \cap \left( B(q,t r) \setminus B(q,\tau r) \right)$. Then $U= U_1 \cup U_2$.
	A simple computation shows that
	\begin{equation}
		U_1 = \Gamma^{\alpha}(q) \setminus B(q,t r) \subset \Gamma^{\alpha_1}(q_1)
	\end{equation}
	if the apertures satisfy
	\[		(1+\alpha) \left( 1+ \frac{C_2}{t} \right) \leq 1+\alpha_1, \]
	that is, if $t$ is sufficiently large such that
	\[ \alpha + \frac{C_2 \left(1+\alpha\right)}{t} \leq \alpha_1. \]
	Therefore
	\begin{equation}
		\iint_{U_1} |\nabla u(X)|^2 \delta(X)^{1-d} dm(X) \leq \iint_{\Gamma^{\alpha_1}(q_1)} |\nabla u(X)|^2 \delta(X)^{1-d} dm(X) = S'u(q_1)^2 \leq \lambda^2.
	\end{equation}
	Let $\Gamma_j(q) = \Gamma^{\alpha}(q) \cap \left( B(q, 2^j \tau r) \setminus B(q,2^{j-1} \tau r) \right)$ for $j=1,2,\cdots$, then 
	\[ U_2 \subset \bigcup_{j:2^{j-1} \tau r < tr} \Gamma_j(q). \]
	Each $\Gamma_j(q)$ can be covered by a finite union (depending on $n$) of balls $B_{j,k}$ with radius $r_{j,k} \sim_{\alpha} 2^j \tau r$. Let $B^*_{j,k}$ denote a slight fattening of $B_{j,k}$ such that we still have $B^*_{j,k} \subset \Gamma^{\beta}(q)$, then by Lemma \ref{lm:weight} (i) $m(B^*_{j,k}) \sim r_{j,k}^{d+1} \sim (2^j \tau r)^{d+1} $. Then
	\begin{align}
		\iint_{U_2} |\nabla u(X)|^2 \delta(X)^{1-d} dm(X) & = \sum_{2^{j-1} \tau r< tr} \iint_{\Gamma_j(q)} |\nabla u(X)|^2 \delta(X)^{1-d} dm(X) \nonumber \\
		& \sim_{\alpha, \beta} \sum_{2^{j-1} \tau r< tr } (2^j \tau r)^{1-d} \sum_{1\leq k\leq C(n)} \iint_{B_{j,k}} |\nabla u(X)|^2  dm(X) \nonumber \\
		& \lesssim \sum_{2^{j-1} \tau r< tr \atop{ 1\leq k\leq C(n)} } (2^j \tau r)^{-1-d}  \iint_{B_{j,k}^*} |u(X)|^2  dm(X) \nonumber \\
		& \lesssim \left(\delta\lambda\right)^2 \sum_{2^{j-1} \tau r< tr} (2^j \tau r)^{-1-d} m(B^*_{j,k}) \nonumber \\
		& \lesssim \left(\delta\lambda\right)^2 \log_2\left( \frac{t}{\tau} \right) \nonumber \\ 
		& < 2\lambda^2,
	\end{align}
	if $\delta$ is sufficiently small depending on the values of $t, \tau$ and $\alpha, \beta$. Therefore \eqref{eq:tmpSub} holds, and thus for any $q\in E$,
	\begin{align}
		|S_{\tau r} u(q)|^2 & =\iint_{\Gamma^{\alpha}(q) \cap B(q,\tau r) } |\nabla u(X)|^2 \delta(X)^{1-d} dm(X) \nonumber \\
		& = \iint_{\Gamma^{\alpha}(q) \setminus U} |\nabla u(X)|^2 \delta(X)^{1-d} dm(X) \nonumber \\
		& > \lambda^2.\label{eq:lbtruncatedsf}
	\end{align}
	
	\textbf{Step 2.}
	Combining \eqref{eq:lbtruncatedsf} with $E\subset F$ we get
	\begin{align}
		\lambda^2 \omega(E) \leq  \int_{E} |S_{\tau r}u(q)|^2 d\omega(q) & \leq \int_{F} \iint_{\Gamma^{\alpha}_{\tau r}(q)} |\nabla u(X)|^2 \delta(X)^{1-d} dm(X) d\omega(q).\label{eq:left}
	\end{align}
	By \eqref{dntc} we have
	\begin{equation}\label{eq:tmpNTC}
		\Gamma_{\tau r}^{\alpha}(q) \subset \interior \Gamma_d(q) \subset \Gamma_d(q) 
	\end{equation} 
	for any $q\in Q$.
	In particular if $X$ belongs to the left hand side of \eqref{eq:tmpNTC}, then $X\in U_{Q'}$ for some dyadic cube $Q'$ containing $ q$. Moreover 
	\begin{equation}\label{eq:tmptruncated1}
		\delta(X) \leq |X-q| < \tau r = \tau \diam Q \sim \tau \ell(Q).
	\end{equation}
	 By the definition of $U_{Q'}$ and \eqref{def:WQ}, we have 
	 \begin{equation}\label{eq:tmptruncated2}
	 	\delta(X) \gtrsim c\eta^{\frac{n-1}{4(n-1-d)} } \ell(Q') .
	 \end{equation} 	
	 By combining \eqref{eq:tmptruncated1}, \eqref{eq:tmptruncated2} and choosing $\tau$ small enough depending on $\eta$, we can guarantee that $\ell(Q') < 2\ell(Q)$. Since $Q'\cap Q\ni q$, by property (ii) of Lemma \ref{lm:dcAR} we know $Q'\in\DD_Q$. Hence $\Gamma_{\tau r}^{\alpha}(q) \subset \Gamma_d^Q(q)$. Again since $\Gamma_{\tau r}^{\alpha}(q)$ is an open set, we also have $\Gamma_{\tau r}^{\alpha}(q) \subset \interior \Gamma_d^Q(q)$.
	  Therefore
	\begin{equation}
		\bigcup_{q\in F} \Gamma^{\alpha}_{\tau r}(q) \subset \bigcup_{q\in F} \left( \interior \Gamma^Q_d(q) \right) \subset \interior \left( \bigcup_{q\in F} \Gamma^Q_d(q) \right)
	\end{equation}
	Applying Fubini's theorem to the right hand side of \eqref{eq:left}, we conclude that it is bounded by
	\begin{equation}\label{eq:leftFubini}
		\iint_{\interior \left(\bigcup_{p\in F} \Gamma^Q_d(p) \right)} |\nabla u(X)|^2 \delta(X)^{1-d} \omega\left(\{q\in F: X\in \Gamma^Q_d(q)\} \right) dm(X).
	\end{equation}
	For any $p\in F$ and any $X\in \Gamma_d^Q(p)$, we have $X\in I \in \mathcal{W}_{Q'}$ for a cube $Q'$ containing $p$ and in $\DD_{\FF_1,Q}$. Thus $|X-q| \sim \ell(Q')\sim \ell(I) \sim \delta(X)$. Since the family $\{I^*\}_{I\in\WW}$ has bounded overlap and harmonic measure $\omega$ has pole at $X_Q$, we conclude by Lemma \ref{lm:CFMS} that
	\begin{equation}\label{eq:leftCFMS}
		\omega\left(\left\{q\in F: X\in \Gamma^Q_d(q)\right\} \right) \sim \omega\left(\bigcup_{Q'\in \DD_Q \atop{\ell(Q')\sim \delta(X)}\sim \dist(X,Q')} Q' \right) \sim G(X_Q,X) \delta(X)^{d-1}.
	\end{equation}
	Combining \eqref{eq:left}, \eqref{eq:leftFubini}, \eqref{eq:leftCFMS} and \eqref{conesawtooth2} and using \eqref{eq:Alb}, we get
	\begin{align}
		\lambda^2 \omega(E) \lesssim \iint_{\Omega_{\FF_2, Q}}  |\nabla u(X)|^2 G(X_Q,X) dm(X) & = \iint_{\Omega_{\FF_2, Q}}  |\nabla u(X)|^2 G(X_Q,X) w(X) dX \nonumber \\
		&  \lesssim \iint_{\Omega_{\FF_2, Q}} A\nabla u \cdot \nabla u G dX.
	\end{align}
	Here we abbreviate $G(X) = G(X_Q,X)$ when there is no ambiguity as to what the pole is. Recall that the pole $X_Q \notin B(x_Q, 2C_3 \ell(Q))$, and similar to \eqref{stinball} we may choose the dilation constant $\theta$ small enough so that $\overline{\Omega_{\FF_2, Q}^{***}} \subset B(x_Q, \frac32 C_3\ell(Q))$. They guarantee that $X_Q \notin \overline{\Omega_{\FF_2, Q}^{***}}$, and moreover $\dist(X_Q, \overline{\Omega_{\FF_2, Q}^{***}}) \gtrsim \ell(Q)$. Hence $G(X)$ is harmonic in the fat sawtooth domain $\Omega_{\FF_2, Q}^{***}$.
	
	\textbf{Step 3.}
	Next we are going to prove
	\begin{equation}\label{eq:estimateIIP}
		\iint_{\Omega_{\FF_2, Q}} A\nabla u \cdot \nabla u G dX \lesssim (\delta\lambda)^2 \omega(Q).
	\end{equation}
	Recall the discussion before Lemma \ref{lm:stcutoff}, we can augment $\FF_2$ by adding all dyadic cubes of lengths smaller or equal to $2^{-N} \ell(Q)$, and denote by $\FF_2^N$ the collection of maximal cubes giving rise to the aforementioned augmented collection. We claim that
	\begin{equation}
		\iint_{\Omega_{\FF_2^N, Q}} A\nabla u \cdot \nabla u G dX \lesssim (\delta\lambda)^2 \omega(Q)
	\end{equation}
	with a constant independent of $N$. Thus by passing $N\to \infty$ we obtain \eqref{eq:estimateIIP}.

	Recall that in Lemma \ref{lm:stcutoff}, we construct a smooth cut-off function $\psi_N$ such that $\chi_{\OFN^{*}} \lesssim \psi_N \leq \chi_{\OFN^{**}}$. Hence
	\begin{equation}
		\iint_{\OFN} A\nabla u \cdot \nabla u G dX \leq  \iint_{\RR^n} A\nabla u \cdot \nabla u G \psi_N dX
	\end{equation} 
	Since $u, G \in W_r(\OFN^{**}) \cap L^{\infty}(\OFN^{**}) $, we have $uG\psi_N, u^2 \psi_N \in W_0^{1,2}(\OFN^{**})$. In particular they can be approximated by smooth functions in $C_0^{\infty}(\OFN^{**}) \subset C_0^{\infty}(\Omega)$. In the sawtooth region $\OFN^{**}$ we have $-\divg(A\nabla u) = -\divg(A\nabla G) = 0$, thus
	%\reversemarginpar
	%\zihui{It seems if we don't assume symmetry, we need to use $G_{L^*}$ and in particular $\omega_{L^*}$ }
\begin{align}
	&  \iint_{\RR^n} A\nabla u \cdot \nabla u G \psi_N dX \nonumber \\
	 = &  \iint_{\RR^n} A\nabla u \cdot \nabla \left(u G \psi_N \right) - \frac{1}{2} A\nabla (u^2) \cdot \nabla (G\psi_N) dX  \nonumber \\
	 = & 0 - \frac{1}{2} \iint_{\RR^n} A \nabla (G\psi_N) \cdot \nabla (u^2) dX \nonumber \\
	 = & - \frac{1}{2} \left( \iint_{\RR^n} \psi_N A\nabla G \cdot \nabla (u^2) +  G A\nabla \psi_N \cdot \nabla(u^2) dX \right) \nonumber \\
	 = & -\frac{1}{2} \left( \iint_{\RR^n} A\nabla G \cdot \nabla (u^2 \psi_N) - u^2 A\nabla G \cdot \nabla \psi_N + 2 uG A \nabla u \cdot \nabla \psi_N dX \right) \nonumber \\
	 = & \frac{1}{2} \iint_{\RR^n} u^2 A\nabla G \cdot \nabla \psi_N dX - \iint_{\RR^n} uG A\nabla u \cdot \nabla \psi_N dX \nonumber \\
	 =: & \frac{1}{2} I - II, \label{eq:right}
\end{align}
where we use the symmetry of $A$ and the equation $-\divg(A\nabla u) = 0$ in the second equality, and $-\divg(A\nabla G) = 0$ in the second to last equality.
We first estimate the second term.
By \eqref{eq:intsawtooth}, the contribution to the integral $II$ only comes from Whitney boxes $I\in\WW_N^{\Sigma}$. Recall the harmonic function $u$ is non-negative and we use $X_I$ to denote the center of Whitney box $I$. By Lemma \ref{lm:stcutoff} (ii), H\"older inequality, estimate of the weight \eqref{eq:weightin}, interior Cacciopoli inequality \eqref{eq:intCcpl}, Harnack inequality \eqref{eq:Hnk} and \eqref{eq:CFMS}, we have
\begin{align}
	|II| & \leq \sum_{I\in \pWN} \frac{u(X_I)G(X_I)}{\ell(I)}   \iint_{I^{***}} |\nabla u| dm \nonumber \\
	& \leq \sum_{I\in \pWN} \frac{u(X_I)G(X_I)}{\ell(I)} \cdot m(I^{***}) \left( \fiint_{I^{***}} |\nabla u|^2 dm \right)^{\frac{1}{2}} \nonumber \\
	& \lesssim \sum_{I\in \pWN} u(X_I)G(X_I) \ell(I)^{d-1} \left( \fiint_{I^{****}} |u|^2 dm \right)^{\frac{1}{2}} \nonumber \\
	& \lesssim \sum_{I\in \pWN} u(X_I)^2 G(X_I) \ell(I)^{d-1} \nonumber \\
	& \sim \sum_{I\in \pWN} u(X_I)^2 \omega(Q_I),
\end{align}
where $Q_I$ is defined as in Lemma \ref{lm:stcutoff} (iv). Using the estimate $|u(X)| \leq \delta\lambda$ for all $X\in\OFN^{***}$ and \eqref{hmalmostdisjoint}, we have
\begin{equation}\label{eq:right2}
	|II| \lesssim \sum_{I\in \pWN} u(X_I)^2 \omega(Q_I) \lesssim \left(\delta\lambda\right)^2 \omega(Q).
\end{equation}
Similarly,
\begin{align}
	|I| \leq \sum_{I\in\mathcal{W}_N^{\Sigma}} \frac{u(X_I)^2}{\ell(I)} \iint_{I^{***}} |\nabla G| dm \lesssim \sum_{I\in \pWN} u(X_I)^2 \omega(Q_I) \lesssim \left(\delta\lambda\right)^2 \omega(Q).\label{eq:right1}
\end{align}
We finish the proof of \eqref{eq:estimateIIP} by combining \eqref{eq:right}, \eqref{eq:right2} and \eqref{eq:right1}.

Finally we combine \eqref{eq:left} and \eqref{eq:estimateIIP}, and get
\begin{align}
	\lambda^2 \omega(E) \lesssim \left(\delta\lambda\right)^2 \omega(Q).
\end{align}
And thus 
\begin{equation}
	\omega(E) \leq C\delta^2 \omega(Q).
\end{equation}
This finishes the proof of the good-$\lambda$ inequality for $\omega$.
\end{proof}

We will also need the following auxiliary fact:  
\begin{lemma}\label{lm:opensf}
	For any apertures $0<\alpha<\alpha'$ and any function $u\in W_r(\Omega)$, let $Su$ and $ \widetilde{S} u$ denote the square function with aperture $\alpha$ and $\alpha'$ respectively.
	Suppose $\widetilde{S} u<\infty$ for $\sigma$-almost every $q\in \pO$, then the set $\{q\in\pO: Su(q) >\lambda\}$ is open for every $\lambda>0$.
\end{lemma} 
\noindent The proof is  similar in spirit to that of Lemma 4.6 in \cite{HACAD}.

\begin{proof}
	If $q\in\pO$ is such that $S'u(q) >\lambda$, then there exists $\eta>0$ so that
\[ \iint_{\Gamma^\alpha(q) \setminus B(q,\eta)} |\nabla u|^2 \delta(X)^{1-d} dm(X) > \left( \frac{Su(q) + \lambda}{2} \right)^2. \]
We claim that there exists $\epsilon>0$ such that for any $p\in \Delta(q,\epsilon \eta)$, we have
\begin{equation}\label{eq:claimsmall}
	\iint_{\Gamma^\alpha(p) \setminus B(p,\eta)} |\nabla u|^2 \delta(X)^{1-d} dm(X) > \lambda^2, 
\end{equation} 
and therefore $Su(p)> \lambda$.

We observe that
\begin{align}
	& \left|\iint_{\Gamma^\alpha(q) \setminus B(q,\eta)} |\nabla u|^2 \delta(X)^{1-d} dm(X)- \iint_{\Gamma^\alpha(p) \setminus B(p,\eta)} |\nabla u|^2 \delta(X)^{1-d} dm(X)\right| \nonumber \\  & \qquad  \qquad \leq \iint_{D} |\nabla u|^2 \delta(X)^{1-d} dm(X),\label{eq:intgsetdiff}
\end{align}
where $D= \left(\Gamma^\alpha(q) \setminus B(q,\eta) \right) \triangle \left( \Gamma^\alpha(p) \setminus B(p,\eta) \right)$ is the set difference. It suffices to show that the integral $\iint_D |\nabla u|^2 \delta(X)^{1-d} dm(X)$ is sufficiently small, if we choose $\epsilon$ sufficiently small.

Suppose that $X\in \Gamma^\alpha(q) \setminus B(q,\eta)$, then $|X-q| < (1+\alpha)\delta(X)$ and $|X-q|\geq \eta$. Thus $\delta(X) > \frac{\eta}{1+\alpha}$. If moreover $X\notin \Gamma^{\alpha}(p) \setminus B(p,\eta)$ and $p\in B(q,\epsilon \eta)$, then $|X-q| > (1+\alpha) (1-\epsilon) \delta(X)$. By symmetry, we need to study sets of the form
\[ V_q = \left\{X\in\Omega: |X-q| \geq \eta,~ (1+\alpha)(1-\epsilon) \delta(X) < |X-q| < (1+\alpha) \delta(X) \right\}, \]
\[ V_p = \left\{X\in\Omega: |X-p| \geq \eta,~ (1+\alpha)(1-\epsilon) \delta(X) < |X-p| < (1+\alpha) \delta(X) \right\}. \]
Without loss of generality we may assume $S'u(q)<\infty$. If not, by the assumption that $S'u<\infty$ almost everywhere, we can always find $q' \in \Delta(q, \epsilon \eta/2)$ such that $S'u(q')<\infty$, and in particular $p\in \Delta(q,\epsilon \eta) \subset \Delta(q',2\epsilon \eta)$. In this case we just replace $q$ by $q'$, and $\epsilon$ by $2\epsilon$. 
Moreover, if $\epsilon < 1/4$, we have that
\[ V_q \cup V_p \subset V_\epsilon := \{X\in\Omega: |X-q| \geq \frac{\eta}{2},~ (1+\alpha)(1-\epsilon)^2\delta(X) < |X-q| < (1+\alpha)\frac{1-\epsilon}{1-2\epsilon} \delta(X) \}. \]
Note that for given $\alpha'>\alpha$, by choosing $\epsilon$ sufficiently small we can guarantee $(1+\alpha) \frac{1-\epsilon}{1-2\epsilon} \leq 1+\alpha'$. Thus $V_\epsilon \subset \Gamma^{\alpha'}(q) \setminus B \left(q, \frac{\eta}{2} \right)=:V_0 $, and as $\epsilon$ tends to zero, the set $V_\epsilon$ decreases to an empty set. Moreover, 
\[ \iint_{V_0} |\nabla u|^2 \delta(X)^{1-d} dm(X) \leq \iint_{\Gamma^{\alpha'}(q)} |\nabla u|^2 \delta(X)^{1-d} dm(X) =|S'u(q)|^2 <\infty, \]
hence by the continuity of measure from above, we deduce that
\[ \iint_{V_\epsilon} |\nabla u|^2 \delta(X)^{1-d} dm(X) \searrow 0. \]
In particular, by choosing $\epsilon$ sufficiently small, we can guarantee
\begin{equation}\label{eq:intgsmall}
	\iint_{D} |\nabla u|^2 \delta(X)^{1-d} dm(X) \leq  \iint_{V_\epsilon} |\nabla u|^2 \delta(X)^{1-d} dm(X) < \left( \frac{Su(q)+\lambda}{2} \right)^2 - \lambda^2
\end{equation}
Combining \eqref{eq:intgsmall} with \eqref{eq:intgsetdiff}, we conclude the proof of the claim \eqref{eq:claimsmall}.
\end{proof}

Now we set out to complete the
\vskip 0.08in
\begin{proof}[Proof of Theorem \ref{thm:SlessthanN}]
%%%%%%%%%%%%%%%%%%%%%%%%%%%%%%%%%%%%%%%%%%%%%%%%%%%%%%%%%%%%%%%%%%%%%%%%%%%
We first prove the theorem assuming that $\|S'u\|_{L^p(\sigma)}$ is finite. 
%%%%%%%%%%%%%%%%%%%%%%%%%%%%%%%%%%%%%%%%%%%%%%%%%%%%%%%%%%%%%%%%%%%%%%%%%%%
Under this assumption, we have that $\|S''u\|_{L^p(\sigma)} \sim \|S'u\|_{L^p(\sigma)} \sim \|Su\|_{L^p(\sigma)}$. For reference, see Proposition 4 of \cite{CMS}. (The stated proof in \cite{CMS} is for the upper half plane, but the argument goes through for Ahlfors regular sets of higher codimension.)
Therefore by a standard argument, the proof of \eqref{eq:SlessthanN} can be reduced to the following good-$\lambda$ inequality: For any $\epsilon>0$ sufficiently small, we can find $\delta= \delta(\epsilon)>0$ such that for all $\lambda>0$,
	\begin{equation}\label{eq:globalgoodlambdas}
		\sigma\left(\left\{ q\in \pO: Su(q) > 2\lambda, Nu(q) \leq \delta\lambda \right\} \right) \leq \epsilon \sigma\left( \left\{q\in \pO: S'u(q) > \lambda \right\} \right),
	\end{equation}
	and $\delta\to 0$ as $\epsilon \to 0$.
If $\{q\in\pO:S'u(q)>\lambda\}$ is empty, \eqref{eq:globalgoodlambdas} is trivial, so we assume the set is not empty. 
We apply Lemma \ref{lm:opensf} with apertures $0<\alpha_1<\beta$. Since $\|S''u\|_{L^p(\sigma)}\sim \|S'u\|_{L^p(\sigma)} <\infty$, in particular $S''u(q)<\infty$ almost everywhere. Therefore $\{q\in\pO: S'u(q)>\lambda\}$ is open. 
We also remark that the set $\{q\in\pO: S'u(q) >\lambda\}$ has finite $\sigma$-measure, and moreover
\begin{equation}
	\sigma\left(\{q\in\pO: S'u(q) >\lambda\}\right) \leq \frac{1}{\lambda^p}\int_{S'u(q) >\lambda} |S'u|^p d\sigma \leq \frac{\|S'u\|_{L^p(\sigma)}^p}{\lambda^p} <\infty.
\end{equation}
In particular, for any dyadic cube $Q\in\DD$ completely contained in $\{q\in\pO: S'u(q) >\lambda\}$ 
\begin{equation}\label{eq:SNtmp1}
	\ell(Q)^{d} \sim \sigma(Q) \leq \sigma\left(\{q\in\pO: S'u(q) >\lambda\}\right) \leq \frac{\|S'u\|_{L^p(\sigma)}}{\lambda^p},
\end{equation}
so its length has a uniform upper bound (albeit depending on the value of $\lambda$). Recall that $\ell(Q) \sim 2^{-k(Q)}$, and suppose $k_0 \in\mathbb{Z}$ is such that
\begin{equation}\label{eq:SNtmp2}
	2^{-k_0 d} \gtrsim \frac{\|S'u\|_{L^p(\sigma)}}{\lambda^p},
\end{equation} 
with a sufficiently large implicit constant.
Then by \eqref{eq:SNtmp1}, any cube $Q_0$ in $ \DD_{k_0}$ can not be completely contained in $\{q\in\pO: S'u(q) >\lambda\}$.

We run a stopping time procedure as follows: For each $Q_0 \in \DD_{k_0}$, we traverse all its descendants, and stop whenever we find a cube $Q\in \DD_{Q_0}$ such that $S'u(q) > \lambda$ for all $q\in Q$. Let $\FF_1=\{Q_l\}$ be the collection of all stopping cubes in $\bigcup_{Q_0 \in \DD_{k_0}} \DD_{Q_0}$. Similar to the proof of \eqref{eq:nm}, we can show that they form a partition:
\begin{equation}\label{eq:st}
	\{q\in \pO: S'u(q) >\lambda \} = \bigcup_{Q_l \in \FF_1} Q_l. 
\end{equation} 
Note that the assumption $Su(q)>2\lambda$ clearly implies $S'u(q)>\lambda$, namely 
\[ \{q\in\pO: Su(q)>2\lambda\} \subset \{q\in \pO: S'u(q)>\lambda \} = \bigcup_{Q_l\in \FF_1} Q_l. \] 
Therefore to prove \eqref{eq:globalgoodlambdas}, it suffices to localize and show that
\begin{equation}\label{eq:goodlambdas}
	\sigma\left(\left\{ q\in Q : Su(q) > 2\lambda, Nu(q) \leq \delta\lambda \right\} \right) \leq \epsilon \sigma\left( Q \right) \quad \text{ for any } Q=Q_l\in\FF_1.
\end{equation}

Recall that by \eqref{nesteddcube}, every $Q\in\DD$ is contained in a surface ball $\Delta(x_Q,C_2 r_Q)$.
Let $X'_Q$ denote a corkscrew point for $B(x_Q, C_2 r_Q)$. Recall Definition \ref{def:Ainfty} of $\omega\in A_{\infty}(\sigma)$ and Remark \ref{rmk:defAinfty} (ii) right afterwards. Assuming $\omega \in A_\infty(\sigma)$, then to prove \eqref{eq:goodlambdas} it suffices to show
\begin{equation}\label{eq:goodlambdatmp}
	\omega^{X'_Q}\left(\{q\in Q: Su(q)>2\lambda, Nu(q) \leq \delta\lambda\} \right) \leq C(\delta) \omega^{X'_Q}(Q),
\end{equation}
with a constant $C(\delta)$ independent of $Q$ and $ \lambda$, and that $C(\delta)\to 0$ as $\delta\to 0$.
Recall that for any collection $\mathcal{F}$ of dyadic cubes, there is a constant $C_3$ such that $\Omega_{\FF, Q}^{***} \subset B(x_Q, C_3 \ell(Q)) \cap \Omega$. Let $X_Q$ be a corkscrew point for $B(x_Q, 2C_3 M\ell(Q))$, then
\begin{equation}
	|X_Q - x_Q| \geq \delta(X_Q) \geq 2C_3 \ell(Q).
\end{equation}
Thus $X_Q \notin B(x_Q, 2C_3 \ell(Q))$, and in particular $X_Q \notin \overline{\Omega_{\FF,Q}^{***}}$.
  Moreover, there is a Harnack chain of finite length (depending only on $M, C_2$ and $C_3$) connecting $X_Q$ to $X'_Q$; in particular the harmonic measures $\omega^{X_Q}(E) \sim \omega^{X'_Q}(E)$ for any Borel set $E\subset Q$. Therefore the proof of \eqref{eq:goodlambdatmp} is equivalent to the proof of
\begin{equation}\label{eq:goodlambdatmp2}
		\omega^{X_Q}\left(\left\{ q\in Q: Su(q) > 2\lambda, Nu(q) \leq \delta\lambda \right\} \right) \leq C(\delta) \omega^{X_Q}\left( Q \right).
\end{equation}

Recall that $Q=Q_l\in\FF_1$ is a maximal cube with respect to the stopping criterion $\{S'u(q)>\lambda\}$. By maximality the parent of $Q$, denoted by $\widetilde Q$, contains at least one point $q_1 \notin \{q \in\pO: S'u(q)>\lambda\}$, that is, $S'u(q_1)\leq \lambda$. For any $q\in Q$ we have
\begin{equation}
	|q_1-q| \leq \diam \widetilde Q \leq A_1 2^{-k(\widetilde Q)} = A_1 2^{-(k(Q)-1)}\leq \frac{A_1}{a_0} \diam Q.
\end{equation} 
Therefore for any maximal cube, we may use Lemma \ref{lm:goodlambdaomega}, with constant $C_2 = A_1/a_0$, to conclude the desired estimate \eqref{eq:goodlambdatmp2}.

\medskip

All the above arguments show that if we know a priori $\|S'u\|_{L^p(\sigma)}$ is finite, we can prove $\|Su\|_{L^p(\sigma)} \lesssim \|Nu\|_{L^p(\sigma)}$. If we do not have this a priori information, then for $\kappa$ sufficiently small we let
\begin{equation}
	\DD_\kappa = \left\{Q\in\DD: \kappa \leq \ell(Q) \leq 1/\kappa \right\},
\end{equation}
\begin{equation}
	\Omega_\kappa = \bigcup_{Q\in\DD_\kappa} U_Q, \quad \Omega^*_{\kappa} = \bigcup_{Q\in\DD_\kappa} U_Q^*, \quad \Omega^{**}_{\kappa} = \bigcup_{Q\in\DD_\kappa} U_Q^{**}\quad \text{etc.}
\end{equation}  
and define the $\kappa$-approximate non-tangential cones as 
\[ \Gamma^{\alpha}_\kappa(q) = \Gamma^{\alpha}(q) \cap \Omega_\kappa, \quad \Gamma_\kappa^{\alpha_1} = \Gamma^{\alpha_1}(q) \cap \Omega_\kappa, \quad \Gamma^\beta_\kappa(q) = \Gamma^\beta(q) \cap \Omega_{\kappa}^{***}, \] 
define the $\kappa$-approximate \textit{dyadic} non-tangential cones as
\[ \Gamma_{d,\kappa}(q) = \Gamma_d(q) \cap \Omega_\kappa = \bigcup_{Q\in\DD^\kappa: Q\ni q} U_{Q}, \quad \widehat{\Gamma}_{d,\kappa}(q) = \widehat{\Gamma}_d(q) \cap \Omega^{***}_\kappa. \]
In this regime we have the following inclusions analogous to \eqref{dntc} and \eqref{fdntc}:
\begin{equation}
	\Gamma^{\alpha}_\kappa(q) \subset \Gamma_{d,\kappa}(q) \subset \Gamma^{\alpha_1}_{\kappa}(q), \quad \widehat{\Gamma}_{d,\kappa}(q) \subset \Gamma^\beta_\kappa(q).
\end{equation}
Moreover, the $\kappa$-approximate local non-tangential cones 
\[ 
	\Gamma^Q_{d,\kappa}(q) = \Gamma_d^Q(q) \cap \Omega_\kappa =  \bigcup_{Q' \in \DD_Q \cap \DD^\kappa: Q'\ni q} U_{Q'}, \quad \widehat{\Gamma}^Q_{d,\kappa}(q) = \widehat{\Gamma}_d^Q \cap \Omega_\kappa^{***} 
\]
satisfy the following inclusions analogous to \eqref{conesawtooth}:
\[ \bigcup_{q\in F} \Gamma^Q_{d,\kappa}(q) \subset \Omega_{\FF,Q} \cap \Omega_\kappa \subset \overline{\Omega_{\FF,Q} \cap \Omega_\kappa} \subset \Omega_{\FF,Q}^{***} \cap \Omega_\kappa^{***} \subset \bigcup_{q\in F} \widehat{\Gamma}_{d,\kappa}^Q(q), \]
for any dyadic cube $Q$ and collection of maximal cubes $\Gamma \subset \DD_Q \setminus \{Q\}$, under the assumption that $F= Q \setminus \bigcup_{Q_j\in\FF} Q_j$ is not empty.
We then define the $\kappa$-approximate square functions $S_\kappa u, S'_\kappa u$ and non-tangential maximal function $N_\kappa u$ accordingly, as integrals defined on the $\kappa$-approximate non-tangential cones instead of standard non-tangential cones.
Since $N_\kappa u(q) \leq Nu(q)$ for all $q\in\pO$, we have $\|N_\kappa u\|_{L^p(\sigma)} \leq \|Nu\|_{L^p(\sigma)} <\infty$. By the interior Cacciopoli inequality \eqref{eq:intCcpl} and that $\beta>\alpha_1>\alpha$, we have
\[ S_\kappa u(q) \leq S'_\kappa u(q) \lesssim C(\kappa) N_{\kappa } u(q), \]
and thus
\begin{equation}\label{eq:cantpasstoinfty}
	\|S'_\kappa u\|_{L^p(\sigma)} \lesssim C(\kappa) \|N_{\kappa} u\|_{L^p(\sigma)} \leq  C(\kappa) \|N u\|_{L^p(\sigma)} <\infty.
\end{equation} 
We can not let $\kappa$ go to zero in \eqref{eq:cantpasstoinfty} since the upper bound in the right hand side depends on $\kappa$ (in fact $C(\kappa) \to \infty$ as $\kappa \to 0$). However, since $\|S'_\kappa u\|_{L^p(\sigma)}$ is finite, we can apply the previous arguments and prove that $\|S_\kappa u\|_{L^p(\sigma)} \lesssim \|N_{\kappa} u\|_{L^p(\sigma)}$, with a constant independent of $\kappa$. Hence
\[ \|S_\kappa u\|_{L^p(\sigma)} \lesssim \|N_\kappa u\|_{L^p(\sigma)} \leq C\|N u\|_{L^p(\sigma)} \]
with a constant $C$ independent of $\kappa$. Therefore we can safely let $\kappa$ go to zero and conclude that
\[ \|Su\|_{L^p(\sigma)} = \limsup_{\kappa \to 0} \|S_\kappa u\|_{L^p(\sigma)} \leq C \|Nu\|_{L^p(\sigma)}. \]
This finishes the proof of Theorem \ref{thm:SlessthanN}.
\end{proof}

%%%%%%%%%%%%%%%%%%%%%%%%%%%%%%%%%%%%%%%%%
%%%%%%%%%%%%%%%%%%%%%%%%%%%%%%%%%%%%%%%%%
%%%%%%%%%%%%%%%%%%%%%%%%%%%%%%%%%%%%%%%%%
\section{$\omega \in A_{\infty}(\sigma)$ is equivalent to BMO-solvability}\label{sect:showBMOs}

\subsection{From $\omega \in A_{\infty}(\sigma)$ to $L^p$-solvability}
\begin{theorem} %[\cite{CBMS} Theorem 1.4.13(vii) and Lemma 1.4.2]
\label{thm:Nu}
Assume $\omega\in A_{\infty}(\sigma)$, then there exist some $p_0 \in (1, \infty)$ such that the elliptic problem \eqref{ellp} is $L^p-$solvable for all $p \in (p_0, \infty)$, in the sense that there exists a universal constant $C>0$ such that for any $f\in C_0^0(\Gamma)$ and any Borel set $E\subset\Gamma$,
	the solution $u(X) = \int_{E} f d\omega^X$ satisfies the estimate $\|Nu\|_{L^p(\sigma)}\leq C\|f\chi_E \|_{L^p(\sigma)}$.	
\end{theorem}
\begin{remark}
	For a bounded set $E$, it suffices to assume that $f\in C_b(\Gamma)$.
\end{remark}

\begin{proof}
	We first treat the case when $E=\Gamma$.
	Let $q\in\pO$ and denote for any $p>1$
	\begin{equation}\label{eq:tmpfinite}
		\mathcal{M}_p f(q) = \sup_{\Delta \ni q} \left( \fint_{\Delta} |f|^p d\sigma \right)^{\frac{1}{p}}<\infty. 
	\end{equation} 
	We claim
	\begin{equation}\label{eq:claimN}
		|u(X)| \leq C \mathcal{M}_p f(q) \quad \text{ for any } X\in\Gamma(q).
	\end{equation}
	Hence $Nu(q) \leq C \mathcal{M}_p f(q)$, and thus by the $L^p$-boundedness ($p>1$) of Hardy-Littlewood maximal function (see \cite{SHT} for spaces of homogeneous type and \cite{Stein})
	\[ \|Nu\|_{L^p(\sigma)} \leq C \|\mathcal{M}f \|_{L^p(\sigma)} \lesssim \|f\|_{L^p(\sigma)}. \]
	
	In fact, let $X\in\Gamma(q)$ be fixed and $\Delta = \Delta(q, (1+\alpha)\delta(X) )$. For $j\in \mathbb{N}$ let $\Delta_j = 2^j \Delta$, and set $\Delta_{-1} = \emptyset$. We have
	\begin{align}\label{reptmp}
		u(X) = \int f d\omega^X = \sum_{j=0}^\infty \int_{ \Delta_{j} \setminus \Delta_{j-1}} f d \omega^X.
	\end{align}
	For each $j \in\mathbb{N}$ let $A_j$ denote a corkscrew point for $\Delta_j$. Recall Definition \ref{def:Ainfty} of $\omega\in A_{\infty}(\sigma)$ and the discussion after that, in particular \eqref{eq:kernelbydensity} and \eqref{eq:RHkernel}. We have that for each $j$, the Radon-Nikodym derivative
	\[ k^{A_j} (q') = \frac{d\omega^{A_j}}{d\sigma} (q') = \lim_{\Delta' \to q'} \frac{\omega^{A_j}(\Delta')}{\sigma(\Delta')}  \] 
	satisfies a reverse H\"older inequality 
	\begin{equation}\label{eq:RHq}
		\left( \fint_{\Delta_j} \left|k^{A_j} \right|^{r} d\sigma \right)^{\frac{1}{r}} \leq C \fint_{\Delta_j} k^{A_j} d\sigma
	\end{equation}
	for all $r\in (1,r_0)$,	with uniform constants $r_0>1$ and $C>0$. 
	For any $j\geq 2$ and any surface ball $\Delta' \subset \Delta_j \setminus \Delta_{j-1}$, by the H\"older regularity of solutions near the boundary (see Lemma \ref{lm:bdHolder}), we have
	\begin{equation}
		\omega^X(\Delta') \lesssim 2^{-j\beta} \omega^{A_{j-2}}(\Delta') \sim 2^{-j\beta} \omega^{A_j}(\Delta').
	\end{equation}
	Hence for any $q'\in \Delta_j \setminus \Delta_{j-1}$,
	\begin{equation}\label{eq:XtoAj}
		k^X(q') =  \lim_{\Delta' \to q'} \frac{\omega^X(\Delta')}{\sigma(\Delta')} =  \lim_{ \Delta' \to q' \atop{\Delta' \subset \Delta_j\setminus \Delta_{j-1}} } \frac{\omega^X(\Delta')} {\sigma(\Delta')} \lesssim 2^{-j\beta} \lim_{ \Delta' \ni q' \atop{\Delta' \subset \Delta_j\setminus \Delta_{j-1}} } \frac{\omega^{A_j}(\Delta')} {\sigma(\Delta')} = 2^{-j\beta} k^{A_j}(q').
	\end{equation}
	Therefore by \eqref{eq:RHq}, \eqref{eq:XtoAj}, and H\"older inequality for conjugates $1/p+1/r=1$ with $r\in (1,r_0)$, we obtain
	\begin{align}
		|u(X)| \leq \sum_{j=0}^\infty \int_{ \Delta_{j} \setminus \Delta_{j-1}} |f k^X| d \sigma & \lesssim \sum_{j=0}^\infty 2^{-j\beta} \int_{ \Delta_{j}} |f| k^{A_j} d \sigma \nonumber \\
		& \leq \sum_{j=0}^\infty 2^{-j\beta} \sigma(\Delta_j) \left( \fint_{ \Delta_{j}} |f|^p d\sigma \right)^{\frac{1}{p}} \left( \fint_{\Delta_j} | k^{A_j}|^r d \sigma \right)^{\frac{1}{r}} \nonumber \\
		& \lesssim \sum_{j=0}^\infty 2^{-j\beta} \sigma(\Delta_j) \left( \fint_{ \Delta_{j}} |f|^p d\sigma \right)^{\frac{1}{p}} \left( \fint_{\Delta_j} k^{A_j} d \sigma \right) \nonumber \\
		& \leq \sum_{j=0}^\infty 2^{-j\beta} \mathcal{M}_p f(q) \omega^{A_j}(\Delta_j) \nonumber \\
		& \lesssim \mathcal{M}_p f(q),\label{eq:kerneltmp}
	\end{align}
	thus we finish proving the claim \eqref{eq:claimN} for any $p\in (p_0, \infty)$, where $p_0$ is the conjugate of $r_0$. Note that we never use the continuity or compact support of $f$, and replacing $f$ by $f\chi_E$ we can repeat the same argument with no change. The assumption that $E$ is bounded or $f$ has compact support guarantees we still have a priori finite integrability in \eqref{eq:tmpfinite}.
\end{proof}

%%%%%%%%%%%%%%%%%%%%%%%%%%%%%%%%%%%%%%%%%
%previously subsection of $S<N$
%%%%%%%%%%%%%%%%%%%%%%%%%%%%%%%%%%%%%%%%%

\subsection{Proof of the BMO-solvability}

\begin{theorem}\label{thm:4.1}
	Assume that $\omega \in A_{\infty}(\sigma)$. For any $f\in C_0^0(\Gamma)$, let $u = Uf\in W_r(\Omega)$ be a solution to $Lu=0$ given by Lemmas \ref{lm:dfsol} and \ref{lm:dfhm}. Then $|\nabla u|^2 \delta(X)dm(X)$ is a Carleson measure, and moreover 
\begin{equation}
	\sup_{\Delta \subset \pO} \frac{1}{\sigma(\Delta )} \iint_{T(\Delta )}|\nabla u|^2 \delta(X) dm(X) \leq C\|f\|_{BMO(\sigma) }^2. \label{eq:CarlesonBMO}
\end{equation}
\end{theorem}

\begin{proof}
	Fix an arbitrary surface ball $\Delta=\Delta(q_0,r)$. Let $\alpha>0$. Denote the constant $c=\max\{\alpha+2, 12\}$ and let $\dt=c\Delta = \Delta(q_0, cr)$ be a concentric dilation. 
We denote the average $ f_{\dt} = \fint_{\dt} f d\sigma$. 
Let 
\[ f_1= (f-f_{\dt})\chi_{\dt}, \quad f_2 = (f-f_{\dt})\chi_{\pO\setminus\dt}, \quad f_3 = f_{\dt}, \]
and for any $X\in\Omega$ let
\[ u_1(X) = \int_{\pO} f_1 d\omega^X = \int_{\dt} \left(f-f_{\dt} \right) d\omega^X,  \]
\[ u_2(X) = \int_{\pO} f_2 d\omega^X =  \int_{\pO\setminus \dt} \left( f-f_{\dt} \right) d\omega^X = \int_{\pO \setminus \dt} f d\omega^X  - f_{\dt} \omega^X(\pO\setminus \dt), \]
\[  u_3 \equiv f_{\dt}. \]
By Lemmas \ref{lm:dfsol}, \ref{lm:dfhm}, \ref{lm:hmassol} and \ref{lm:fchiE}, they are solutions to $L$, and $u_1, u_2$ can be continuously extended to $\pO\setminus \dt$ and $\dt$, respectively. Moreover
\[ \left(u_1 + u_2 + u_3 \right) (X) = \int_{\pO} f d\omega^X = Uf(X) = u(X). \]
Clearly the Carleson measure of the constant function $u_3$ is trivial. 

Apply Theorem \ref{thm:Nu} to $f_1$ and $u_1$ we get $\|Nu_1\|_{L^p(\sigma)} \leq C\|f_1\|_{L^p(\sigma)}<\infty$. Combined with Theorem \ref{thm:SlessthanN}, we get
\begin{equation}
	\|S u_1\|_{L^p(\sigma)} \lesssim \|Nu_1\|_{L^p(\sigma)} \lesssim \|f_1\|_{L^p(\sigma)} = \left(\int_{\dt} |f-f_{\dt}|^p d\sigma\right)^{1/p}\label{eq:SquareLp}
\end{equation} 
for any $p\in (p_0,\infty)$.
By \eqref{eq:CarlesonSquarepre} and \eqref{eq:CarlesonSquareU}
\[ \iint_{T(\Delta)} \Carl{u_1} dm(X) \leq C\int_{(\alpha+2)\Delta} |S_{(\alpha+1) r} u_1|^2 d\sigma\]
Recall that $\dt = c\Delta \supset (\alpha+2)\Delta$, thus
\begin{align}
	\iint_{T(\Delta)} \Carl{u_1} dm(X) & \leq C \int_{\dt} |S_{(\alpha+1) r} u_1|^2 d\sigma \nonumber \\
	& \leq C \sigma(\dt)^{1-\frac 2 p} \left(\int_{\dt} |S u_1|^p d\sigma\right)^{\frac{2}{p}} \nonumber \\
	& \leq C \sigma(\dt)^{1-\frac 2 p} \|S u_1\|_{L^p(\sigma)}^2, \label{eq:SquareHolder}
\end{align} 
for any $p> \max\{2,p_0\}$.
Combining \eqref{eq:SquareHolder} and \eqref{eq:SquareLp} we get
\begin{align}
	\iint_{T(\Delta)} \Carl{u_1} dm(X) \leq C\sigma(\Delta) \|f\|_{BMO(\sigma)}^2<\infty. \label{uone}
\end{align}

Turning to the estimate for $u_2$, let $\{I_k\} \subset \WW$ be a collection of dyadic Whitney boxes that intersect of $T(\Delta)$ (recall the properties of Whitney decomposition $\WW$ in \eqref{Wbox}). On each Whitney box $I_k$, we have by the interior Cacciopoli inequality \eqref{eq:intCcpl}
\begin{align*}
	 \iint_{I_k} \Carl{u_2} dm(X) & \lesssim \ell(I_k)\iint_{I_k} |\nabla u_2|^2 dm(X) \\
	 & \lesssim \ell(I_k)\cdot \frac{1}{\ell(I_k)^2} \iint_{I_k^*} |u_2(X)|^2 dm(X) \\
	 & \lesssim \iint_{I_k^*} \frac{|u_2(X)|^2}{\delta(X)} dm(X),
\end{align*}
Recall $I_k^* = (1+\theta) I_k$ is the dilation of $I_k$ satisfying \eqref{eq:Wboxdl}.
Then summing up we get
\begin{align}\label{eq:utwo}
	\iint_{T(\Delta)} \Carl{u_2} dm(X) & \lesssim \sum_k \iint_{ I_k^*} \frac{|u_2(X)|^2}{\delta(X)} dm(X) \nonumber \\
	& \lesssim \iint_{T\left(\frac 3 2 \Delta \right)} \frac{|u_2(X)|^2}{\delta(X)} dm(X).
\end{align}
In the last line we use the finite overlap of $\{I_k^*\}$, and the fact that by taking $\theta$ sufficiently small, we can ensure that $I^*_k \subset T(\frac{3}{2} \Delta)$ for all $I_k$ intersects $T(\Delta)$.
Recall that $\frac{3}{2}\Delta= \Delta(q_0, \frac{3}{2}r)$ and $T(\frac{3}{2}\Delta)$ denotes $B(q_0, \frac{3}{2} r)\cap\Omega$.

Let $f_2^{\pm}$ denote the positive and negative part of $f_2$, and let $u_2^{\pm} = \int_{\pO \setminus \dt} f_2^{\pm} d\omega^X \geq 0$. There is a technical issue that $f_2^{\pm} \notin C_0^0(\Gamma)$, however by splitting $u_2^{\pm}$ as follows,
\[ u_2^{+}(X) = \int_{\{f\geq f_{\dt}\} \setminus \dt } f d\omega^X - f_{\dt} \omega^X \left( \{f\geq f_{\dt}\} \setminus \dt \right), \]
\[ u_2^{-}(X) = -\int_{\{f< f_{\dt}\} \setminus \dt } f d\omega^X + f_{\dt} \omega^X \left( \{f < f_{\dt}\} \setminus \dt \right), \]
we can confirm by combining Lemmas \ref{lm:hmassol} and \ref{lm:fchiE} that $u_2^{\pm} \in W_r(\Omega)$ are indeed legitimate solutions of $L$, and they can be continuously extended to $\dt$ by zero.
By the linearity of integration, we have $u_2 = \int_{\pO} f_2 d\omega^X =  u_2^+ - u_2^-$.  
Let $v(X) : = u_2^+(X) + u_2^-(X)$, again by linearity we have
\begin{equation}\label{def:v}
	v(X) = \int_{\pO} |f_2| d\omega^X = \int_{\pO\setminus \dt} |f-f_{\dt}| d\omega^X.
\end{equation} 
Thus $|u_2(X)| \leq v(X)$ for all $X\in\Omega$. Moreover by the properties of $u_2^{\pm}$, we know that $v\in W_r(\Omega)$ is a solution of $L$, $Tv = 0 $ on $\dt$ and that $v\in W_r(B(q_0,cr))$. (Recall that $\dt =c\Delta =  B(q_0,cr) \cap \pO $.)
We claim that
\begin{equation}\label{eq:vBMObd}
	v(X) \leq C \|f\|_{BMO(\sigma)} \quad \text{ for all } X\in T(6\Delta).
\end{equation}	
	By the definition \eqref{def:v}, the function $v$ vanishes on $\dt$. Note that $\dt\supset 12\Delta$ by the choice of $\dt$, $v\in W_r(B(q_0,12r))$ is a non-negative solution in $T(12\Delta)$ and $Tv \equiv 0$ on $12\Delta$. Let $A$ be a corkscrew point for $T(12\Delta)$, by the boundary Harnack inequality \eqref{eq:bdHnk}
	\[ v(X) \leq C v(A),\quad\text{for all~} X\in T\left( 6\Delta\right). \]
	For any $j\in\mathbb{N}$, let $A_j$ be a corkscrew point for the surface ball $2^j \dt$.
	Similar to \eqref{eq:kerneltmp}, we get
	\begin{align}
		v(A) & \lesssim \sum_{j=1}^{\infty} 2^{-j\beta} \int_{2^j\dt\setminus 2^{j-1}\dt} |f-f_{\dt}| k^{A_j} d\sigma \nonumber \\
		& \leq \sum_{j=1}^{\infty} 2^{-j\beta} \left(\fint_{2^j\dt} |f-f_{\dt}|^p d\sigma \right)^{\frac{1}{p}} \left( \fint_{2^j\dt} \left| k^{A_j} \right|^r d\sigma \right)^{\frac{1}{r}} \sigma(2^j\dt) \nonumber \\
		& \lesssim \sum_{j=1}^{\infty} 2^{-j\beta} \|f\|_{BMO(\sigma)} \omega^{A_j}(2^j\dt) \nonumber \\
		& \lesssim \|f\|_{BMO(\sigma)}.
	\end{align}
	Here $p$ is a conjugate to $r$. We conclude the proof of \eqref{eq:vBMObd}. 
	
	Next, we show a finer estimate based off \eqref{eq:vBMObd}, which is
	\begin{equation}\label{eq:vHol}
		v(X) \leq C\left(\dfrac{\delta(X)}{r}\right)^\beta \|f\|_{BMO(\sigma)} \quad \text{ for all } X\in T\left(\frac{3}{2}\Delta \right),
	\end{equation}  
	where $\beta\in (0,1]$ is the exponent from Lemma \ref{lm:bdHolder}.
	To this end, for any $X\in T(\frac{3}{2}\Delta)$, let $q_X$ be a boundary point such that $|X-q_X|=\delta(X)$. Note that
	\[|X-q_X|=\delta(X) \leq |X-q_0| < \frac{3}{2}r , \]
	i.e. $X\in B(q_X,3r/2)\cap\Omega$. Note also
	\[ |q_X-q_0| \leq |q_X-X|+|X-q_0| <\frac{3r}{2} + \frac{3r}{2} =3r, \]
	so $\overline{B\left(q_X,3r\right)} \subset B(q_0, 6r)$.
	Since $\dt \supset 6\Delta \supset \Delta(q_X, 3r)$, $v\in W_r(B(q_X,3r))$ is a non-negative solution in $B(q_X, 3r)\cap\Omega$ and $Tv \equiv 0$ on $\Delta(q_X, 3r)$. By the boundary H\"older regularity \eqref{eq:bdHolder} and the first part of this lemma \eqref{eq:vBMObd}, we conclude
	\begin{align*}
		v(X) & \lesssim \left( \frac{|X-q_X|}{3r}\right)^{\beta} \left( \frac{1}{m\left(B(q_X,3r)\right)} \iint_{B(q_X,3r)\cap\Omega} |v|^2 dm \right)^{\frac{1}{2}} \\
		& \lesssim \left( \frac{\delta(X)}{r}\right)^{\beta} \sup_{T(6\Delta)} v   \lesssim \left( \frac{\delta(X)}{r}\right)^{\beta} \|f\|_{BMO(\sigma)}.
	\end{align*}

Combining \eqref{eq:vHol} and \eqref{eq:utwo}, we get
\begin{equation}
	\iint_{T(\Delta)} \Carl{u_2} dm(X) \lesssim \dfrac{\|f\|_{BMO(\sigma)}^2}{r^{2\beta}}\left( \iint_{T \left(\frac{3}{2}\Delta \right)} \delta(X)^{2\beta - 1} dm(X) \right) . \label{eq:utwoBMO}
\end{equation} 
Since $2\beta-1>-1$, we can use Lemma \ref{lm:delta} with exponent $\alpha=2\beta-1$ to get
\begin{equation}
	\iint_{T(\Delta)} \Carl{u_2} dm(X) \lesssim r^{d} \|f\|_{BMO(\sigma)}^2 \lesssim \sigma(\Delta)\|f\|_{BMO(\sigma)}^2. \label{utwo}
\end{equation} 
Combining \eqref{uone} and \eqref{utwo} finishes the proof.
\end{proof}

%%%%%%%%%%%%%%%%%%%%%%%%%%%%%%%%%%%%%%%%%%%%%%%%%%%%%%%%%%%%%%%%%%%%%%%%%%%%%%%%%%%%%%%%%%%%
\subsection{From BMO-solvability to $\omega \in A_{\infty}(\sigma)$}\label{sect:showAinfty}

In this subsection, we prove the other half of Theorem \ref{thm:main}:
\begin{theorem}\label{thm:5.1}
	Assume that for any $f\in C_0^0(\Gamma)$, the solution $u = Uf\in W_r(\Omega)$ given by Lemmas \ref{lm:dfsol} and \ref{lm:dfhm} satisfies the property that $|\nabla u|^2 \delta(X)dm(X)$ is a Carleson measure with 
\begin{equation}
	\sup_{\Delta \subset \pO} \frac{1}{\sigma(\Delta )} \iint_{T(\Delta )}|\nabla u|^2 \delta(X) dm(X) \leq C\|f\|_{BMO(\sigma) }^2. \label{eq:CarlesonBMO}
\end{equation}
Then $\omega \in A_\infty(\sigma)$, with the implicit constant depending on $d, n, C_0, C_1$ and the above constant $C$.
\end{theorem}

Let us start with proving the following Lemma.
\begin{lemma}\label{lm:BMOprelim}
	Suppose the Dirichlet problem \eqref{ellp} is BMO-solvable. Then any non-negative function $f\in C_0^0(\Gamma)$ whose supporte is contained in a surface ball $\Delta$ satisfies
	\begin{equation}\label{eq:BMO}
	 \int_{\Delta}f d\omega^A \leq C \|f\|_{BMO(\sigma)}.
	\end{equation} 
	Here $A$ is a corkscrew point for $\Delta$.
\end{lemma}

\begin{proof}

Since $f\in C_0^0(\Gamma)$ is a non-negative function, by Lemma \ref{lm:dfsol} $u= Uf\in W_r(\Omega)$ is a non-negative solution of $L$. Suppose $\Delta$ has radius $r$. Consider another surface ball $\Delta'=B(q',r) \cap \Gamma$ of the same radius $r$ and which is $2r-$distance away from $\Delta$. Thus in particular, $Tu = 0 $ on $3\Delta'$ and that $u\in W_r(B(q',3r))$, by Lemma \ref{lm:dfsol} (i) and (iv).
Applying the BMO-solvability assumption to $u=Uf$ and the surface ball $\Delta'$, we have
\begin{equation}
	\iint_{T(\Delta')} \Carl{u} dm(X)\leq C\sigma(\Delta')\|f\|_{BMO(\sigma)}^2 \label{eq:Carlcond}
\end{equation}

We have shown in \eqref{eq:CarlesonSquareL} that
\begin{equation} \label{eq:CarlSquaregeq}
	\iint_{T(\Delta')}\Carl{u} dm(X) \gtrsim\int_{\Delta'/2} |S_{r/2} u|^2 d\sigma,
\end{equation}
where $S_{r/2}u$ is the truncated square function of aperture $\overline{\alpha}>\alpha$, whose value is determined in Lemma \ref{lm:Poincare} and only depends on $n, d, C_0, C_1$ and $\alpha$.
In order to get a lower bound of the square function $S_{r/2}u$, we decompose the non-tangential cone $\Gamma_{r/2}(q)$ into stripes as in \eqref{def:stripe} and use the Poincar\'e-type inequality proved in Lemma \ref{lm:Poincare} for surface ball $\Delta'$.
Let $m_1, m_2$ be integers determined in Lemma \ref{lm:Poincare}. We obtain
\begin{align*}
	|S_{r/2} u|^2 (q) & = \iint_{\Gamma^{\overline \alpha}_{r/2} (q)} |\nabla u|^2 \delta(X)^{1-d} dm(X) \\
	& \geq \frac{1}{m_1+m_2} \sum_{j=m_1+1}^{\infty} \iint_{\Gamma^{\overline\alpha}_{j-m_1 \rightarrow j + m_2}(q)} |\nabla u|^2 \delta(X)^{1-d} dm(X) \\
	& \gtrsim \sum_{j=m_1+1}^{\infty}(2^{-j}r)^{1-d} \iint_{\Gamma^{\overline \alpha}_{j-m_1\rightarrow j+m_2 } (q)} |\nabla u|^2 dm(X) \\
	& \gtrsim \sum_{j=m_1+1}^{\infty} (2^{-j}r)^{1-d}\cdot (2^{-j}r)^{-2} \iint_{\Gamma_j^{\alpha}(q)} u^2 dm(X) \\
	& \gtrsim \sum_{j=m_1+1 }^{\infty} u^2(A_j),
\end{align*}
where $A_j\in \Gamma_j(q)$ is a corkscrew point at the scale $2^{-j}r$.
In the last inequality, we use the interior corkscrew condition, as each stripe of cone $\Gamma_j(q)$ contains a ball of radius comparable to $2^{-j-1}r$ (as long as $\alpha$ is chosen to be big, say $\alpha>2M$, where $M$ is the corkscrew constant). Moreover,
\begin{equation}\label{eq:lbgs}
	\sum_{j=m_1+1 }^{\infty} u^2(A_j) \geq u^2(A_{m_1}) \gtrsim u^2(A_1).
\end{equation}
Recall for any $q\in \Delta'$, the point $A_1= A_1(q)$ is a corkscrew point of $B(q,2^{-1}r)$. 
Let $A'$ be the corkscrew point for $T(\Delta'/2)$, by Lemma \ref{lm:Hcc} and Harnack inequality, $u(A') \approx u(A_1)$. 
Therefore
\[ |S_{r/2}u|^2 (q) \gtrsim u^2(A_1) \gtrsim u^2(A'), \quad \text{for any~} q\in \Delta' . \]
Combining this with \eqref{eq:Carlcond} and \eqref{eq:CarlSquaregeq}, we get
\[ \sigma(\Delta')\|f\|_{BMO(\sigma)}^2 \gtrsim \int_{\Delta'/2} |S_{r/2}u|^2 d\sigma \gtrsim \sigma(\Delta'/2)u^2(A') \gtrsim \sigma(\Delta')u^2(A') , \]
and thus
\begin{equation}\label{eq:pointwiseBMO}
	u(A') \lesssim \|f\|_{BMO(\sigma)}.
\end{equation} 
Let $A$ be a corkscrew point for $\Delta$. Since $\Delta$ and $\Delta'$ have the same radius $r$ and they are $2r-$distance apart, we have $u(A) \sim u(A')$. By assumption $f$ is supported on $\Delta$, hence
\begin{equation}\label{eq:repatA}
	u(A) = \int_{\Delta} f d\omega^A.
\end{equation}
The lemma follows by combining \eqref{eq:pointwiseBMO} and \eqref{eq:repatA}.
\end{proof}

%%%%%%%%%%%%%%%%%%%%%%%%%%%%%%%%%%%%%%%%%%%%%%%%%%%
With that at hand, we pass to the
\vskip 0.08in
\begin{proof}[Proof of Theorem \ref{thm:5.1}]

By the change of pole formula in Lemma \ref{lm:cop} and Harnack inequality, to prove $\omega\in A_{\infty}(\sigma)$ and in particular \eqref{eq:defAinfty}, it suffices to show:
For any $\epsilon>0$ fixed, we can find $\eta= \eta(\epsilon) $, such that for any Borel set $E \subset \Delta$,
\begin{equation}\label{eq:proveAinfty}
	\frac{\sigma(E)}{\sigma(\Delta)} <\eta \quad\text{implies} \quad \frac{\omega^A (E)}{\omega^A (\Delta)} < \epsilon.
\end{equation} 
Here $\Delta$ is a surface ball and $A$ is a corkscrew point for $\Delta$.
In fact, since $\sigma$ and $\omega$ are regular Borel measures, we may assume $E$ is an open subset of $\Delta$.

Recall from Lemma \ref{lm:nondeg} that 
\[ \omega^A(\Delta) \geq C^{-1} \] for some $C>1$. Thus to show $\omega^A(E)/\omega^A(\Delta)<\epsilon$ it suffices to show $\omega^A(E) < C^{-1}\epsilon$.
Let $\delta>0$ be a small constant to be determined later, we define a function
\begin{equation}
	f(x) = \max\left\{ 0, 1+\delta \log M_{\sigma}\chi_{E}(x)\right\} \label{def:f}
\end{equation} 
where $M_{\sigma}$ is the Hardy-Littlewood maximal function with respect to $\sigma$.
Similar to Section 5.3 of \cite{Zh}, $f$ satisfies
\begin{itemize}
	\item $0\leq f\leq 1$, and $f\equiv 1$ on the open set $E$;
	\item $\|f\|_{BMO(\sigma)} \leq A\delta$, where $A$ is a constant independent of $E$;
	\item If 
		\begin{equation}\label{conditioneta}
			\frac{\sigma(E)}{\sigma(\Delta)} < \eta(\delta) \sim e^{-1/\delta},
		\end{equation}  
		then $f$ is supported in $2\Delta$.
\end{itemize}

Next we use a mollification argument to approximate $f$ by continuous functions. Let $\varphi$ be a radially symmetric smooth function on $\mathbb{R}^n$ such that $\varphi = 1$ on $B_{1/2}$, $\supp \varphi\subset B_1 $ and $0\leq \varphi \leq 1$. Let 
\begin{equation}
	\varphi_{\epsilon}(z) = \frac{1}{\epsilon^{d}} \varphi\left(\frac{z}{\epsilon}\right), \quad f_{\epsilon}(x) = \frac{ \int_{y\in \pO} f(y) \varphi_{\epsilon} (x-y) d\sigma(y)}{\int_{y\in\pO} \ve(x-y)d\sigma(y)} \text{ for } x\in\pO . \label{deffe}
\end{equation} 
Then these $f_\epsilon$'s satisfy the following properties:
\begin{itemize}
	\item each $f_{\epsilon}$ is continuous, and is supported in $3\Delta$;
	\item there is a constant $C$ (independent of $\epsilon$) such that $\|f_{\epsilon}\|_{BMO(\sigma)} \leq C \|f\|_{BMO(\sigma)}$;
	\item $f(x) \leq \liminf_{\epsilon\rightarrow 0} f_{\epsilon}(x)$ for all $x$ in their support $3\Delta$. 
\end{itemize} 
The proof of the above properties is a slight modification of Appendix A of \cite{Zh}: here the mollifier $\{\varphi_\epsilon\}$ is an approximation of identity of dimension $d$, instead of dimension $n-1$. The proof uses standard mollification arguments and the Ahlfors regularity of $\pO$. Moreover, the proof of the last property also uses  the precise definition of $f$ in \eqref{def:f}.

Let $A'$ be a corkscrew point with respect to $3\Delta$. The last property and Fatou's lemma imply
\begin{align}
	\int_{3\Delta} f(x) d\omega^{A'}(x) \leq \int_{3\Delta} \liminf_{\epsilon\rightarrow 0} f_{\epsilon}(x) d\omega^{A'}(x) \leq \liminf_{\epsilon\rightarrow 0} \int_{3\Delta} f_{\epsilon}(x) d\omega^{A'}(x) . \label{eq:Fatou}
\end{align} 
Since each $f_{\epsilon}$ is non-negative, continuous and supported on $3\Delta$, we apply Lemma \ref{lm:BMOprelim} and get
\begin{equation}
	 \int_{3\Delta} f_{\epsilon}(x) d\omega^{A'}(x) \leq C \|f_{\epsilon}\|_{BMO(\sigma)}\leq C' \|f\|_{BMO(\sigma)}. \label{eq:feBMO}
\end{equation} 
Combining \eqref{eq:Fatou} and \eqref{eq:feBMO}, we get
\[  \int_{3\Delta} f(x) d\omega^{A'}(x)\leq C' \|f\|_{BMO(\sigma)} \leq C''\delta. \]
On the other hand, since $f\geq\chi_E$
\[  \int_{3\Delta} f(x) d\omega^{A'}(x) \geq \omega^{A'}(E) \gtrsim \omega^{A}(E). \]
The last inequality follows from the Harnack inequality and the fact that $A, A'$ are corkscrew points to surface balls $\Delta, 3\Delta$ respectively.
 Therefore $\omega^A(E) \leq C\delta$ as long as the condition \eqref{conditioneta}, i.e. $\sigma(E)/\sigma(\Delta)<\eta$ is satisfied. In other words, $\omega \in A_{\infty}(\sigma)$.
\end{proof}

%%%%%%%%%%%%%%%%%%%%%%%%%%%%%%%%%%%%%%%%%%%%%%%%%%%%%%%%%%%%%%%%%%%%%%%%%%%%%%%%%%%%%%%%%%%%
%%%%%%%%%%%%%%%%%%%%%%%%%%%%%%%%%%%%%%%%%%%%%%%%%%%%%%%%%%%%%%%%%%%%%%%%%%%%%%%%%%%%%%%%%%%%

%\appendix
%\appendixpage
%\addappheadtotoc
%
%\setcounter{equation}{0}
%\renewcommand{\theequation}{\arabic{equation}} 

%\Addresses


\newpage

\begin{thebibliography}{9}

\bibitem[ABaHM]{ABaHM} M. Akman, M. Badger, S. Hofmann, J.M. Martell, \textit{Rectifiability and elliptic measures on $1$-sided NTA domains with Ahlfors-David regular boundaries}. T. Am. Math. Soc. (8) \textbf{369} (2017) 5711-5745.

\bibitem[ABoHM]{ABoHM} M. Akman, S. Bortz, S. Hofmann and J.M. Martell, \textit{Rectifiability, interior approximation and harmonic measure}. in preprint, arXiv:1601.08251 (2016).

\bibitem[Az]{Az} J. Azzam, \textit{Semi-uniform domains and a characterization of the $A_\infty$ property for harmonic measure}. in preprint arXiv:1711.03088 (2017).

\bibitem[AGMT]{AGMT} J. Azzam, J. Garnett, M. Mourgoglou, X. Tolsa, \textit{Uniform rectifiability, elliptic measure, square functions, and $\epsilon$-approximability via an ACF monotonicity formula}. in preprint arXiv:1612.02650 (2017).

\bibitem[AHM3TV]{AHMMMTV} J. Azzam, S. Hofmann, J.M. Martell, S. Mayboroda, M. Mourgoglou, X. Tolsa, A. Volberg, \textit{Rectifiability of harmonic measure}. Geom. Funct. Anal.  Volume 26, Issue 3 (2016) 703-728.

\bibitem[AHMNT]{AHMNT}
J. Azzam, S. Hofmann, J. M. Martell, K. Nystr\"om, T. Toro, \textit{A new characterization of chord-arc domains}. J. of European Math. Soc. Volume 19, Issue 4 (2017) 967-981.

\bibitem[AM]{AM} J. Azzam, M. Mourgoglou, \textit{Tangent measures of elliptic harmonic measure and applications}. in preprint arXiv:1708.03571 (2017).

\bibitem[Ba]{Ba} M. Badger, \textit{Null sets of harmonic measure on NTA domains: Lipschitz approximation revisited}. Math. Z. (2012) 270:241-262.

\bibitem[BJ]{BJ} C.J. Bishop, P.W. Jones, \textit{Harmonic measure and arclength}. Ann. Math. (3) \textbf{132} (1990), 511-547.

\bibitem[BL]{BL} B. Bennewitz, J.L. Lewis, \textit{On weak reverse H\"older inequalities for nondoubling harmonic measures}. Complex Var. Theory Appl. \textbf{49} (2004), no. 7-9, 571-582.

\bibitem[Ch]{Ch} M. Christ, \textit{A $T(b)$ theorem with remarks on analytic capacity and the Cauchy integral}. Colloq. Math., \textbf{LX/LXI} (1990), 601-628.

\bibitem[CMS]{CMS} R.R Coifman, Y. Meyer, E.M. Stein, \textit{Some new function spaces and their applications to harmonic analysis}. J. Funct. Anal. \textbf{62} (1985), 304-335.

%\bibitem[CR]{logBMO} R.R. Coifman and R. Rochberg, \textit{Another characterization of BMO}. Proc. Amer. Math. Soc. \textbf{79} (1980), 249-254.

\bibitem[CW]{SHT} R.R. Coifman and G. Weiss, \textit{Extensions of Hardy spaces and their use in analysis}. Bull. Amer. Math. Soc. Volume 83, Number 4 (1977), 569-645.

\bibitem[Da]{D} B.E.J. Dahlberg, \textit{On estimates for harmonic measure}. Arch. for Rational Mech. and Anal. \textbf{65} (1977), 272-288.

\bibitem[DFM1]{elliptic} G. David, J. Feneuil, and S. Mayboroda, \textit{Elliptic theory for sets with higher co-dimensional boundaries}. arXiv:1702.05503.

\bibitem[DFM2]{Ainfty} G. David, J. Feneuil, and S. Mayboroda, \textit{Dahlberg's theorem in higher co-dimension}. arXiv:1704.00667.

\bibitem[DJ]{DJ} G. David, D. Jerison, \textit{Lipschitz approximation to hypersurfaces, harmonic measure, and singular integrals}. Indiana Univ. Math. J. \textbf{39} (3), 831–845 (1990).

\bibitem[DJK]{DJK} B.E.J. Dahlberg, D.S. Jerison, and C.E. Kenig, \textit{Area integral estimates for elliptic differential operators with nonsmooth coefficients}. Arkiv. Mat., \textbf{22} (1984), 97-108.

\bibitem[DKP]{DKP} M. Dindos, C. Kenig, J. Pipher, \textit{BMO solvability and the $A_{\infty}$ condition for elliptic operators}.  J. Geom. Anal. \textbf{21} (2011), no. 1, 78-95.

\bibitem[DS1]{DS1} G. David and S. Semmes, \textit{Singular integrals and rectifiable sets in $\RR^n$: Beyond Lipschitz graphs}. Ast\'erisque, \textbf{193} (1991).

\bibitem[DS2]{DS2} G. David, S. Semmes, \textit{Analysis of and on uniformly rectifiable sets}. Mathematical Surveys and Monographs, 38 Amer. Math. Soc., Providence, RI, 1993.

\bibitem[FJK]{FJK} E.B. Fabes, D.S. Jerison, C.E. Kenig, \textit{Necessary and sufficient conditions for absolute continuity of elliptic-harmonic measure}. Ann. Math. \textbf{119} (1984), 121-141.

\bibitem[GMT]{GMT} J. Garnett, M. Mourgoglou, X. Tolsa, \textit{Uniform rectifiability in terms of Carleson measure estimates and $\epsilon$-approximability of bounded harmonic functions}. in preprint arXiv:1611.00264 (2016).

\bibitem[GR]{weight} J Garc\'ia-Cuerva, J.L. Rubio de Francia, \textit{Weighted norm inequalities and related topics}. North-Holland Mathematics Studies, \textbf{116} (1985).

\bibitem[HLMN]{HLMN} S. Hofmann, P. Le, J.M. Martell, K. Nystr\"om, \textit{The weak-$A_\infty$ property of harmonic and $p$-harmonic measures implies uniform rectifiability}. Anal. \& PDE (3) \textbf{10} (2017), 513-558.

\bibitem[HM1]{HM} S. Hofmann, J.M. Martell, \textit{Uniform rectifiability and harmonic measure I: Uniform rectifiability implies Poisson kernels in $L^p$}. Ann. Sci. Ecole Norm. Sup. \textbf{47} (2014), no. 3, 577-654.

\bibitem[HM2]{HM2} S. Hofmann, J.M. Martell, \textit{A sufficient geometric criterion for quantitative absolute continuity of harmonic measure}. in preprint arXiv:1712.03696 (2017).

\bibitem[HMM]{HMM} S. Hofmann, J.M. Martell, S. Mayboroda, \textit{Uniform rectifiability, Carleson measure estimates, and approximation of harmonic functions}. to appear in Duke Math. J. arXiv:1408.1447.

\bibitem[HMMTZ]{HMMTZ} S. Hofmann, J.M. Martell, M. Mayboroda, T. Toro, Z. Zhao, \textit{Uniform rectifiability and elliptic operators with small Carleson norm}. in preprint arXiv:1710.06157 (2017).

\bibitem[HMT]{HMT} S. Hofmann, J.M. Martell, T. Toro, \textit{$A_\infty$ implies NTA for a class of variable coefficient elliptic operators}. J. Differ. Equations. Volume 263, Issue 10 (2017), 6147-6188.

\bibitem[HMU]{HMU} S. Hofmann, J.M. Martell, I. Uriarte-Tuero, \textit{Uniform rectifiability and harmonic measure II: Poisson kernels in $L^p$ imply uniform rectifiability}. to appear in Duke Math. J. arXiv:1202.3860v2.

\bibitem[JK]{NTA} D. Jerison, C. Kenig, \textit{Boundary behavior of harmonic functions in non-tangentially accessible domains}. Adv. in Math. \textbf{46} (1982) 80147.

\bibitem[JN]{JN} F. John, L. Nirenberg, \textit{On functions of bounded mean oscillation}. Comm. Pure Appl. Math. \textbf{14} (1961), 415-426.

%\bibitem[KKiPT]{KKiPT} C. Kenig, B. Kirchheim, J. Pipher, T. Toro, \textit{Square functions and the $A_{\infty}$ property of elliptic measures}. J Geom Anal (2016) \textbf{26}: 2383. 

\bibitem[KP]{KP} C. Kenig, and J. Pipher, \textit{The Dirichlet problem for elliptic equations with drift terms}, Publ. Mat. {\bf 45}, (2001) 199-217.

\bibitem[La]{La} M. Lavrentiev, \textit{Boundary problems in the theory of univalent functions}. (in Russian) Math. Sb. (N.S.) 1, 815-845 (1036) [Transl.: Amer. Math. Soc. Transl. (2) 32, 1-35 (1936)]

\bibitem[MPT]{HACAD} E. Milakis, J. Pipher, T. Toro, \textit{Harmonic analysis on chord arc domain}. J. Geometric Analysis, (2013), 23, 2091-2157.

\bibitem[Mo]{Mo} M. Mourgoglou, \textit{Uniform domains with rectifiable boundaries and harmonic measure}. in preprint arXiv:1505.06167 (2015).

\bibitem[MT]{MT} M. Mourgoglou, X. Tolsa, \textit{Harmonic measure and Riesz transform in uniform and general domains}. in preprint arXiv:1509.08386, to appear in J. Reine Ang. Math.

\bibitem[RR]{RR} F. Riesz, and M. Riesz, \textit{\"Uber die randwerte einer analtischen funktion}. Compte Rendues du Quatri\`eme Congr\`es des Math\'ematiciens Scandinaves, Stockholm 1916, Almqvists and Wilksels, Upsala, 1920.

\bibitem[Se]{S} S. Semmes, \textit{Analysis vs. geometry on a class of rectifiable hypersurfaces in $\mathbb{R}^n$}. Indiana Univ. Math. J. \textbf{39} (4), 1005-1035 (1990).

\bibitem[St1]{Stein} E.M. Stein, \textit{Harmonic analysis: real-variable methods, orthogonality, and oscillatory integrals}. 1st edition, Princeton University Press,  Princeton N.J. 1993.

\bibitem[St2]{SteinSI} E.M. Stein, \textit{Singular integrals and differentiability properties of functions}. Princeton University Press, Princeton N.J. 1970.

\bibitem[TZ]{TZ} T. Toro, Z. Zhao, \textit{Boundary rectifiability and elliptic operators with $W^{1,1}$ coefficients}. arXiv:1708.00539, to appear in Adv. Calc. Var.

\bibitem[Zh]{Zh} Z. Zhao, \textit{BMO solvability and the $A_{\infty}$ condition of the elliptic measure in uniform domains}. arXiv:1602.00717, to appear in J. Geom. Anal.

\end{thebibliography}
\end{document}